\documentclass[preprint]{elsarticle}
\usepackage{lmodern}
\usepackage[T1]{fontenc}
\usepackage[%
  shortlabels,
  inline,
]{enumitem}
\usepackage[%
  tocdep=2, 
]{secnum} 
\usepackage[leqno]{mathtools}
\usepackage[%
  scr=stixtwofancy, 
  cal=euler, 
  bb=stixtwo, 
]{mathalpha}
\usepackage{amssymb}
\usepackage{amsthm}
\usepackage[Symbol]{upgreek}
\usepackage{mleftright}
\usepackage{interval} 
\usepackage{xfrac} 
\usepackage{fixdif}
\usepackage{leftindex}
\usepackage{accents}
\usepackage{aligned-overset}%
\usepackage{mathcommand}
\usepackage{tikz}
\usepackage[%
  pagebackref, 
  breaklinks, 
]{hyperref}
\usepackage[title]{appendix}%
\usepackage[%
  hypcap=true,
  font=small,%
]{caption}
\usepackage[%
  capitalize,
  nameinlink,
]{cleveref}
\hypersetup{%
  linktoc=all, 
  colorlinks=true,
}
\AtBeginDocument{
\hypersetup{%
  urlcolor=cyan,
}}

\crefname{section}{\S}{\S}
\NewDocumentCommand\Crefnameitem { m m m O{\textup} O{(\roman*)}} {%
  \Crefname{#1enumi}{#2}{#3} 
  \AtBeginEnvironment{#1}{%
    \crefalias{enumi}{#1enumi}%
    \setlist[enumerate,1]{
        label={#4{#5}.},
        ref={#5}
    }%
  }  
}

\setsecnum{1.a,.i}

\biboptions{sort&compress} 

\makeatletter
\patchcmd\NAT@citexnum{\let\NAT@last@num\NAT@num}{\MakeLinkTarget[cite]{}\Hy@backout{\@citeb\@extra@b@citeb}\let\NAT@last@num\NAT@num}{}{\fail}
\makeatother
\numberwithin{equation}{section}

\newtagform{bold}[\textbf]{(}{)}
\usetagform{bold}

\theoremstyle{plain}
\newtheorem{THM}{Theorem}

\newtheorem{Cor}[THM]{Corollary}

\swapnumbers
\newtheorem{theorem}[equation]{Theorem}
\newtheorem{proposition}[equation]{Proposition}
\newtheorem{corollary}[equation]{Corollary}
\newtheorem{lemma}[equation]{Lemma}

\makeatletter
\renewcommand\swappedhead[3]{%
  \thmnumber{\textbf{\textup{#2}.}}%
  \thmname{\@ifnotempty{#2}{~}#1}%
  \thmnote{ {\the\thm@notefont(#3)}}}
\makeatother

\theoremstyle{remark}
\newtheorem{remark}[equation]{Remark}
\newtheorem{warn}[equation]{Warning}

\newtheorem{example}[equation]{Example}
\newtheorem{definition}[equation]{Definition}

\newtheorem{notation}[equation]{Notation}
\Crefname{construction}{Construction}{Constructions}

\Crefnameitem{theorem}{Theorem}{Theorems}
\Crefnameitem{lemma}{Lemma}{Lemmas}

\usetikzlibrary{%
  cd, 
  shapes, 
  arrows, 
  fit, 
}
\NewDocumentCommand\placeholder{}{\:\cdot\:} 
\NewDocumentCommand\NewPairedDelimiterS{mmm}{%
  \DeclarePairedDelimiterX{#1}[1]{#2}{#3}%
    {\ifblank{##1}{\placeholder}{##1}}%
}
\NewDocumentCommand\NewPairedDelimiterSS{mmmO{,}}{%
  \DeclarePairedDelimiterX{#1}[2]{#2}{#3}%
    {\ifblank{##1}{\placeholder}{##1}%
    #4%
    \ifblank{##2}{\placeholder}{##2}}%
}
\NewPairedDelimiterS\abs\lvert\rvert 
\NewPairedDelimiterS\norm\lVert\rVert 
\NewPairedDelimiterS\lrangle\langle\longrightarrowngle 
\NewPairedDelimiterS\paren\lparen\rparen 
\NewPairedDelimiterS\lrbrack\lbrack\rbrack 
\NewPairedDelimiterS\lrbrace\lbrace\rbrace 
\NewPairedDelimiterS\ceil\lceil\rceil 
\NewPairedDelimiterS\floor\lfloor\rfloor 
\NewPairedDelimiterS\bra\langle\vert 
\NewPairedDelimiterS\ket\vert\rangle 
\NewPairedDelimiterS\normord{\mathopen{:}}{\mathclose{:}} 
\NewPairedDelimiterSS\braket\langle\rangle[%
  \,\delimsize\vert\,\mathopen{}%
] 
\NewPairedDelimiterSS\pairing\langle\rangle
\NewPairedDelimiterSS\inner\lparen\rparen
\NewPairedDelimiterSS\Liebracket\lbrack\rbrack
\DeclarePairedDelimiterX\bracket[3]\langle\rangle%
  {\ifblank{#1}{\placeholder}{#1}%
  \,\delimsize\vert\,\mathopen{}%
  \ifblank{#2}{\placeholder}{#2}
  \,\delimsize\vert\,\mathopen{}%
  \ifblank{#3}{\placeholder}{#3}}%
\NewDocumentMathCommand\dparen{m}{\lparen\!\lparen{#1}\rparen\!\rparen}
\NewDocumentMathCommand\dbrack{m}{\lbrack\!\lbrack{#1}\rbrack\!\rbrack}
\providecommand\given{}
\NewDocumentCommand \SetSymbol {o}
  { \nonscript\:#1\vert\allowbreak\nonscript\:\mathopen{} }
\DeclarePairedDelimiterX\Set[1]\{\}
  { \renewcommand\given{\SetSymbol[\delimsize]}#1 }
\DeclarePairedDelimiterX\GSet[1]\langle\rangle
  { \renewcommand\given{\SetSymbol[\delimsize]}#1 }
\intervalconfig{soft open fences}

\RenewDocumentCommand \subset {} {\subseteq}

\renewmathcommand\le\leqslant
\renewmathcommand\ge\geqslant
\NewDocumentCommand \concept {m} {\textbf{#1}}
\NewDocumentCommand \splitle {O{\ast}} {\[#1#1#1\]}
\NewDocumentCommand \red {m} {{\color{red}#1}}

\DeclareRobustCommand\longepimorphism
  {\relbar\joinrel\twoheadrightarrow} 

\NewDocumentCommand\epimorphism{}{\twoheadrightarrow}
\DeclareRobustCommand\longmonomorphism
  {\lhook\joinrel\longrightarrow} 

\DeclareRobustCommand\monomorphism
  {\lhook\joinrel\rightarrow} 
\DeclareRobustCommand\dashrightarrow
  {\relbar\rightarrow} 
\declaremathcommand\dashto{\dashrightarrow}
\NewDocumentCommand \txforall {O{\qquad}} {#1\text{for all}#1}%
\NewDocumentCommand \txand {O{\qquad}} {#1\text{and}#1}%
\NewDocumentCommand \txst {O{~}} {#1\text{s.t.}#1}%
\NewDocumentMathCommand \sequence { m O{1} O{n} }
  { \ensuremath{ {#1}_{#2}, \cdots, {#1}_{#3} } }
\NewDocumentMathCommand \supsequence { m O{1} O{n} }
  { \ensuremath{ {#1}^{#2}, \cdots, {#1}^{#3} } }
\NewDocumentCommand \fun { m e{^_} O{} }
  {%
    \operatorname{#1}%
    \IfValueT{#2}{\sp{#2}}%
    \IfValueT{#3}{\sb{#3}}%
    \ifblank{#4}{}{\mleft(#4\mright)}%
  }
\LoopCommands{%
  {Arg}%
  {Aut}%
  {Coker}%
  {coker}%
  {Cor}%
  {Der}%
  {End}%
  {Ext}%
  {Gal}%
  {Gr}%
  {gl}%
  {gr}%
  {Ho}%
  {Hom}%
  {Id}%
  {id}%
  {Ind}%
  {Ker}%
  {Lie}%
  {Log}%
  {Mod}%
  {Mor}%
  {Ob}%
  {op}%
  {Pic}%
  {Proj}%
  {pr}%
  {pt}%
  {Rad}%
  {Res}%
  {Spec}%
  {Sq}%
  {supp}%
  {Tor}%
  {tr}%
  {wt}%
}[#1]
  {\declaremathcommand#2{\fun{#1}}}
\NewDocumentCommand\grp{m}{\fun{\mathsf{#1}}}
\NewDocumentMathCommand\Dfrac{mm}{%
  \dfrac{\displaystyle #1}{\displaystyle #2}%
}
\NewDocumentMathCommand\odv{m}{\frac{\d}{\d{#1}}}
\NewDocumentMathCommand\pdv{m}{\frac{\partial}{\partial{#1}}}
\LoopCommands{GNZQRCFKTDWM}[#1]
  {\declaremathcommand#2{\mathbb#1}}

\LoopCommands{ADEFGHIJKMNOSUVWY}[c#1]
  {\declaremathcommand#2{\mathcal#1}}

\LoopCommands{abcduvw}[sp#1]
  {\declaremathcommand#2{\mathscr#1}}
  
\LoopCommands{FJLMNRU}[sf#1]
  {\declaremathcommand#2{\mathsf#1}}

\LoopCommands{mnpq}[#1#1]
  {\declaremathcommand#2{\mathfrak#1}}
\LoopCommands{hgxX}[#1]
  {\declaremathcommand#2{\mathfrak#1}}

\declaremathcommand\iu{\mathbb{i}}
\declaremathcommand\L{\mathcal{L}}
\declaremathcommand\U{\mathscr{U}}
\declaremathcommand\H{\mathsf{H}}
\declaremathcommand\o{\otimes}
\declaremathcommand\spn{\fun{span}}
\declaremathcommand\Cc{{C}^{\circ}}
\declaremathcommand\oC{\overline{C}}
\declaremathcommand\O{\mathcal{O}}
\declaremathcommand\PP{\mathbb{P}}
\declaremathcommand\la{\lambda}
\declaremathcommand\ds{\dots}
\declaremathcommand\al{\alpha}
\declaremathcommand\ra{\longrightarrow}
\declaremathcommand\tilX{\widetilde{\X}}
\declaremathcommand\Xc{{\X}^{\circ}}
\declaremathcommand\gr{\mathsf{gr}}

\declaremathcommand\Pc{\mathring{\mathbb{P}}^1}
\declaremathcommand\Cfb{\fun{\mathscr{C}}}
\declaremathcommand\Image{\fun{Im}}
\declaremathcommand\dQ{\mathfrak{Q}}
\declaremathcommand\dv{\fun{div}}

\NewDocumentMathCommand\fusion{O{M^1}O{M^2}O{M^3}}{\textstyle\binom{#3}{#1\;#2}}
\NewDocumentMathCommand\rfusion{O{M^1}O{M^2(0)}O{M^3(0)}}{\textstyle\binom{#3}{#1\;#2}}
\declaremathcommand\Fusion{\mathfrak{I}\fusion}
\declaremathcommand\Nusion{N\fusion}
\newmathcommand\vac{\mathbb{1}}
\NewDocumentMathCommand\delfun{mmO{{#2}^{-1}}}{#3\delta\left(\dfrac{#1}{#2}\right)}
\newmathcommand\del{\fun{\delta}}
\NewDocumentMathCommand\pfrac{mm}{\left(\dfrac{#1}{#2}\right)}
\NewDocumentMathCommand\vo{mm}{#1_{(#2)}}
\NewDocumentMathCommand\lo{mm}{#1_{[#2]}}
\NewDocumentMathCommand\loo{mO{}mm}{{#1^{\ifblank{#2}{}{(#2)}}}^{[#3]}_{[#4]}}
\NewDocumentMathCommand\lqo{mmm}{#1_{#2[#3]}}
\NewDocumentMathCommand\sqo{mmm}{#1_{#2}^{(#3)}}

\declaremathcommand\ptseq{\sequence{\pp}}
\declaremathcommand\zseq{\sequence{z}}
\declaremathcommand\rseq{\sequence{r}}
\declaremathcommand\aseq{\supsequence{a}}
\NewDocumentMathCommand \cfbseq { O{1} O{n} }
  { \ensuremath{ (a^{#1},\pp_{#1})\cdots(a^{#2},\pp_{#2}) } }
\NewDocumentMathCommand \ainVseq { O{1} O{n} }
  { \ensuremath{ a^{#1}\in V^{r_{#1}}, \cdots, a^{#2}\in V^{r_{#2}} } }
\NewDocumentMathCommand \cfbseqb { O{1} O{p} }
  { \ensuremath{ (b^{#1},\pp_{+#1})\cdots(b^{#2},\pp_{+#2}) } }
\NewDocumentMathCommand \lobmseq { O{1} O{p} }
  { \ensuremath{ \lo{b^{#1}}{\frac{r_{#1}}{T}+m_{#1}}\cdots\lo{b^{#2}}{\frac{r_{#2}}{T}+m_{#2}} } }

\NewDocumentCommand \mybox {m} {\fbox{\begin{tabular}{@{}c@{}}#1\end{tabular}}}

\NewDocumentMathCommand\charge{m}{V_{L}^{#1}}
\NewDocumentMathCommand\charhalf{mm}{V_{#1+L}^{#2}}
\NewDocumentMathCommand\charlam{m}{V_{#1+L}}

\allowdisplaybreaks
\journal{ArXiv}

\begin{document}

\begin{frontmatter}

\title{Twisted restricted conformal blocks of vertex operator algebras II: twisted restricted conformal blocks on totally ramified orbicurves}

\author[1]{{Xu} {Gao}}
\ead{gausyu@tongji.edu.cn}

\author[2]{{Jianqi} {Liu}}
\ead{jliu230@sas.upenn.edu}

\author[3]{{Yiyi} {Zhu}}
\ead{yzhu51@scut.edu.cn}

\affiliation[1]{
  organization={School of Mathematical Science, Tongji University}, 
  addressline={1239 Siping Road}, 
  city={Shanghai}, 
  postcode={200092}, 
  state={Shanghai}, 
  country={China},
}

\affiliation[2]{
  organization={Department of Mathematics, University of Pennsylvania}, 
  addressline={209 South 33rd Street}, 
  city={Philadelphia}, 
  postcode={19104}, 
  state={PA}, 
  country={USA},
}

\affiliation[3]{
  organization={Department of Mathematics, South China University of Technology}, 
  addressline={381 Wushan Road}, 
  city={Guangzhou}, 
  postcode={510641}, 
  state={Guangdong}, 
  country={China},
}

\begin{abstract}
In this paper, we introduce a notion of twisted restricted conformal blocks on totally ramified orbicurves and establish an isomorphism between the space of twisted restricted conformal blocks and the space of twisted conformal blocks. The relationships among twisted (restricted) conformal blocks, $g$-twisted (restricted) correlation functions, and twisted intertwining operators are explored. Furthermore, by introducing a geometric generalization of Zhu's algebra and its modules, we obtain a description of the space of coinvariants by modules over associative algebras and show it is finite-dimensional under some conditions. 
In particular, a more conceptual proof of the $g$-twisted fusion rules theorem in vertex operator algebra theory is provided. 
\end{abstract}


\end{frontmatter}

\tableofcontents

\section{Introduction}
This paper is the second part of a series aimed at investigating the general theory of twisted (restricted) conformal blocks of vertex operator algebras (VOAs). 
Here, we focus on the twisted restricted VOA-conformal blocks on a totally ramified orbicurve and shift the viewpoint from \emph{correlation functions} to \emph{conformal blocks}. 
An overview of twisted modules of vertex operator algebras and fusion rules among twisted modules was provided in \cite{PartI}.

\subsection{History of VOA-conformal blocks}

A few years after Belavin, Polyakov, and Zamolodchikov introduced the two-dimensional conformal field theory with Virasoro symmetries \cite{BPZ84}, Tsuchiya, Kanie, Ueno, and Yamada formulated a mathematically rigorous construction of the holomorphic conformal field theory with gauge symmetries given by affine Kac-Moody algebras $\hat{\g}$ in \cite{TK, TUY89}, corresponding to the Wess-Zumino-Novikov-Witten model in physics. 
As the mathematical counterpart of the physical conformal blocks in \cite{BPZ84}, 
the space of conformal blocks (vacua) and the space of coinvariants (covacua) on an $n$-pointed stable curve were introduced in \cite{TUY89} and constructed explicitly using the integrable highest-weight representations in a fixed level $\ell$ over $\hat{\g}$. 
The algebro-geometric formulation in \cite{TUY89} facilitated extensive mathematical studies of the WZNW conformal field theory.

From the VOA point of view, the WZNW conformal field theory can be formulated using the affine VOAs $L_{\hat{\g}}(\ell,0)$ constructed by Frenkel and Zhu in \cite{FZ}. 
More generally, the study of conformal field theory can be understood as the study of the modular tensor category structure on the representations of a VOA (see \cite{H97}). 
In such a study, the notion of \emph{intertwining operators}, introduced in \cite{FHL}, plays a crucial role: the dimension of the space of intertwining operators is a mathematical formulation of the notion of \emph{fusion rules}. 
The finiteness of the fusion rules and a formula to compute them are thus fundamental in the theory. In \cite{FZ}, I. Frenkel and Zhu claimed such a formula using the bimodule theory over \emph{Zhu's algebra} $A(V)$. 
This formula, known as the \emph{fusion rules theorem}, was first proved by Li in \cite{Li, Li3} under the framework of $A(V)$-theory. 

The history up to this stage (i.e. the early 90s) can be summarized in the following diagram, where the general notions of conformal blocks and correlation functions were still shrouded in the mystery of physics.
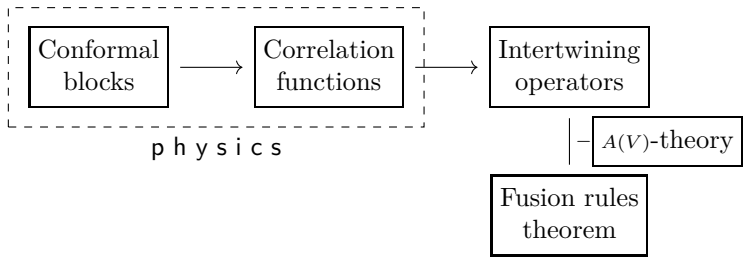
\[
  \begin{tikzcd}[/tikz/execute at end picture={%
    \node (large) [rectangle, draw, dashed, fit=(A1) (A2),label=below:\textsf{p~h~y~s~i~c~s}] {};%
    }]
    |[alias=A1]|\mybox{Conformal\\ blocks}\arrow[r] &
    |[alias=A2]|\mybox{Correlation\\ functions}\arrow[r] & 
    \mybox{Intertwining\\ operators}\arrow[d,no head,"-\mybox{$A(V)$-theory}"] \\
    && \mybox{Fusion rules\\ theorem}
  \end{tikzcd}
\]

The mathematical notion of (genus-zero and genus-one) \emph{correlation functions} was introduced in \cite{FZ,Z}. These functions are constructed using intertwining operators and have been extensively studied by Huang in \cite{H05a,H05b}.
With such a robust notion, the second author gave another proof of the \emph{fusion rules theorem} using correlation functions in \cite{Liu}. 

Moving one step further, a rigorous mathematical formulation of \emph{conformal blocks} would unveil their mystery. 
However, this step took a long time. 
The initial effort is due to Zhu \cite{Z94}, who proposed a formulation generalizing the constructions of Tsuchiya-Ueno-Yamada \cite{TUY89} by introducing \emph{global vertex operators}. The finiteness of conformal blocks under this formulation was later established by Abe and Nagatomo in \cite{AN1, AN2}.  
Zhu's pioneering work, along with contributions from \cite{H97,MSV,MS99}, paved the way for the celebrated algebro-geometric formulation of \emph{conformal blocks on smooth curves} by Beilinson-Drinfeld \cite{BD04} and Frenkel-Ben-Zvi \cite{FBZ04}. 
Drawing insights from conformal field theory, it is anticipated that the conformal blocks should have the \emph{factorization property} or \emph{sewing property}. 
In the genus-zero case, Nagatomo and Tsuchiya proved these properties in \cite{NT}.
The general case, with an extended formulation on arbitrary stable curves, was established by Damiolini, Gibney, Krashen, and Tarasca in \cite{DGT19,DGK23}. 
Meanwhile, in the analytic realm, a similar long march toward the proof of the \emph{convergence of sewing} was initiated by Huang \cite{H05a,H05b} and completed recently by Gui in \cite{G}.

\splitle

The notion of \emph{twisted modules} of VOAs originated in the realization of irreducible representations of twisted affine Lie algebras \cite{Lepowsky}, the construction of the celebrated moonshine module \cite{FLM}, and the study of orbifold models in conformal field theory \cite{DHVW86,DVVV89}. 
The \emph{orbifold theory} of VOAs aims to systematically investigate these objects and has been continuously developed over the past few decades, see e.g. \cite{D94,DLM1,DLM00,BDM02,DX06,H10,DRX17}. 
However, despite this progress, many fundamental topics remain unclear, including a general version of \emph{twisted fusion rule theorem} and the construction of tensor product for twisted modules.

The difficulty of orbifold theory arises from its reliance on the intricate formal calculus of Puiseux series, making constructions and computations exceedingly challenging. 
Backing to its physical origin, the mathematical formulation of \emph{twisted correlation functions} makes it possible to overcome such a difficulty by using actual rational functions on orbicurves rather than formal Puiseux series. 
A special but important case of such a formulation has been developed in \cite{PartI}, where we applied it to produce a $g$-twisted version of the \emph{fusion rules theorem} following the approach in \cite{Liu}. 
However, the proof employed in \cite{PartI} still necessitates a substantial amount of computation. 
Given the theoretical nature of conformal blocks, they should offer a proof requiring less computation compared to correlation functions. 
The algebro-geometric formulation of \emph{twisted conformal blocks} by Frenkel and Szczesny in \cite{FBZ04,FS04} laid a foundation for this approach.

\subsection{Main results}

Following \cite{FBZ04, FS04}, associated to the datum $\Sigma=(\x\colon\tilX\to\X,\pp_{\bullet},\qq_{\diamond}, M^{\bullet}, N^{\diamond})$, where $\x\colon\tilX\to\X$ is a \emph{totally ramified orbicurve} with the \emph{Galois group} $G$ (see \cref{def:orbicurve}), $\pp_{\bullet}$ is an $m$-tuple of distinct points on the unramified locus $\Xc$, $\qq_{\diamond}$ are all the branch points, each provides a \emph{Deck generator} $g_{i}$ of $G$, $M^{\bullet}$ is an $m$-tuple of admissible \emph{untwisted} $V$-modules, with each $M^i$ attached to the point $\pp_{i}$, and $N^{\diamond}$ is a family of admissible \emph{$g_{\diamond}$-twisted} $V$-modules, with each $N^i$ attached to the branch point $\qq_{i}$, a \textbf{twisted conformal block} is a linear functional $\varphi$ on $\cM^{\bullet}_{\pp_{\bullet}}\otimes\cN^{\diamond}_{\qq_{\diamond}}$ that is invariant under the \emph{twisted chiral Lie algebra} $\L^{\circ}_{\Xc\setminus\pp_{\bullet}}(\cV^G)$ (see \cref{def:twistedchiralLie}). 

Mimicking the above definition and consider the following datum: 
\begin{equation}\label{eq:1.2}
\Sigma(0)=(\x\colon\tilX\to\X,\pp_{\bullet},\qq_{\diamond},M^{\bullet},U^{\diamond}),
\end{equation}
where we replace the family $N^{\diamond}$ of $g_{\diamond}$-twisted $V$-modules attached to the branch points $\qq_{\diamond}$ by a family $U^{\diamond}$ of modules over the \emph{$g_\diamond$-twisted Zhu's algebra $A_{g_{\diamond}}(V)$}.
We introduce the notion of a \textbf{twisted restricted conformal block associated to the datum $\Sigma(0)$} as a linear functional $\varphi$ on $\cM^{\bullet}_{\pp_{\bullet}}\otimes\cU^{\diamond}_{\qq_{\diamond}}$ that is invariant under the \emph{constrained twisted chiral Lie algebra} $  \L^{\circ}_{\X\setminus\pp_{\bullet}}(\cV^G)_{\le0}$ (see \cref{def:constraint}).

When $U^{\diamond}$ are the corresponding bottom levels of the $g_{\diamond}$-twisted $V$-modules $N^{\diamond}$, which are automatically $A_{g_{\diamond}}(V)$-modules, 
the inclusion $\cU^{\diamond}\monomorphism\cN^{\diamond}$ induces a canonical linear map $\uppi\colon\Cfb[\Sigma]\to\Cfb[\Sigma(0)]$
from the space $\Cfb[\Sigma]$ of twisted conformal blocks associated to the datum $\Sigma$ to the space $\Cfb[\Sigma(0)]$ of twisted restricted conformal blocks associated to the datum $\Sigma(0)$. 
One of our main results is the following theorem (see \cref{thm:pi-inj} and \cref{thm:pi-surj}), which demonstrates an explicit relation between restricted and unrestricted conformal blocks: 

\begin{THM}\label{thm:main}
If $N^{\diamond}$ are lowest-weight $g_\diamond$-twisted $V$-modules, then $\uppi$ is injective. 
If, furthermore, $N^{\diamond}$ are $g_\diamond$-twisted generalized Verma modules of their bottom levels, then $\uppi$ is an isomorphism. 
In particular, if $V$ is $g_\diamond$-rational, the canonical map $\uppi$ is an isomorphism for all irreducible $g_\diamond$-twisted $V$-modules $N^{\diamond}$.
\end{THM}

\splitle

The main objects investigated in \cite{PartI} are the \emph{$g$-twisted (restricted) correlation functions} on the twisted projective line  $\x\colon\oC\to\PP^1$.
These notions generalize their prototypes in the WZNW model \cite[Section 2.4]{TUY89} and the genus-zero (restricted) VOA-correlation functions in \cite{Liu,Z}. 
In the present paper, besides the introduction of the general twisted restricted conformal blocks, we also studied the relations between $g$-twisted (restricted) correlation functions and $3$-pointed twisted (restricted) conformal blocks on the \emph{twisted projective line}  $\x\colon\oC\to\PP^1$.
The following is our main result in \cref{sec:Cfb} (see \cref{thm:Cfb=Cor}): 

\begin{THM}\label{thm:main2}
  Let $M^1$ (resp. $M^2$ and $N^3$) be an admissible \emph{untwisted} (resp. \emph{$g$-twisted} and \emph{$g^{-1}$-twisted}) $V$-module of conformal weight $h_1$ (resp. $h_2$ and $h_3$). 
  Then the space of twisted conformal blocks $\Cfb[\Sigma_{1}(N^3, M^1, M^2)]$ associated to the datum
  \begin{equation}\label{eq:Sigma1}
    \Sigma_{1}(N^3, M^1, M^2):=(\x\colon\oC\to\PP^1, \infty, 1, 0, N^3, M^1, M^2)
  \end{equation}
  is isomorphic to the space $\Cor[\Sigma_{1}(N^3, M^1, M^2)]$ of $g$-twisted correlation functions associated to the same datum. 
\end{THM}
This theorem generalizes the correspondence between the WZNW-conformal blocks and the WZNW-correlation functions (\cite[Theorem 2.4.1]{TUY89}) to the ($g$-twisted) general VOA-scenario. 
Using a similar method, we prove the same correspondence between the twisted restricted conformal blocks and the $g$-twisted restricted correlation functions in  \cref{sec:RestrictedCfb} (see \cref{thm:Cfb=Cor_res}):

\begin{THM}\label{thm:main3}
 Let $M^1$ be an admissible untwisted $V$-module of some conformal weight $h_1$, and let $U^2$ (resp. $U^3$) be a left (resp. right) $A_g(V)$-module on which $[\upomega]$ acts as a scale $h_2$ (resp. $h_3$). Then the space of twisted restricted conformal blocks $\Cfb[\Sigma_{1}(U^3, M^1, U^2)]$  associated to the datum
  \begin{equation}\label{eq:Sigma1(0)}
    \Sigma_{1}(U^3, M^1, U^2):=(\x\colon\oC\to\PP^1, \infty, 1, 0, U^3, M^1, U^2).
  \end{equation}
  is isomorphic to the space $\Cor[\Sigma_{1}(U^3, M^1, U^2)]$ of $g$-twisted restricted correlation functions associated to the same datum. 
\end{THM}

\splitle

A main application in \cite{PartI} is the \emph{$g$-twisted fusion rules theorem}. 
By definition, the \emph{fusion rule} associated to $V$-modules $M^1$, $M^2$, and $M^3$ is the dimension of the space $\Fusion$ of \emph{intertwining operators} among them. 
In \cite{PartI}, we have considered the case where $M^1$ is \emph{untwisted} while $M^2$ and $M^3$ are \emph{$g$-twisted}. 
The introduction of twisted conformal blocks on totally ramified orbicurves enables us to consider a more general case: where $M^1$ (resp. $M^2$ and $M^3$) is \emph{$g_1$-twisted} (resp. \emph{$g_2$-twisted} and \emph{$g_3$-twisted}) such that $g_1g_2=g_3$, and one of the following conditions holds:
\begin{enumerate}
    \item $g_1=\id$; 
    \item each of $g_1$, $g_2$, and $g_3$ generates the same cyclic group.
\end{enumerate}
The datum $\Sigma_{1}((M^3)', M^1, M^2)$ corresponds to the first situation. As for the second one, we consider the following datum: 
\begin{equation}\label{eq:Sigmaw}
  \Sigma_{w}\left((M^3)', M^1, M^2\right):=\left(\x\colon\oC_{w}\to\PP^1, \infty, w, 0, (M^3)', M^1, M^2\right),
\end{equation}
where $\x\colon\oC_w\to\PP^1$ is an obicurve with exactly three branch points $w$, $0$, and $\infty$ (see \cref{eg:3ptP1}) with Deck generators $g_1$, $g_2$, and $g_3$ respectively.
Then we have a theorem stating the relationship between twisted conformal blocks and twisted intertwining operators (see \cref{thm:IO}): 
\begin{THM}\label{thm:main5}
    Let $M^k$ be a $g_k$-twisted $V$-module of conformal weight $h_k$ for $k=1, 2, 3$. Set $h=h_1+h_2-h_3$. Then, as linear spaces,
    \[\Cfb\left(\Sigma\left((M^3)', M^1, M^2\right)\right)\cong \Fusion,\]
    here $\Sigma$ stands for $\Sigma_1$ or $\Sigma_w$.
\end{THM}

\splitle

In the study of \emph{$g$-twisted correlation functions} and fusion rules in \cite{PartI}, we introduced a space $\Cfb[U^3, M^1, U^2]=((U^3\otimes M^1\otimes U^2)/J)^\ast$ termed as the space of \emph{$3$-pointed $g$-twisted restrict conformal blocks} associated to the datum \labelcref{eq:Sigma1(0)}. 
This datum is a special case of \labelcref{eq:1.2}: 
the \emph{twisted projective line} is a totally ramified obicurve, 
$M^1$ is an irreducible $V$-module attached to an unramified point $1\in\PP^1$, 
and $U^2$ (resp. $U^3$) is an irreducible left (resp. right) module over the \emph{$g$-twisted Zhu's algebra} $A_g(V)$ \cite{DLM1} attached to the branch point $0\in \PP^1$ (resp. $\infty \in \PP^1$). 
The relation subspace $J$ in the definition of $\Cfb[U^3, M^1, U^2]$ was found through the calculation of \emph{recursive formulas} of $g$-twisted restricted correlation functions in $\Cor(\Sigma_{1}(U^3, M^1, U^2))$, see \cite[Definition 4.1]{PartI}. 
Our main results in \cite{PartI} establish isomorphisms among the spaces $\Cfb[U^3, M^1, U^2]$, $\Cor (\Sigma_{1}(U^3, M^1, U^2))$, and $(U^3\otimes_{A_g(V)} A_g(M^1)\otimes_{A_g(V)} U^2)^\ast$, leading to the \emph{$g$-twisted fusion rules theorem}. 

In this paper, we prove the following theorem (see \cref{thm:Cfb=Cfb}):
\begin{THM}\label{thm:main4}
    The subpaces $\Cfb[\Sigma_{1}(U^3, M^1, U^2)]$ and $\Cfb[U^3, M^1, U^2]$ coincide in the space of linear functionals $(U^3\otimes M^1\otimes U^2)^{\ast}$.
\end{THM}
Hence $\Cfb[U^3, M^1, U^2]$ is just a special case of our space of twisted restricted conformal blocks on totally ramified orbicurves. 

\splitle

The space of (twisted restricted) \emph{coinvariants} associated to the datum \labelcref{eq:1.2} is the largest quotient of  $\cM^{\bullet}_{\pp_{\bullet}}\otimes\cU^{\diamond}_{\qq_{\diamond}}$ with respect to the action of constrained chiral Lie algebra. We denote this space by $(\cM^{\bullet}_{\pp_{\bullet}}\otimes\cU^{\diamond}_{\qq_{\diamond}})_{\L^{\circ}_{\X\setminus\pp_{\bullet}}(\cV^G)_{\le0}}$. Its dual space is the space of twisted restricted conformal blocks $\Cfb[\Sigma(0)]$.

After introducing a geometric generalization of the \emph{twisted Zhu's algebra over an affine open} (see \cref{def:Zhuoveraffine}), we give an 
 $A(V)$-theoretical description of the space of restricted coninvariants (see \cref{thm:coinvariants}):

\begin{THM}\label{thm:main'}
    We have the following isomorphism of vector spaces: 
    \[
    (\cM^{\bullet}_{\pp_{\bullet}}\otimes\cU^{\diamond}_{\qq_{\diamond}})_{\L^{\circ}_{\X\setminus\pp_{\bullet}}(\cV^G)_{\le0}} 
    \cong 
    \cA_{\X\setminus\pp_\bullet}(\cM^{\bullet}_{\pp_{\bullet}})\otimes_{\cA_{\X\setminus\pp_\bullet}(\cV^G)}\cU^{\diamond}_{\qq_{\diamond}}.
    \]  
\end{THM}
The $\cA_{\X\setminus\pp_\bullet}(\cV^G)$-module $\cA_{\X\setminus\pp_\bullet}(\cM^{\bullet}_{\pp_{\bullet}})$ is introduced in \cref{def:A(M)} and it is the space of coinvariants $(\cM^{\bullet}_{\pp_{\bullet}})_{\L^{\circ}_{\X\setminus\pp_{\bullet}}(\cV^G)_{<0}}$ with respect to the constrained chiral Lie algebra $\L^{\circ}_{\X\setminus\pp_{\bullet}}(\cV^G)_{<0}$. 
The upshot is that it is defined over an unramified locus and thus can be studied in the framework of untwisted conformal blocks.
 
We can view the associative algebra $\cA_{\X\setminus\pp_\bullet}(\cV^G)$ and its module $\cA_{\X\setminus\pp_\bullet}(\cM^{\bullet}_{\pp_{\bullet}})$ as geometric generalizations of the $g$-twisted Zhu's algebra $A_g(V)$ \cite{DLM1} and its bimodules $A_g(M)$ \cite{FZ,PartI,JJ}. As a corollary of \cref{thm:main} and \cref{thm:main'}, we have the following identification (see \cref{cor:FusionRule}), which can be viewed as a generalization of the {\em $g$-twisted fusion rules theorem} to $n$-points conformal blocks on arbitrary totally ramified orbicurves: 
\begin{Cor}\label{Coro}
  If $N^{\diamond}$ are lowest-weight $g_\diamond$-twisted generalized Verma modules, then we have isomorphisms of vector spaces
  \[
  \Cfb[\Sigma]\cong\Cfb[\Sigma(0)]
  \cong 
  \left(\cA_{\X\setminus\pp_\bullet}(\cM^{\bullet}_{\pp_{\bullet}})\otimes_{\cA_{\X\setminus\pp_\bullet}(\cV^G)}\cU^{\diamond}_{\qq_{\diamond}}\right)^{\ast}.
  \]
\end{Cor}

\cref{Coro} also leads to the finiteness of twisted conformal blocks and fusion rules among $g_1,g_2,$ and $g_3$-twisted $V$-modules $M^1,M^2$, and $M^3$, where $g_1,g_2,g_3$ generate a cyclic group of order $T$, for a $C_2$-cofinite VOA $V$ (see \cref{thm:finiteness}). 

\splitle

To summarize, for the datum $\Sigma_1(N^3,M^1,M^2)$ \labelcref{eq:Sigma1}, where $M^1$ is \emph{untwisted} and $M^2$ (resp. $N^3$) is a \emph{$g$-twisted} (resp. \emph{$g^{-1}$-twisted}) generalized Verma module with bottom level $U^2$ (resp. $U^3$), 
we have the following commutative diagrams of isomorphisms: 
\begin{equation*}
  \begin{tikzcd}[column sep=-0.1in,row sep=0.5in]
  &\Fusion
  \arrow[dl, leftrightarrow, "\text{\cite[Theorem 2.24]{PartI}}"',"{N^3=(M^3)'}"]
  \arrow[dr, leftrightarrow, "\text{\cref{thm:main5}}", "{N^3=(M^3)'}"']&\\
  \Cor[\Sigma_{1}(N^3, M^1, M^2)]
  \arrow[rr,leftrightarrow,"\text{\cref{thm:main2}}"] 
  \arrow[d,leftrightarrow, "\text{\cite[Theorem 3.16]{PartI}}"']&& 
  \Cfb[\Sigma_{1}(N^3, M^1, M^2)]
  \arrow[d,leftrightarrow, "\text{\cref{thm:main}}"] \\
  \Cor[\Sigma_{1}(U^3, M^1, U^2)]
  \arrow[rr,leftrightarrow,"\text{\cref{thm:main3}}"]
  \arrow[dr, leftrightarrow, "\text{\cite[Theorem 4.5]{PartI}}"']&& 
  \Cfb[\Sigma_{1}(U^3, M^1, U^2)]
  \arrow[dl, leftrightarrow, "\text{\cref{thm:main4}}"]
  \arrow[dd, leftrightarrow, "\text{\cref{Coro}}"]\\
  &\Cfb[U^3, M^1, U^2]
  \arrow[dl,leftrightarrow,"\text{\cite[Theorem 6.5]{PartI}}"']&\\
  {\scriptstyle \left(U^3\otimes_{A_g(V)} A_g(M^1)\otimes_{A_g(V)} U^2\right)^\ast}
  \arrow[rr, leftrightarrow, "\text{\cref{eg:Ag(M)}}"]&&
  {\scriptstyle \left(\cA_{\PP^1\setminus\Set{1}}(\cM^1_{1})\otimes_{\cA_{\PP^1\setminus\Set{1}}(\cV^G)}(U^3\otimes U^2)\right)^\ast}
  \end{tikzcd}
\end{equation*}
The left half of the diagram gives a proof of \emph{$g$-twisted fusion rules theorem} using correlation functions, which was accomplished in \cite{PartI} and almost computational. The right half gives a more conceptual proof using twisted conformal blocks. 
Furthermore, the right half generalizes to the datum $\Sigma_{w}(N^3,M^1,M^2)$ \labelcref{eq:Sigmaw}, where $M^2$ (resp. $M^1$ and $N^3$) is \emph{$g$-twisted} (resp. \emph{$g^\epsilon$-twisted} and \emph{$g^{-\epsilon-1}$-twisted}) generalized Verma module with bottom level $U^2$ (resp. $U^1$ and $U^3$) and yields (see \cref{eg:CfbSn+1})
\[
  \Fusion \cong \left(\cA_{\PP^1\setminus\pp}(\cV_\pp)\otimes_{\cA_{\PP^1\setminus\pp}(\cV^G)}(U^3\otimes U^1\otimes U^2)\right)^\ast,
\]
where $(M^3)'=N^3$.

\splitle

Our main result \cref{thm:main} is a twisted and VOA generalization of a similar lifting theorem (see \cite[Theorem 3.18]{Ueno}) between the restricted WZNW-conformal blocks defined by $\g$-modules and WZNW-conformal blocks defined by $\hat{\g}$-modules. 
The action of the finite-dimensional Lie algebra $\g$ on the bottom level $V_{\la_i}$ of an integrable highest-weight $\hat{\g}$-module $\mathcal{H}_{\lambda_{i}}$ in \cite{TUY89,Ueno} is now replaced by the action of the $g$-twisted Zhu's algebra $A_g(V)$ on the bottom level of a $g$-twisted VOA-module. 
We believe the consequences of \cite[Theorem 3.18]{Ueno}, such as the factorization property, in the theory of WZNW-conformal blocks would still be valid for the general twisted VOA-conformal blocks.

In a recent work \cite{DG23}, Damiolini and Gibney proved that the vector bundle of coinvariants $\mathbb{V}(V; W^\bullet)$ on the moduli space $\overline{\mathcal{M}}_{g,n}$ is globally generated when the VOA $V$ is strongly generated in degree $1$. The proof depends on the existence of a surjective map $W^1(0)\otimes \cdots \otimes W^N(0)\ra \mathbb{V}(V; W^\bullet)$, where $W^i(0)$ is the bottom level of $W^i$, which is difficult to verify when $V$ is not strongly generated by $V_1$. Since the space of coinvariants is the dual space of the space of conformal blocks, we believe the restricted conformal blocks would provide a similar surjective map onto $\mathbb{V}(V; W^\bullet)$, which might lead to the global generation property of this vector bundle for general VOAs.

\subsection{Outline of this paper}

This paper is organized as follows: 

In \cref{sec:preliminaries}, we introduce the geometric conventions adopted in this paper, review the theory of totally ramified orbicurves, vertex operator algebras, their twisted modules, twisted intertwining operators, twisted universal enveloping algebras, and twisted Zhu's algebras. 

In \cref{sec:tcL},  we review the vertex algebra bundle and the twisted chiral Lie algebra introduced in \cite{FBZ04, FS04} and study the \emph{constraints} of the chiral Lie algebra. In particular, we give an explicit description of these Lie algebras.

In \cref{sec:Cfb}, we review the notion of \emph{twisted conformal blocks} and study its relation with the \emph{$g$-twisted correlation functions} introduced in \cite{PartI} and with the \emph{twisted intertwining operators}. 

In \cref{sec:RestrictedCfb}, we introduce the notion of \emph{twisted restricted conformal blocks} and study its relation with the notions of \emph{$g$-twisted restricted correlation functions} and \emph{$3$-pointed $g$-twisted restrict conformal blocks} introduced in \cite{PartI}. 
We also study the restriction map from the space of \emph{twisted conformal blocks} to the space of \emph{twisted restricted conformal blocks}. 

In \cref{sec:coinvariants}, we introduce a geometric generalization of \emph{twisted Zhu's algebra} and describe the space of coinvariants by its modules. As corollaries, we give a general version of the \emph{twisted fusion rules theorem} and show that the space of coinvariants is finite-dimensional under some conditions.

\section{Preliminaries}\label{sec:preliminaries}
In this section, we introduce the conventions on algebraic geometry adopted in this paper and review the basic theory of \emph{totally ramified orbicurves} and \emph{twisted modules} of \emph{vertex operator algebras}.

\subsection{Conventions on algebraic geometry}
Throughout this paper, we will freely use facts about the geometry of smooth algebraic curves. 
For the general theory, we refer to various algebraic geometry textbooks such as \cite[Ch.~7]{liu2002algebraic} and \cite[Ch.~IV]{Hartshorne}, or consult \cite{fulton} for a traditional treatment and \cite{Schlag} for an analytic approach to Riemann surfaces.

\subsubsection*{Terminology and notations}
\begin{description}
  \item[curve] an integral scheme $(\X,\O_{\X})$ that is proper and is of dimension $1$ over $\C$. In this paper, we focus on \emph{smooth} curves. 
  \item[$\cK_{\X}$] the constant sheaf of the \textbf{field of rational functions} on $\X$. 
  \item[$\Omega_{\X}$] the sheaf of \textbf{differentials}. The \textbf{de Rham complex} of $\X$ is denoted by \[0\to\O_{\X}\overset{\d}{\longrightarrow}\Omega_{\X}\to0.\]
  \item[bundle] a locally free sheaf, \emph{but not necessarily of finite rank}.
  \item[point] a closed point, unless otherwise specified.
  \item[$\cI_\pp$] the \textbf{ideal sheaf} of the point $\pp$.
  \item[$\kappa_{\pp}$] the \textbf{residue field} $\O_{\X}/\cI_\pp$. 
  \item[$\O_{\pp}$] the \textbf{complete local ring at $\pp$}. It is the \emph{$\cI_\pp$-adic completion} of $\O_{\X}$, namely, the limit $\varprojlim\O_{\X}/\cI_\pp^n$. It is also the completion of the stalk of $\O_\X$ at $\pp$. 
  \item[$D_\pp$] $\Spec\O_{\pp}$, the \textbf{formal disk at $\pp$}, with a canonical morphism $\iota_{\pp}\colon D_\pp\to \X$. 
  \item[$\cK_{\pp}$] the fraction field of $\O_{\pp}$.
  \item[$D_{\pp}^{\times}$] $\Spec\cK_{\pp}$, the \textbf{punctured formal disk at $\pp$}, equipped with a canonical morphism $\iota_{\pp}\colon D_{\pp}^{\times}\monomorphism D_\pp\to \X$, where $D_{\pp}^{\times}$ is viewed as a subscheme of $D_{\pp}$.
  \item[local coordinate at $\pp$] a topological generator $z$ of $\O_{\pp}$ such that $\O_{\pp}\cong\C\dbrack{z}$.
  \item[$D_{z}$] the \textbf{abstract formal disk} $\Spec\C\dbrack{z}$.
  \item[$D_{z}^{\times}$] the \textbf{punctured abstract formal disk} $\Spec\C\dparen{z}$.
  \item[local chart] a pair $(\pp,z)$ of a point and a local coordinate at it. A local chart facilitates a morphism $\iota_z$ from $D_{z}$ to $\X$ through $D_{\pp}$.
  \item[$v_\pp$] the \emph{normalized} (i.e. $v_{\pp}(z)=1$) \textbf{valuation} on the DVR $\O_{\pp}$. This valuation extends to the function field $\cK_\X$ via the canonical map $\cK_\X\to\cK_{\pp}$.
  \item[$V\dbrack{z^{\pm1}}$] the vector space of formal power series $\sum\limits_{n\in\Z}c_nz^n$ with coefficients $c_n$ in a vector space $V$.
  \item[$\Res_{z}f(z)$] the \textbf{residue} of a series $f(z)\in V\dbrack{z^{\pm1}}$, i.e. the coefficient of $z^{-1}$. 
  For any rational function $f\in\cK_{\X}$ and any local chart $(\pp,z)$, the \textbf{residue} of the differential $f\d{z}$ at $\pp$ can be computed by $\Res_{\pp}f\d{z}=\Res_{z}\iota_{z}f$. 
  \item[$\cF|_{U}$] the \textbf{inverse image} $j^{\ast}\cF$ of a sheaf $\cF$ along the implied $j\colon U\to\X$.
  \item[$\Gamma(U,\cF)$] the space of \textbf{gloabl sections} of $\cF|_{U}$.
  \item[$\cF$ is trivialized on $U$ as $W$] $\cF|_U$ is a trivial bundle on $U$ with the space of global section (isomorphic via the trivialization to) $W$.
  \item[$\cF_{\pp}$] the \textbf{fiber of $\cF$ at $\pp$} (NOT the \textbf{stalk}). 
  It is the skyscraper sheaf $\cF|_{\pp}$ and also the $\kappa_{\pp}$-vector space $\Gamma(\pp,\cF)$. 
  \item[value at $\pp$] the image of a section $s$ under the canonical map $\cF\to\cF_{\pp}$.
  \item[local sheaf of $\cF$ at $\pp$] the {inverse image} $\cF|_{D_{\pp}}$ or $\cF|_{D_{\pp}^{\times}}$.
  \item[(regular) local sections of $\cF$ at $\pp$] elements of $\Gamma(D_{\pp},\cF)$.
  \item[rational local sections of $\cF$ at $\pp$] elements of $\Gamma(D_{\pp}^{\times},\cF)$.
  \item[$\iota_{\pp}$] the canonical map $\cF\to\iota_{\pp\ast}\iota_{\pp}^{\ast}\cF=\cF|_{D_{\pp}^{\times}}$.
  \item[local expansion of $s$ at $\pp$] the image of a section $s$ of $\cF$ under $\iota_{\pp}$.
  \item[$\iota_{z}$] the canonical map $\cF\to\iota_{z\ast}\iota_{z}^{\ast}\cF=\cF|_{D_{z}^{\times}}$.
  \item[formal expansion of $s$ at $(\pp,z)$] the image of a section $s$ of $\cF$ under $\iota_{z}$.
  \item[divisor] a \emph{Weil divisor} on $\X$, i.e. a $\Z$-linear combination of points on $\X$.
  \item[$\dv(f)$] the \textbf{divisor of a rational function} $f\in\cK_\X$. 
  It is $\sum\limits_{\pp\in\X}v_{\pp}(f)\pp$.
  \item[$\supp\Delta$] the \textbf{support of a divisor $\Delta$}. It is the set of points involved (i.e. has nonzero coefficient) in the linear combination $\Delta$. 
  \item[$\deg(\Delta)$] the \textbf{degree of a divisor $\Delta$}. It is the sum of coefficients in $\Delta$. 
  \item[$\O_\X(\Delta)$] the sheaf of rational functions $f$ \textbf{constrained by $\Delta$} in the sense that the divisor $\dv(f)+\Delta$ is \emph{effective}: its coefficients are non-negative.
  \item[$\cF(\Delta)$] $\cF\otimes\O_\X(\Delta)$.
  \item[$\cF(\infty\Delta)$] $\varinjlim_n\cF(n\Delta)$. It is the sheaf given by $U\mapsto\Gamma(U\setminus\supp\Delta,\cF)$.
  \item[rational sections of $\cF$ with possible poles along $\Delta$] sections of $\cF(\infty\Delta)$.
\end{description}

\subsubsection*{Conventions}
\begin{itemize}
  \item \emph{Calligraphic letters}, like $\cF,\cM,\cN,\dots$, denote sheaves. All tensor products of \emph{$\O_\X$-modules} are over $\O_\X$ otherwise specified.
  \item \emph{Fraktur letters}, like $\pp$ and $\qq$, denote points on $\X$.
  \item We do not distinguish a \emph{skyscraper sheaf} supported at a point $\pp$ on $ \X$ with its \emph{stalk} at $\pp$.
  \item We will consistently view a sheaf $\cF$ on $U$, where $U$ is equipped with a implied $j\colon U\to\X$, as a sheaf on $\X$ via the \emph{direct image} along $j$.
\end{itemize}

\subsection{Geometry of a totally ramified orbicurve}
The following is our main geometric object under consideration.
\begin{definition}\label{def:orbicurve}
  An \textbf{orbicurve} is a \emph{Deligne-Mumford stack} that is obtained as the quotient of a smooth algebraic curve by a finite automorphism group $G$. 
  We represent this stack as a \emph{ramified covering} $\x\colon\tilX\to\X$ with the \emph{Galois group} $G$ and the dense \emph{unramified locus} $\Xc$. 
  Sheaves on this orbicurve can be understood as $G$-equivariant sheaves on $\tilX$. 
  An orbicurve $\X$ is called \textbf{totally ramified} if each \emph{branch point} $\pp\in \X$ is \emph{totally ramified} in the sense that $\pp$ has a unique lifting in $\tilX$. It follows that the group $G$ has to be cyclic. In the rest of this paper, we assume that $G$ has order $T$.
\end{definition}

\begin{example}[\textbf{twisted projective line}]\label{eg:twistedP1}
  In \cite[\S 2.2]{PartI}, we have introduced an \emph{ad-hoc} construction of the \emph{twisted projective line} as a ramified covering $\x\colon\oC\to\PP^1$. Roughly speaking, it corresponds to the inclusion $\C[z]\subset\C[z^{\sfrac{1}{T}}]$, where the symbol $z^{\sfrac{1}{T}}$ satisfies $(z^{\sfrac{1}{T}})^T=z$. 
  The branch points on this covering are $0$ and $\infty$ on $\PP^1$, and the unramified locus is $\PP^1\setminus\Set{0,\infty}$. 
\end{example}
\begin{example}[\textbf{cyclic covering of the 3-pointed projective line}]\label{eg:3ptP1}
Consider the equation (where $w\in\C^\times$)
  \[
    y^T = x^a(x-w)^b,
  \]
  where $a$, $b$, and $a+b$ are all coprime to $T$. Then there is a connected smooth projective curve $\oC_w$ irrational to the projective curve obtained as the compactification of the affine curve defined by the above equation. 
  The projection $(x,y)\mapsto x$ induces a totally ramified covering $\x\colon\oC_w\to\PP^1$ with Galois group of order $T$ and branch points $0,w,\infty$. 
\end{example}
\begin{remark}\label{rem:3ptP1}
  The obicurves in \cref{eg:twistedP1,eg:3ptP1} forms a family $\mathfrak{C}\to\PP^1$, whose fiber at $w\neq0,\infty$ is $\oC_w$ while the fiber at $0$ and $\infty$ are $\oC$ and its reverse.
\end{remark}

\begin{example}[\textbf{trivially ramified orbicurve}]\label{eg:trivial}
  Our definition of an orbicurve allows the Galois group to be trivial. In this case, an orbicurve is the same as a smooth curve $\X$ equipped with a specified dense open $\Xc$. 
  We call points in $\X\setminus\Xc$ \textbf{branch points} even if they are actually unramified. 
  Such orbicurves are totally ramified.
\end{example}

\begin{example}[\textbf{cyclic covering of prime order}]\label{eg:primeorder}
  When $T$ is a prime number, the Galois group $G$ is simple, and thus the orbicurve $\x\colon\tilX\to\X$ is automatically totally ramified.
\end{example}

\begin{definition}\label{def:Deck}
  Let $\qq$ be a branch point. 
  Its stabilizer $G_\qq$ is generated by the \emph{Deck transformation} corresponding to the simple monodromy around $\qq$. We call this Deck transformation the \textbf{Deck generator at $\qq$} and denote it by $g_{\qq}$.
\end{definition}

\begin{definition}
  Let $\x\colon\tilX\to\X$ be a totally ramified orbicurve, $\qq\in\X$ a branch point, and $\tilde\qq$ the unique lifting of $\qq$.  
  We say that a local coordinate $t$ at $\tilde\qq$ is \textbf{special} if $t^T$ gives rise to a local coordinate $z$ at $\qq$. 
  If this is the case, we denote the coordinate $t$ by $z^{\sfrac{1}{T}}$ and call it a \textbf{special local coordinate above $\qq$}. 
\end{definition}

Then we have a commutative diagram of formal expansions:
\[
  \begin{tikzcd}
    {\Spec\C\dparen{z^{\sfrac{1}{T}}}} \ar[r, hook]\ar[d] & {\Spec\C\dbrack{z^{\sfrac{1}{T}}}} \ar[r, "\sim"', "\iota_{z^{\sfrac{1}{T}}}"]\ar[d] & {D_{\tilde\qq}} \ar[r,"\iota_{\tilde\qq}"]\ar[d, "\x"] & {\tilX} \ar[d, "\x"] \\
    {\Spec\C\dparen{z}} \ar[r, hook] & {\Spec\C\dbrack{z}} \ar[r, "\sim"', "\iota_{z}"] & {D_{\qq}} \ar[r,"\iota_{\qq}"] & {\X}  
  \end{tikzcd}
\]
By abuse of notations, we denote $\iota_{\tilde\qq}$ (resp. $\iota_{z^{\sfrac{1}{T}}}$) by $\iota_{\qq}$ (resp. $\iota_{z}$) when there is no ambiguities. We also extend the conventions on formal series in $z$ to formal series in $z^{\sfrac{1}{T}}$. In particular, the \textbf{residue} $\Res_zf(z)$ of a series $f(z)$ in $z^{\sfrac{1}{T}}$ is also the coefficient of $z^{-1}$. 
Then for any rational function $f\in\cK_{\tilX}$ and a branch point $\qq\in \X$, we have $\Res_{\qq}f\d{z}=\tfrac{1}{T}\Res_{\tilde\qq}f\d{z}=\Res_{z}\iota_{z}f$.

\subsubsection*{Equivariant bundles on a totally ramified orbicurve}
Let $\cF$ be a \emph{$G$-equivariant bundle} on $\tilX$. 
Since $G$ acts freely on $\Xc$, the bundle $\cF|_{\x^{-1}\Xc}$ descends to a bundle $\cF^G$ on $\Xc$. 
At each point $\pp\in\Xc$, we have 
\begin{equation*}
  \cF^{G}|_{D_{\pp}}=(\prod_{\tilde\pp\in\x^{-1}\pp}\cF|_{D_{\tilde\pp}})^{G}\cong\cF|_{D_{\tilde\pp}}.
\end{equation*}
By abuse of notation, \emph{we denote all these local sheaves by $\cF|_{D_{\pp}}$.}
Consequently, fibers of $\cF$ at all liftings $\tilde{\pp}$ of $\pp\in\Xc$ are isomorphic to the fiber of $\cF^G$ at $\pp$. 
By abuse of notation, \emph{we denote all these fibers $\cF_{\tilde\pp}$ and $\cF^G_{\pp}$ by $\cF_{\pp}$.}

Since $G$ acts freely on the orbit $\x^{-1}\pp$, any local coordinate $z$ at $\pp$ lifts to a local coordinate at any lifting $\tilde{\pp}$ of $\pp$. Then the aforementioned local sheaves share a common formal expansion map $\iota_{z}\colon\cF|_{D_{\pp}^\times}\to\cF|_{\Spec\C\dparen{z}}$.

The bundle $\cF^G$ further extends to a sheaf on the entire curve $\X$ via the pushforward functor along the inclusion $\Xc\monomorphism\X$. 
In particular, at a branch point $\qq$, we have 
$\cF^{G}|_{D_{\qq}^\times}=(\cF|_{D_{\tilde\qq}^\times})^{G}$, where $\tilde\qq$ is the (unique) lifting of $\qq$.

\begin{example}
  Consider the \emph{twisted projective line} $\x\colon\oC\to\PP^1$ in \cref{eg:twistedP1}.
  A bundle $\cV$ on $\oC$ is equivalent to the following data (where $V$ is a vector space):
  \begin{itemize}
    \item a trivial bundle $\cV_{C}=V\otimes_{\C}\O_{C}$ on the affine open $C:=\oC\setminus\Set{\infty}$,
    \item a trivial bundle $\cV_{C^\dagger}=V\otimes_{\C}\O_{C^\dagger}$ on the other affine open $C^\dagger:=\oC\setminus\Set{0}$,
    \item a \emph{transition map} $\vartheta\colon\cV_{C}|_{\Cc}\cong\cV_{C^\dagger}|_{\Cc}$ on the intersection $\Cc:=C\cap C^\dagger$.
  \end{itemize}
  Given a group representation $\rho\colon G\to\grp{GL}(V)$, we have an action of $G$ on the trivial bundle $\cV_{C}$:
  \begin{equation}\label{g-bun}
    g.(a\otimes f) = (\rho(g)a)\otimes(g^{-1})^{\ast}f,
  \end{equation}
  for all $g\in G$, $a\in V$, and sections $f$ of $\O_{C}$. 
  A similar action of $G$ exists on the other trivial bundle $\cV_{C^\dagger}$. 
  Then the representation $\rho$ gives rise to a $G$-equivalent structure on the bundle $\cV$ if and only if the restricted actions of $G$ on $\cV_{C}|_{\Cc}$ and $\cV_{C^\dagger}|_{\Cc}$ given by \labelcref{g-bun} commute with $\vartheta$. 

  Let $\cV$ be a $G$-equivalent bundle on $\oC$. 
  The coordinate $z^{\sfrac{1}{T}}$ (resp. $(z^\dagger)^{\sfrac{1}{T}}$) on the affine open $C$ (resp. $C^\dagger$) trivializes $\cV_C$ (resp. $\cV_{C^\dagger}$) as $V[z^{\sfrac{1}{T}}]$ (resp. $V[(z^\dagger)^{\sfrac{1}{T}}]$). 
  Then the bundle $\cV^G$ on the unramified locus $\x(\Cc)$ is trivialized as $V[z^{\pm1}]$ and $V[(z^\dagger)^{\pm1}]$ via these two trivializations respectively. 
  The two trivializations are related via the transition map $\vartheta$, extending the transition map $z^{\sfrac{1}{T}}\mapsto (z^\dagger)^{-\sfrac{1}{T}}$ of the structure sheaves. 
  
  At an unramified point $w\neq 0,\infty$, 
  the local sheaf $\cV|_{D_{w}^{\times}}$ is trivialized as $V\dparen{z-w}$ (or $V\dparen{z^\dagger-w^{-1}}$). 
  At the branch point $0$ (resp. $\infty$), the local sheaf $\cV^G|_{D_{0}^{\times}}$ (resp. $\cV^G|_{D_{\infty}^{\times}}$) is trivialized as $(V\dparen{z^{\sfrac{1}{T}}})^G$ (resp. $(V\dparen{(z^\dagger)^{\sfrac{1}{T}}})^G$)
\end{example}

\begin{example}
  The similar argument applies to \cref{eg:3ptP1}, where we have one extra branch point $w$ at which the local sheaf $\cV^G|_{D_{w}^{\times}}$ is trivialized as $(V\dparen{(z-w)}^{\sfrac{1}{T}})^G$ (or $(V\dparen{(z^\dagger-w^{-1})}^{\sfrac{1}{T}})^G$). 
\end{example}

\subsubsection*{Equivariant $k$-differentials}
The sheaf $\Omega_\X$ of differentials is also the \emph{dualizing sheaf} on a smooth curve $\X$. 
We write $\Omega_\X^{\otimes k}$ for the $k$-th tensor power of $\Omega_\X$. When $k$ is negative, $\Omega_\X^{\otimes k}$ is defined as the inverse of $\Omega_\X^{\otimes -k}$ under the tensor product.
Sections of $\Omega_\X^{\otimes k}$ are called \textbf{$k$-differentials} on $\X$. 

As a line bundle, $\Omega_\X^{\otimes k}$ inherits the \emph{normalized valuation} at each point $\pp$. 
Let $(\pp,z)$ be a local chart. Through the isomorphism: 
\[
  \iota_{z}\colon\Gamma\big(D_{\pp}^{\times},\Omega_\X^{\otimes k}\big)\cong\C\dparen{z}(\d{z})^k,
\]
we can identify the \textbf{valuation} $v_{\pp}(\mu)$ of a $k$-differential $\mu$ (at $\pp$) with the largest integer $m$ such that $\iota_{z}(\mu)\in z^m\C\dbrack{z}(\d{z})^k$. 
For a divisor $\Delta=\sum n_{\pp}\pp$ on $\X$, the line bundle $\Omega_\X^{\otimes k}(\Delta)$ can be characterized by
\[
  \Gamma\big(U,\Omega_\X^{\otimes k}(\Delta)\big)=
  \Set*{\mu\in\Gamma\big(U\setminus\supp\Delta,\Omega_\X^{\otimes k}\big)\given v_{\pp}(\mu)+n_{\pp}\ge0\txforall[\ ]\pp\in U}.
\]

For an orbicurve $\x\colon\tilX\to\X$, the line bundle $\Omega_{\tilX}^{\otimes k}$ is $G$-equivariant. We thus obtain a line bundle $(\Omega_{\tilX}^{\otimes k})^G$ on $\Xc$, which is precisely $\Omega_\X^{\otimes k}|_{\Xc}$. 
For an unramified point $\pp\in\Xc$, all the local sheaves $\Omega_{\tilX}^{\otimes k}|_{D_{\tilde\pp}^\times}$ (where $\tilde\pp$ is a lifting of $\pp$) and $\Omega_\X^{\otimes k}|_{D_{\pp}^\times}$ share a common formal expansion map $\iota_{z}$ for all local coordinate $z$ at $\pp$. In particular, $v_{\pp}=v_{\tilde\pp}$.  
At a branch point $\qq\in\X\setminus\Xc$, by choosing a special local coordinate $z^{\sfrac{1}{T}}$ above $\qq$, we have 
\[
  \iota_{z^{\sfrac{1}{T}}}\Gamma\big(D_{\qq}^{\times},(\Omega_{\tilX}^{\otimes k})^G\big)\cong\iota_{z}\Gamma\big(D_{\qq}^{\times},\Omega_{\X}^{\otimes k}\big)\cong
  \C\dparen{z}(\d{z})^k\subset\C\dparen{z^{\sfrac{1}{T}}}(\d{z})^k.
\] 
Note that the last space admits a $g^{\ast}$-eigenspace decomposition 
\[\C\dparen{z^{\sfrac{1}{T}}}(\d{z})^k=\bigoplus_{r=0}^{T-1}z^{\sfrac{r}{T}}\C\dparen{z}(\d{z})^k.\]
The valuation $v_{\qq}$ of $k$-differentials in $\Omega_{\X}^{\otimes k}$ extends to a valuation on $\Omega_{\tilX}^{\otimes k}$ with codomain $\frac{1}{T}\Z$. 
Under the special local coordinate $z^{\sfrac{1}{T}}$, $v_{\qq}(\mu)$ is the largest element $m\in\frac{1}{T}\Z$ such that $\iota_{z^{\sfrac{1}{T}}}(\mu)\in z^m\C\dbrack{z^{\sfrac{1}{T}}}(\d{z})^k$. 
Then we have $v_{\qq}=\frac{1}{T}v_{\tilde{\qq}}$, where $\tilde\qq$ is the lifting of $\qq$.
  
In the rest of this paper, we omit the subscript and denote the sheaf of differentials by $\Omega$ when there are no ambiguities.

\subsubsection*{Key lemmas}
The following lemmas will be used later. 

The first lemma is a consequence of the \emph{Riemann-Roch Theorem}, see, for instance, \cite{Z94,DGT19}.
\begin{lemma}\label{lem:RRcoro}
  Let $U$ be an affine open subset\footnote{For instance, let $\pp_\bullet$ be a finite family of points such that each irreducible component of $\X$ contains at least one of them, then $\X\setminus\pp_\bullet$ is affine.} of a curve $\X$ and let $\qq_{\diamond}$ be an $n$-tuple of points in $U$. 
  For each $1\le i\le n$, let $z_i$ be a local coordinate at $\qq_{i}$.
  Let $k$ be an integer. 
  Then for all integers $d$ and $m$, there exists a $k$-differential $\mu\in\Gamma\big(U\setminus\qq_{\diamond},\Omega^{\otimes k}\big)$ such that 
  \[
    \begin{array}{ll}
      \iota_{z_i}(\mu) \equiv z_i^d(\d{z_i})^k &\mod z_i^m\C\dbrack{z_i}(\d{z_i})^k, \text{ for a fixed $i$, and} \\
      \iota_{z_j}(\mu) \equiv 0 &\mod z_j^m\C\dbrack{z_j}(\d{z_j})^k, \text{ for all }j\neq i. \\
    \end{array}
  \]
\end{lemma}

The second lemma can be found in, for instance, \cite[Remark 9.2.10]{FBZ04} and was originally proved by Tate \cite{Tate}.
\begin{lemma}[Strong Residue Theorem]\label{lem:SRT}
  Let $\cF$ be a bundle on a curve $\X$, $\pp_{\bullet}$ an $n$-tuple of distinct points, and $\tau\in\bigoplus_{i=1}^{n}\Gamma(D_{\pp_i}^{\times},\cF^{\ast})$ a local section. Then
  \[ \sum_{i=1}^{n}\Res_{\pp_i}\braket*{\tau}{\alpha} = 0,\txforall\alpha\in\Gamma(\X\setminus\pp_{\bullet},\cF\otimes\Omega),
  \]
  if and only if $\tau$ can be extended to a regular section of $\cF^{\ast}$ on $\X\setminus\pp_{\bullet}$.
\end{lemma}

\begin{example}\label{eg:RSF}
  As a special case of the \nameref{lem:SRT}, we have a \emph{Residue Sum Formula}  for the \emph{twisted projective line} $\x\colon\oC\to\PP^1$ (see \cite[Lemma 2.18]{PartI}): 
  for any meromorphic $1$-form $\alpha$ on $\PP^1$, we have 
  \[
    \tfrac{1}{T}\Res_{0}\x^{\ast}\alpha+\tfrac{1}{T}\Res_{\infty}\x^{\ast}\alpha+\sum_{i=1}^{n}\Res_{\pp_i}\x^{\ast}\alpha=0,
  \] 
  where the family of points $\pp_{\bullet}$ on $\oC$  consists of liftings of the poles of $\alpha$, one lifting for each pole.

  A similar formula\footnote{Here, $\Res_w$ means the residue at the point $w\in\PP^1$, we will never treat $w$ as a formal variable when taking residue.} holds for \cref{eg:3ptP1}:
  \[
    \tfrac{1}{T}\Res_{0}\x^{\ast}\alpha+\tfrac{1}{T}\Res_{w}\x^{\ast}\alpha+\tfrac{1}{T}\Res_{\infty}\x^{\ast}\alpha+\sum_{i=1}^{n}\Res_{\pp_i}\x^{\ast}\alpha=0,
  \]
  where the family of points $\pp_{\bullet}$ on $\oC$  consists of liftings of the poles of $\alpha$, one lifting for each pole.
\end{example}

\subsubsection*{Flat connections}
We finish this subsection with a definition.
\begin{definition}
  Let $\cV$ be a bundle on a curve $\X$. A \textbf{(flat) connection} on $\cV$ is a $\C$-linear map
  $\nabla\colon\cV\to\cV\otimes\Omega$ satisfying the \emph{Leibniz rule}:
  \[
    \nabla(f\tau) = f\nabla(\tau) + \tau\otimes\d{f},
  \]
  for any regular function $f$ on $\X$ and any section $\tau$ of $\cV$. 
  The pair $(\cV,\nabla)$ gives rise to an example of \emph{left $\cD$-modules} on $\X$. 
  Its \textbf{de Rham complex} is 
  \[
    0\longrightarrow\cV\overset{\nabla}{\longrightarrow}
    \cV\otimes\Omega\longrightarrow0,
  \]
  where $\cV\otimes\Omega$ (which is a \emph{right $\cD$-module}) is placed in the cohomology degree $0$, following the convention in \cite{FBZ04}.
  Sections of the $-1$-th {de Rham cohomology}, namely $\Ker\nabla$, are called \textbf{horizontal sections} of $\cV$.  
\end{definition}

\subsection{Vertex operator algebras and their twisted modules}
In this paper, we concentrate on \emph{non-negatively graded} VOAs.
\begin{definition}[\cite{B,FBZ04,FHL,FLM,LL,Z,Z94}]\label{def:VOA}
  A (non-negatively graded) \textbf{vertex operator algebra} (\emph{VOA} for short) is a quadruple $(V, Y(\cdot,z), \vac, \upomega)$, throughout simply denoted by $V$, where 
  \begin{itemize}
    \item $V=\bigoplus_{k\in\N}V_k$ is a $\N$-graded vector space with $\dim{V_k}<\infty$; 
    \item $Y(\placeholder,z)\colon V\to \End{V}\dbrack{z^{\pm1}}$ is a linear map assigning to each $a\in V$ the \textbf{vertex operator} \[Y(a,z)=\sum\limits_{n\in\Z}\vo{a}{n}z^{-n-1};\]
    \item $\vac\in V_0$ and $\upomega\in V_2$ are two distinguished vectors, called the \textbf{vacuum vector} and the \textbf{conformal vector} respectively.
  \end{itemize}
  These data must satisfy the following axioms:
  \begin{enumerate}[label={\textbf{V}\arabic*}]
    \item \textbf{\small(Truncation property)} $\vo{a}{n}b=0$ for all $a,b \in V$ and sufficiently large $n$;
    \item \textbf{\small(Vacuum property)} $Y(\vac, z)=\id_V$;
    \item \textbf{\small(Creation property)} $Y(a, z){\vac}\in V\dbrack{z}$ and $\lim\limits_{z\to 0}Y(a, z){\vac}=a$ for $a\in V$;
    \item \textbf{\small(Jacobi identity)} for all $a,b \in V$,
    \begin{multline*}
      \Res_{z}(Y(a,z)Y(b,w)\iota_{z^{-1}}F(z,w))
      -\Res_{z}(Y(b,w)Y(a,z)\iota_{z}F(z,w))\\
      =\Res_{z-w}(Y(Y(a,z-w)b,w)\iota_{z-w}F(z,w))
    \end{multline*}
    holds for every rational function $F(z,w)\in\C[z^{\pm1},(z-w)^{\pm1}]$;
    \item \textbf{\small(Virasoro relations)} defining $\vo{L}{n}$ ($n\in\Z$) as the coefficients of $Y(\upomega, z)$: 
    \[
      Y(\upomega, z)=\sum_{n\in \mathbb{Z}}\vo{\upomega}{n}z^{-n-1}=\sum_{n\in \mathbb{Z}}\vo{L}{n}z^{-n-2},
    \]
    then there is a complex number $c$, called the \textbf{central charge}, such that
    \[
      \Liebracket*{\vo{L}{m}}{\vo{L}{n}} = (m-n)\vo{L}{m+n} + \frac{c}{12}\delta_{m+n,0}(m^3-m)\id_V;
    \]
    \item \textbf{\small($\vo{L}{0}$-eigenspace decomposition)} $\vo{L}{0}a=na$ for all homogeneous $a\in V_n$. That is to say, the \textbf{weight} of $a$, denoted by $\wt a$, is the corresponding eigenvalue of $\vo{L}{0}$; 
    \item \textbf{\small($\vo{L}{-1}$-derivative property)} $Y(\vo{L}{-1}a, z)=\partial_{z} Y(a, z)$ for all $a\in V$.
  \end{enumerate}
\end{definition}
\begin{remark}\label{rem:VirasoroAction}
  The Lie algebra $\grp{Der}_0\O$ of the algebraic group of \emph{changes of local coordinates} $\grp{Aut}\O\colon R\leadsto \Aut(R\dbrack{z})$ reads $R\leadsto R\dbrack{z}z\partial_{z}$. It is a subalgebra of the \emph{Virasoro Lie algebra} $\mathsf{Vir}$. 
  Then the axioms encodes an action of $\grp{Der}_0\O$ on $V$ via $z^{p+1}\partial_{z}\mapsto\vo{L}{p}$ ($p\ge 0$). 
  Furthermore, this action is integrable, and we thus have an action of $\grp{Aut}\O$ on $V$.
\end{remark}
\begin{definition}\label{def:automorphism-of-VOA}
  An automorphism of a vertex operator algebra $V$ is a linear transformation $g\in\grp{GL}(V)$ preserving $\vac$ and $\upomega$ and satisfying
  \[
    gY(a, z)g^{-1}=Y(g.a, z)
    \txforall a\in V.
  \]
\end{definition}

Let $V$ be a VOA and $g$ an automorphism of $V$ satisfying $g^T=\id_V$. 
Then the underlying vector space $V$ is decomposed into $g$-eigenspaces
\[
  V_g^r=\Set*{ a\in V \given g.a=e^{2\pi\iu\frac{r}{T}}a },
\]  
where \emph{$r$ is throughout assumed to be an integer between $0$ and $T-1$.}
\begin{definition}[\cite{DLM1}]\label{def:twistedmodule}
  A \textbf{weak $g$-twisted $V$-module} is a vector space $M$ equipped with a linear map 
  $Y_{M}(\placeholder,z)\colon V\to \End{M}\dbrack{z^{\pm\sfrac{1}{T}}}$ assigning to each $a\in V$ a \textbf{twisted vertex operator}
  \[
    Y_{M}(a,z) = \sum_{n\in\frac{1}{T}\Z} \vo{a}{n} z^{-n-1},
  \]
  satisfying the following axioms: 
 \begin{enumerate}[label={\textbf{M}\arabic*}]
    \item \textbf{\small(Truncation property)} $\vo{a}{n}u=0$ for all $a\in V$, $u\in M$, and sufficiently large $n\in\frac{1}{T}\Z$;
    \item \textbf{\small(Vacuum property)} $Y_{M}(\vac,z)=\id_{M}$.
    \item \textbf{\small(Index property)} $Y_M(a,z)z^{\sfrac{r}{T}}\in\End{M}\dbrack{z^{\pm1}}$ for all $a\in V_g^r$. 
    \item \textbf{\small(Twisted Jacobi identity)} for all $a,b\in V$,
    \begin{multline*}
      \Res_{z}(Y_M(a,z)Y_M(b,w)\iota_{z^{-1}}F(z,w))
      -\Res_{z}(Y_M(b,w)Y_M(a,z)\iota_{z}F(z,w))\\
      =\Res_{z-w}(Y_M(Y(a,z-w)b,w)\iota_{z-w}F(z,w))
    \end{multline*}
    holds for every rational function\footnote{They are the rational functions on the curve $\oC$ in \cref{eg:twistedP1} with possible poles at $0$, $\infty$, and $w$.} $F(z,w)\in\C[z^{\pm\sfrac{1}{T}},(z-w)^{-1}]$.
  \end{enumerate}  
\end{definition}
\begin{definition}\label{def:admissible}
  A weak $g$-twisted $V$-module $M$ is called an \textbf{admissible $g$-twisted $V$-module} if it admits a subspace decomposition $M=\bigoplus_{n\in \frac{1}{T}\N}M(n)$ such that 
  \begin{equation*}
    \vo{a}{m}M(n)\subset M(\wt a-m-1+n)
  \end{equation*}
  for any homogeneous $a\in V$, any $m\in \Z$, and any $n\in \frac{1}{T}\N$.
\end{definition}
\begin{definition}
  An admissible $g$-twisted module $M$ is said to be \textbf{of conformal weight $h\in\C$}, if  $\vo{L}{0}$ acts semi-simply on $M$ and each $M(n)$ is the eigenspace $M_{h+n}$ of eigenvalue $h+n$.
\end{definition}

Let $g_1,g_2,g_3$ be automorphisms of a VOA $V$ satisfying\footnote{We make this assumption since if there is a nonzero intertwining operator among $g_1$-twisted module $M^1$, $g_2$-twisted module $M^2$, and $g_3$-twisted module $M^3$, then essentially $g_3=g_1g_2$, see \cite{X95, H18}.}  $g_1g_2=g_3$. Suppose one of the following conditions holds\footnote{These conditions come from the requirement for the involved obicurve being totally ramified. Note that, when $T$ is a prime, one of the two conditions has to be held.}:
\begin{enumerate}
    \item $g_1=\id$; 
    \item each of $g_1$, $g_2$, and $g_3$ generates the same cyclic group (of order $T$).
\end{enumerate}
Then $g_1$, $g_2$, and $g_3$ share a common eigenspace decomposition 
\[
  V_{g_2}^r=\Set*{ a\in V \given g_2.a=e^{2\pi\iu\frac{r}{T}}a }.
\]  
Indeed, let $g_1=g_2^\epsilon$ for some integer $\epsilon$, then we have $V_{g_2}^r=V_{g_1}^{r\epsilon}=V_{g_3}^{r(\epsilon+1)}$. 

\begin{definition}[\cite{X95, H18}]
  Let $M^i$ be a $g_i$-twisted admissible $V$-module of conformal weight $h_i$ for $i=1, 2, 3$. A \concept{twisted intertwining operator of type $\fusion$} is a linear map $I(\cdot, w)\colon M^1 \to w^{-h}\Hom(M^2, M^3)\dbrack{w^{\pm\sfrac{1}{T}}}$ assigning to each $v\in M^1$ an formal series
  \[
    I(v,w)=\sum_{n\in \frac{1}{T}\Z} \vo{v}{n} w^{-n-h-1},
  \] 
  where $h=h_1+h_2-h_3$, satisfying the following axioms:
  \begin{enumerate}[label={\textbf{I}\arabic*}]
    \item \textbf{\small(Truncation property)}\label{axiom:IO-trunc} 
    $\vo{v}{m}=0$ for all $v\in M^1$, $ v_2\in M^2$, and sufficiently large $n\in\frac{1}{T}\Z$;
    \item \textbf{\small(Twisted Jacobi identity)}\label{axiom:IO-Jac} for all $a\in V$, $v\in M^1$,
    \begin{multline*}
      \Res_{z}(Y_{M^3}(a,z)I(v,w)\iota_{z^{-1}}F(z,w))
      -\Res_{z}(I(v,w)Y_{M^2}(a,z)\iota_{z}F(z,w))\\
      =\Res_{z-w}(I(Y_{M^1}(a,z-w)v,w)\iota_{z-w}F(z,w))
    \end{multline*}
    holds for every rational function\footnote{They are the rational functions on either the curve $\oC$ in \cref{eg:twistedP1} or the curve $\oC_w$ in \cref{rem:3ptP1} with possible poles at $0$, $w$, and $\infty$.}  $F(z,w)\in\C[z^{\pm\sfrac{1}{T}},(z-w)^{\pm\sfrac{1}{T}}]$.
    \item \textbf{\small($\vo{L}{-1}$-derivative property)}\label{axiom:IO-der} $Y(\vo{L}{-1}v, w)=\partial_{w} I(v, w)$ for all $v\in M^1$.
  \end{enumerate}
  The space of all intertwining operators of type $\fusion$ is denoted as $\Fusion$.
\end{definition}

\subsection{Algebras associated to a vertex operator algebra}\label{sec:assoalg}

\subsubsection*{Ancillary Lie algebras}
Following \cite{DLM1,B}, there is a Lie algebra $\L_g(V)$ associated to a VOA $V$ and an automorphism $g\in \Aut(V)$ with order $T$. The definition of $\L_g(V)$ is given as follows: 

First, we extend $g$ to an automorphism of  $V\otimes\C[z^{\pm\sfrac{1}{T}}]$ by letting 
\begin{equation*}
  g(a\otimes z^\frac{m}{T}):=e^{-2\pi\iu\frac{m}{T}}(ga\otimes z^\frac{m}{T}).
\end{equation*}
Then the fixed-point subspace has a decomposition $(V\otimes\C[z^{\pm\sfrac{1}{T}}])^g=\bigoplus_{r=0}^{T-1}V_g^r\otimes z^{\frac{r}{T}}\C[z^{\pm1}]$. Let $\nabla$ be the operator $\vo{L}{-1}\otimes\id + \id\otimes\partial_{z}$. The underlying vector space of $\L_g(V)$ is defined by the quotient 
\[
  \L_g(V) := (V\otimes\C[z^{\pm\sfrac{1}{T}}])^g/\Image\nabla = 
  \bigoplus_{r=0}^{T-1}V_g^r\otimes z^{\frac{r}{T}}\C[z^{\pm1}]\d{z}/\Image\nabla.
\]
For any $m\in \Z$ and $a\in V$, we denote the equivalent class of $a\otimes z^\frac{m}{T}$ in $\L_g(V)$ by $\lo{a}{\frac{m}{T}}$. 
The Lie algebra structure on $\L_g(V)$ is given by the following formula: 
\[
  \Liebracket*{\lo{a}{m+\frac{r}{T}}}{\lo{b}{n+\frac{s}{T}}}
  =\sum_{j\ge 0}\binom{m+\frac{r}{T}}{j} \lo{(\vo{a}{j}b)}{m+n+\frac{r+s}{T}-j}.
\]

Moreover, $\L_g(V)$ has a natural gradation given by $\deg (\lo{a}{\frac{m}{T}}):=\wt a-\frac{m}{T}-1$, 
where $m\in\Z$ and $a$ is a homogeneous element of $V$. 
Let $\L_g(V)_n$ be the subspace of $\L_g(V)$ spanned by elements of degree $n\in \frac{1}{T}\Z$. 
Then we have a triangular decomposition:
\begin{equation*}
  \L_g(V)=\L_g(V)_{<0}\oplus \L_g(V)_0\oplus \L_g(V)_{>0}.
\end{equation*}

The above constructions extend to their $z^{\sfrac{1}{T}}$-adic completions (resp. $z^{-\sfrac{1}{T}}$-adic completions), resulting a topologically complete Lie algebra denoted by $\L_g(V)^{\sfL}$ (resp. $\L_g(V)^{\sfR}$). 
The following definition generalizes the corresponding one in \cite{DGK23}: 

\begin{definition}
  The Lie algebras $\L_g(V)^{\sfL}$, $\L_g(V)$, and $\L_g(V)^{\sfR}$ are called the \textbf{left}, \textbf{finite}, and \textbf{right} \textbf{$g$-twisted ancillary Lie algebras} of $V$ respectively. 
  When $g=\id$, they recover the notions of \textbf{ancillary Lie algebras} in \cite[\S 2.3]{DGK23}.
\end{definition}

\subsubsection*{Universal enveloping algebras}

Following the untwisted constructions in \cite[Section 2.4, Appendix A]{DGK23}, we briefly discuss the $g$-twisted version of the \emph{left}, \emph{finite}, and \emph{right} universal enveloping algebras of the VOA $V$. First, the universal enveloping algebras $\sfU_g^{\sfL}=\sfU(\L_g(V)^{\sfL})$, $\sfU_g=\sfU(\L_g(V))$, and $\sfU_g^{\sfR}=\sfU(\L_g(V)^{\sfR})$ of the Lie algebras $\L_g(V)^{\sfL}$, $\L_g(V)$, and $\L_g(V)^{\sfR}$, respectively, give rise to a \emph{good triple $(\sfU_g^{\sfL},\sfU_g,\sfU_g^{\sfR})$ of associative algebras} in the sense of \cite[Definition A.9.1]{DGK23}. 
Moreover, this triple of associative algebras admits \emph{good seminorms} induced by the gradation on $\L_g(V)$ and the natural filtrations on $\L_g(V)^{\sfL}$ and $\L_g(V)^{\sfR}$. 

Then by \cite[Corollary A.9.9]{DGK23}, the completion $(\widehat{\sfU}_g^{\sfL},\widehat{\sfU}_g,\widehat{\sfU}_g^{\sfR})$ of $(\sfU_g^{\sfL},\sfU_g,\sfU_g^{\sfR})$ with respect to the good seminorms is a good triple of associative algebras with \emph{complete} good seminorms. 
Consider the \textbf{$g$-twisted Jacobi relations}:
\begin{equation}\label{eq:JacobiRel}
  \begin{multlined}
    \sum_{i\ge 0} 
    \binom{l}{i}(-1)^i
    \left(
      \lo{a}{\frac{r}{T}+m+l-i}\lo{b}{\frac{s}{T}+n+i}-
      (-1)^{l}\lo{b}{\frac{s}{T}+n+l-i}\lo{a}{\frac{r}{T}+m+i}
    \right)\\
    =\sum_{j\ge 0}
    \binom{m+\frac{r}{T}}{j}
    \lo{(\vo{a}{l+j}b)}{\frac{r+s}{T}+m+n-j},
  \end{multlined}
\end{equation}
for all $a\in V_g^r$, $b\in V_g^s$, and $l,m,n\in\Z$. 
Let $\sfJ_g$ be the ideal of $\widehat{\sfU}_g$ generated by the twisted Jacobi relations \labelcref{eq:JacobiRel} (we need to take the completion first since the $g$-twisted Jacobi relations involve infinite many terms), and let $\sfJ_g^{\sfL}$ and $\sfJ_g^{\sfR}$ be the ideals of $\widehat{\sfU}_g^{\sfL}$ and $\widehat{\sfU}_g^{\sfR}$ generated by $\sfJ_g$. 
Let $\bar{\sfJ_g}^{\sfL}$, $\bar{\sfJ_g}$, and $\bar{\sfJ_g}^{\sfR}$ be the closures with respect to the seminorm of the corresponding ideals in $\widehat{\sfU}_g^{\sfL}$, $\widehat{\sfU}$, and $\widehat{\sfU}_g^{\sfR}$, respectively. Then $(\bar{\sfJ_g}^{\sfL},\bar{\sfJ_g}, \bar{\sfJ_g}^{\sfR})$ is a good triple of ideals by \cite[Lemmas A.9.5 and A.9.6]{DGK23}. 
Finally, the quotients $\U_g(V)^{\sfL}=\widehat{\sfU}^{\sfL}/\bar{\sfJ_g}^{\sfL}$, $\U_g(V)=\widehat{\sfU}/\bar{\sfJ_g}$, and $\U_g(V)^{\sfR}=\widehat{\sfU}^{\sfR}/\bar{\sfJ_g}^{\sfR}$ form a good triple of associative algebras with complete good seminorms by \cite[Corollary A.9.9]{DGK23}. 
\begin{definition}\label{def:UgV}
  The associative algebras $\U_g(V)^{\sfL}$, $\U_g(V)$, and $\U_g(V)^{\sfR}$ are called the \textbf{left}, \textbf{finite}, and \textbf{right} \textbf{$g$-twisted universal enveloping algebras of $V$} respectively. When $g=1$, we recover the notions of \textbf{universal enveloping algebras} in \cite[Definition 2.4.2]{DGK23}.
\end{definition}
\begin{remark}
  The construction of each of the $g$-twisted universal enveloping algebras $\U_g(V)^{\sfL}$, $\U_g(V)$, and $\U_g(V)^{\sfR}$ is independent of the others. 
  The purpose of the notion of \emph{good triples} is to conduct these constructions simultaneously. 
\end{remark}

\subsubsection*{Zhu's algebra}
Let $\U_g(V)_{0}$ be the (coincided) degree-zero subalgebra of the good triple of algebras $(\U_g(V)^{\sfL},\U_g(V),\U_g(V)^{\sfR})$. 
By quotienting out the first neighborhood of the identity in $\U_g(V)_{0}$, we obtain an associative algebra $\cA_g(V)$ \cite{He17,HY12,Han20,DGK23}. 
\begin{definition}
  The associative algebra $\cA_g(V)$ is called the \textbf{$g$-twisted Zhu's algebra} of $V$. When $g=1$, we recover the notion of \textbf{Zhu's algebra} in \cite[Definition 2.5.1]{DGK23}. 
\end{definition}

In \cite{DLM1}, the $g$-twisted Zhu's algebra $A_g(V)$ is defined as the quotient of $V$ modulo the subspace $O_g(V)$, where  
\[
  O_g(V)=\spn\Set*{
    a\circ_g b:=\Res_z \frac{(1+z)^{\wt a-1+\delta(r)+\frac{r}{T}}}{z^{1+\delta(r)}} Y(a,z)b
    \given
    a\in V_g^r,b\in V
  },
\] 
and 
$
  \delta(r)=
  \begin{cases*}
    1 & if $r=0$,\\
    0 & otherwise.
  \end{cases*}
$

By \cite[Lemma 2.1]{DLM1}, $V_g^r\subseteq O_g(V)$ for $r\neq 0$. Hence $A_g(V)$ is a quotient space of $V_g^0$. 
Let $[a]$ denote the image of $a\in V$ in $A_g(V)$. The multiplication on $A_g(V)$ is defined by the formula
\begin{equation}\label{def:g-star-product}
  [a]\ast_g [b]:=
  \begin{cases*}
    [\Res_z Y(a,z)b\frac{(1+z)^{\wt a}}{z}]& 
    if $a\in V_g^0$,\\
    0& otherwise.
  \end{cases*}
\end{equation}
Moreover, $[\vac]$ is the identity of $A_g(V)$, and $[\upomega]$ lies in its center. 

It was proved in \cite{He17,HY12,Han20} that the two associative algebras $\cA_g(V)$ and $A_g(V)$ are isomorphic by identifying $[a]\in A_g(V)$ with $\lo{a}{\wt{a}-1}\in\cA_g(V)$ for every homogeneous $a\in V$. 
In particular, $A_g(V)_{\Lie}$ is a Lie algebra quotient of the degree-zero subalgebra $\L_g(V)_{0}$ of the ancillary Lie algebra of $V$.

\subsection{Actions on a $g$-twisted $V$-module}

Let $M$ be a weak $g$-twisted $V$-module. Both the Lie algebra $\L_g(V)$ and the associative algebra $\U_g(V)$ act on $M$ via the Lie algebra homomorphism 
\[\L_g(V)\to \End(M),\ \lo{a}{\frac{m}{T}}\mapsto \vo{a}{\frac{m}{T}}:=\Res_{z}Y_{M}(a,z)z^{\frac{m}{T}},\quad a\in V, m\in \Z.\]
Moreover, following a similar argument as in \cite[\S 5.1.6]{FBZ04}, we can show that the category of weak $g$-twisted $V$-modules is equivalent to the category of \emph{smooth} $\U_g(V)$-modules. In particular, admissible $g$-twisted $V$-modules correspond to graded smooth $\U_g(V)$-modules. Furthermore, we have
\begin{lemma}[{\cite[Proposition 5.4]{DLM1}}]
Let $M$ be an irreducible admissible $g$-twisted $V$-module. Then the bottom level $M(0)$ is an irreducible $\cA_g(V)$-module.
\end{lemma}

Conversely, given any $\cA_g(V)$-module $U$, one can define an associated graded smooth $\U_g(V)$-module (and hence an admissible $g$-twisted $V$-module) $\sfM(U)$, called the \textbf{$g$-twisted generalized Verma module}, by the induction 
\begin{equation}\label{def:verma}
  \sfM(U):=\U_g(V)\otimes_{\U_g(V)_{\le 0}}U,
\end{equation}
where the $\U_g(V)_{\le 0}$-module action on $U$ is given by the trivial action of $\L_g(V)_{<0}$, and the action through the homomorphism $\L_g(V)_{0}\epimorphism\cA_g(V)_{\Lie}$.

\begin{remark}
Our definition \labelcref{def:verma} of the $g$-twisted generalized Verma module $\sfM(U)$ is the twisted analogy of the \emph{left generalized Verma module} $\Phi^{\sfL}(U)$ in \cite[Definition 3.1.1]{DGK23}. 
Furthermore, $\sfM(U)$ is also isomorphic to the $g$-twisted generalized Verma module $\bar{M}(U)$ constructed in \cite{DLM1} since they satisfy the same universal property. 
 \end{remark}

\begin{lemma}[{\cite[Theorem 6.3 and Lemma 7.1]{DLM1}}]
  Given an $\cA_g(V)$-module $U$, there is a unique maximal graded $\U_g(V)$-submodule $\sfM_+(U)$ of $\sfM(U)$ such that $\sfM_+(U)\cap U=0$. Furthermore, if $U$ is an irreducible $\cA_g(V)$-module, then $\sfM(U)/\sfM_+(U)$ is an irreducible $V$-module.
\end{lemma}

\begin{definition} \label{def:lowest-weight}
  A weak $g$-twisted module $M$ is called a \textbf{lowest-weight module} if there exists an eigenvalue $h\in\C$ of $\vo{L}{0}$ such that its eigenspace $M_h$ is an irreducible $\L_g(V)_{0}$-module, and $M$ is a quotient of the generalized Verma module $\sfM(M_h)$. 
\end{definition}

\subsubsection*{Contragredient modules}
There is an anti-isomorphism between the Lie algebras $\L_g(V)$ and $\L_{g^{-1}}(V)$ given by the formula
\begin{equation}\label{eq:def:theta}
  \theta(\lo{a}{\frac{m}{T}}):=\sum_{j\ge 0} \frac{(-1)^{\wt a}}{j!}
  \lo{(\vo{L}{1}^ja)}{2\wt a-\frac{m}{T}-j-2}.
\end{equation}
This anti-isomorphism extends to the good triple $(\L_g(V),\L_g(V)^{\sfL}, \L_g(V)^{\sfR})$ of ancillary Lie algebras and the good triple $(\U_g(V)^{\sfL},\U_g(V),\U_g(V)^{\sfR})$ of universal enveloping algebras. 
In particular, we have an anti-isomorphism between $\U_g(V)$ and $\U_{g^{-1}}(V)$, and isomorphisms 
\[
  \U_g(V)^{\sfL}\cong\U_{g^{-1}}(V)^{\sfR}
  \txand
  \U_g(V)^{\sfR}\cong\U_{g^{-1}}(V)^{\sfL}.
\] 

Let $M$ be an admissible $g$-twisted $V$-module. 
As a smooth \emph{left} $\U_{g}(V)$-module, its \emph{graded dual space} $M'=\bigoplus_{n\in \frac{1}{T}\N}M(n)^{\ast}$ naturally carries a smooth \emph{right} $\U_{g}(V)$-module structure, hence a smooth \emph{left} $\U_{g^{-1}}(V)$-module structure via the anti-isomorphism $\theta$. Thus it is a $g^{-1}$-twisted admissible $V$-module. More precisely, for any $a\in V$, $u'\in M'$, and $u\in M$, we have
\begin{equation}\label{eq:theta=dual}
  \braket*{\lo{a}{\frac{m}{T}}.u'}{u} = 
  \braket*{u'}{\theta(\lo{a}{\frac{m}{T}}).u}.
\end{equation}

\begin{definition}[\cite{FHL,X95}]
  The admissible $g^{-1}$-twisted $V$-module $M'$ is called the \textbf{contragredient module} of $M$.
\end{definition}

\begin{warn}\label{warn}
In general, an admissible $g$-twisted $V$-module $M$ needs NOT to have finite-dimensional components. 
  Consequently, its \emph{double} contragredient module $M''$ is not necessarily equal to $M$ itself.
  Thus, in general, an admissible $g^{-1}$-twisted $V$-module is not guaranteed to be the contragredient module of an admissible $g$-twisted $V$-module.
\end{warn}

\section{The twisted chiral Lie algebra}\label{sec:tcL}
In this section, we introduce the twisted chiral Lie algebras $\L_U(\cV^G)$ and their constrained variants.

\subsection{The twisted vertex algebra bundle and its connection}
Following \cite[\S 6.5]{FBZ04}, a VOA $V$ gives rise to an
\emph{bundle} $\cV$ on a curve $\X$, which is locally constructed from $V$ twisted by the choices of local coordinates\footnote{More precisely, assigning $\rho\in\grp{Aut}\O$ to the composition $\iota_{\rho(z)}\circ\iota_{z}^{-1}\colon V\to\cV_{\pp}\to V$ gives rise to an action of $\grp{Aut}\O$ on $V$, and this action coincides with the one in \cref{rem:VirasoroAction}.}. 
In particular, a choice of a local chart $(\pp,z)$ gives rise to a local trivialization of the bundle as $\Gamma(D_{\pp},\cV)\cong V\dbrack{z}$. 
\begin{example}
  On the \emph{twisted projective line} $\x\colon\oC\to\PP^1$,
  the bundle $\cV$ is trivialized on the affine charts $(C,z^{\sfrac{1}{T}})$ and $(C^\dagger,(z^\dagger)^{\sfrac{1}{T}})$ as
  \[
    \Gamma(C,\cV)\cong V[z^{\sfrac{1}{T}}],\txand 
    \Gamma(C^\dagger,\cV)\cong V[(z^\dagger)^{\sfrac{1}{T}}],
  \]
  with the transition map $\vartheta\colon V[z^{\sfrac{1}{T}}]\to V[(z^\dagger)^{\sfrac{1}{T}}]$ given by
  \begin{equation}\label{eq:def:vartheta}
    \begin{aligned}
      \vartheta(a\otimes z^{\frac{m}{T}})
      &:=\left.e^{z\vo{L}{1}}(-z^{-2})^{\vo{L}{0}}(a\otimes z^{\frac{m}{T}})\right|_{z^{\sfrac{1}{T}}\mapsto(z^\dagger)^{-\sfrac{1}{T}}}\\
      &=\sum_{j=0}^{\wt a} \frac{(-1)^{\wt a}}{j!}
      (\vo{L}{1}^ja)\otimes (z^\dagger)^{2\wt a-\frac{m}{T}-j}.
    \end{aligned}
  \end{equation}
\end{example}

By \cite[6.5.9]{FBZ04}, following \cite[\S 2.6]{DGT19}, we may degree-wise trivialize $\cV$.
\begin{lemma}\label{lem:GammaUV}
  The bundle $\cV$ is filtered by the sheaves $\cV_{\le k}$ whose associated graded sheaf is isomorphic to the graded sheaf $\bigoplus_{k\in\N}V_k\otimes_{\C}\Omega^{\otimes -k}$.
  In particular, on an affine open subset $U$ of $\X$, we have
  \[
    \Gamma(U,\cV) \cong 
    \bigoplus_{k=0}^{\infty}
    V_k\otimes\Gamma\big(U,\Omega^{\otimes -k}\big).
  \]
\end{lemma}
\begin{definition}\label{def:wt}
  The filtration $\cV_{\le\star}$ is called the \textbf{weight filtration}. 
  For any section $\alpha$ of the bundle $\cV$, we write $\wt\alpha\le k$ for $\alpha\in\cV_{\le k}$. 
  Passing to the associated graded sheaf $\gr_{\star}\cV$, we use $\wt\alpha$ to denote the \textbf{weight} of $\alpha$. 
\end{definition}

The bundle $\cV$ is equipped with a \emph{connection} $\nabla\colon\cV\to\cV\otimes\Omega$, which can be expressed locally as 
\begin{equation}\label{eq:nabla}
  \iota_{z}\nabla=\vo{L}{-1}\otimes\id + \id\otimes\partial_{z}
\end{equation}
on a local chart $(\pp,z)$. 
At any point $\pp$, the vertex operator $Y(-,z)$ on $V$ gives rise to an $\End(\cV_{\pp})$-valued meromorphic section $\cY_{\pp}$ of $\cV^{\ast}$, \emph{horizontal} with respect to the dual connection $\nabla^{\ast}$. 
The section $\cY_{\pp}$ can be expressed by the formula 
\[
  \braket*{\varphi}{\cY_{\pp}(\spa)\nu} = 
  \braket*{\iota_{z}(\varphi)}{Y(\iota_{z}(\spa),z)\iota_{z}(\nu)},
\]
for all $\varphi\in\cV_{\pp}^{\ast}$, $\nu\in\cV_{\pp}$ and all sections $\spa$ of $\cV$. 

Now, we consider a totally ramified orbicurve $\x\colon\tilX\to\X$ and \emph{suppose its Galois group $G$ acts on $V$ as VOA-automorphisms.}
In particular, the action of $G$ commutes with the Virasoro action in \cref{rem:VirasoroAction}. 
Consequently, the bundle $\cV$ is $G$-equivariant, and we thus obtain a bundle $\cV^G$ on the unramified locus $\Xc$. 
The connection $\nabla$ in \labelcref{eq:nabla} descends to one on the bundle $\cV^G$, for which we will adopt the notation $\nabla$ again. 

\begin{definition}
  The $\cD$-module provided by the pair $(\cV^G,\nabla)$ (or its direct image along the inclusion $\Xc\monomorphism\X$) is called the \textbf{twisted vertex algebra bundle} generated by $V$ on the orbicurve $\x\colon\tilX\to\X$.
\end{definition}

\subsection{The twisted chiral Lie algebra}
Let $\pi_i$ ($i=1,2$) be the projections $\X^2\to\X$ and $\Delta\colon\X\to\X^2$ the diagonal morphism. 
These morphisms can be pullback to obtain morphisms between orbicurves, for which we will adopt the same notation.
The following formula defines a $\C$-bilinear Lie bracket on the \emph{de Rham cohomology} $\H^0(-,\cV^G\otimes\Omega/\Image\nabla)$
\begin{equation}\label{eq:LieBracket}
  \alpha\otimes\beta\longmapsto
  \Delta^{\ast}\Res\braket*{\cY}{\pi_1^{\ast}\alpha}\pi_2^{\ast}\beta  
\end{equation}

\begin{definition}\label{def:twistedchiralLie}
  For an open subset $U$ of $\Xc$ (resp. a point $\pp$ of $\X$), we define the \textbf{twisted chiral Lie algebra} $\L_{U}(\cV^G)$ (resp. $\L_{\pp}(\cV^G)$) \textbf{on} $U$ (resp. \textbf{at} $\pp$) as the de Rham cohomology $\H^0(U,\cV^G\otimes\Omega/\Image\nabla)$ (resp. $\H^0(D_{\pp}^{\times},\cV^G\otimes\Omega/\Image\nabla)$) equipped with the Lie bracket defined by the formula \labelcref{eq:LieBracket}. 
\end{definition}

The restrictions of twisted chiral Lie algebras are always injective since our underlying schemes are integral. In particular, we have the following embedding of Lie algebras:
\[
  \L_{U}(\cV^G) \longmonomorphism
  \prod_{\pp\in\X\setminus U}\L_{\pp}(\cV^G).
\]

For a point $\pp\in\Xc$, a local coordinate $z$ at $\pp$ identifies the (twisted) chiral Lie algebra with the \emph{untwisted} left ancillary Lie algebra:
\[
  \L_{\pp}(\cV^G)\cong
  \L_{\tilde\pp}(\cV)\overset{\iota_{z}}{\cong}   
  \L(V)^{\sfL},
\]
where $\tilde\pp$ is any lifting of $\pp$. 

For a branch point $\qq$, a special local coordinate $z^{\sfrac{1}{T}}$ above $\pp$ identifies the twisted chiral Lie algebra with the \emph{$g_\qq$-twisted} left ancillary Lie algebra
\begin{equation*}
  \L_{\qq}(\cV^G)\overset{\iota_{z}}{\cong}
  \L_{g_{\qq}}(V)^{\sfL}.
\end{equation*}

\begin{example}\label{eg:LgV}
  On the \emph{twisted projective line} $\x\colon\oC\to\PP^1$ in \cref{eg:twistedP1}, we have the following explicit descriptions of the twisted chiral Lie algebra $\L_{U}(\cV^G)$.
  \begin{itemize}[wide]
    \item On $\PP^1\setminus\Set{0,\infty}$, the coordinates $z^{\sfrac{1}{T}}$ and $(z^\dagger)^{\sfrac{1}{T}}$ provide the identifications
    \[
      \L_{\PP^1\setminus\Set{0,\infty}}(\cV^G)\overset{\iota_z}{\cong}\L_g(V)\txand 
      \L_{\PP^1\setminus\Set{0,\infty}}(\cV^G)\overset{\iota_{z^\dagger}}{\cong}\L_{g^{-1}}(V),
    \]
    where $g=g_0$, and we have $g^{-1}=g_{\infty}$.
    The two Lie algebras are related via the transition map $\vartheta$ in \labelcref{eq:def:vartheta}. 
    One can see that this isomorphism is precisely $-\theta$, where $\theta$ is the anti-isomorphism in \labelcref{eq:def:theta}.
    \item The restriction from $\PP^1\setminus\Set{0,\infty}$ to $\PP^1\setminus\Set{0,w,\infty}$, where $w\neq 0,\infty$, provides the following localization of $\L_g(V)$:
    \[
      \L_{\PP^1\setminus\Set{0,w,\infty}}(\cV^G)\overset{\iota_z}{\cong}(V\otimes\C[z^{\pm\sfrac{1}{T}},(z-w)^{-1}]\d{z})^{g}/\Image\nabla.
    \]
    The latter vector space is spanned by classes of the form 
    \[\loo{a}[r]{m}{n}:=a\otimes\frac{z^{\frac{r}{T}+m}}{(z-w)^{n}}\d{z}+\Image\nabla,\] 
    where $a\in V^{r}$ is homogeneous and $m\in\Z$, $n\in\N$. 
    The Lie bracket is as follows:
    \begin{equation*}
      \Liebracket*{\loo{a}[r]{m}{n}}{\loo{b}[r']{m'}{n'}} = 
      \sum_{i,j\ge 0}
      \binom{\frac{r}{T}+m}{i}\binom{-n}{j}
      \loo{(a_{i+j}b)}[r+r']{m+m'-i}{n+n'+j}.
    \end{equation*}
  \end{itemize}
\end{example}
\begin{example}\label{eg:3ptsLV}
  For the obicurve $\oC_w\to\PP^1$ in \cref{eg:3ptP1}, we have the following explicit description of the twisted chiral Lie algebras:
    \[
      \L_{\PP^1\setminus\Set{0,w,\infty}}(\cV^G)
      \overset{\iota_z}{\cong}
      \left(V\otimes\C[z^{\pm\sfrac{1}{T}},(z-w)^{\pm\sfrac{1}{T}}]\d{z}\right)^{g}/\Image\nabla,
    \]
    where $g=g_0$, and we have $g_w=g^\epsilon$.  The latter vector space is spanned by classes of the form (where \emph{the superscript $(r)$ is omitted if $r=0$}) 
    \[
      \loo{a}[r]{m}{n}:=a\otimes\frac{z^{\frac{r}{T}+m}}{(z-w)^{-\epsilon\frac{r}{T}+n}}\d{z}+\Image\nabla,
    \] 
    where $a\in V^{r}$ is homogeneous and $m\in\Z$, $n\in\N$. 
    The Lie bracket is as follows:
    \begin{equation*}
      \Liebracket*{\loo{a}[r]{m}{n}}{\loo{b}[r']{m'}{n'}} = 
      \sum_{i,j\ge 0}
      \binom{\frac{r}{T}+m}{i}\binom{\epsilon\frac{r}{T}-n}{j}
      \loo{(a_{i+j}b)}[r+r']{m+m'-i}{n+n'+j}.
    \end{equation*}
\end{example}

\subsubsection*{Actions of twisted chiral Lie algebras on twisted modules}
Following \cite[\S 7.3]{FBZ04}, an admissible \emph{untwisted} module $M$ gives rise to a bundle $\cM$, locally obtained from $M$ twisted by the choices of local coordinates\footnote{More precisely, we should further twist it by an action of $\G_m$. But when $M$ is of conformal weight $h$, such an action is obtained from the change of coordinates plus a degree shifting.}. 
Then the action of $G$ further gives rise to a bundle $\cM^G$ on $\Xc$, whose sections are $G$-equivariant sections of $\cM$. In particular, we have a $\kappa_\pp$-space $\cM_{\pp}$ at each unramified point $\pp\in\Xc$.

Following \cite[\S 5]{FS04}, we can attach an admissible \emph{$g_{\pp}$-twisted} $V$-module $M$ to a branch point $\pp$ (and hence its unique lifting) and obtain a $\kappa_{\pp}$-space $\cM_{\pp}$ by twisting the choices of \emph{special} local coordinates. 

In any case, the vertex operator $Y_{M}$ gives rise to a $G$-equivariant $\End\cM_{\pp}$-valued local section $\cY_{M,\pp}$ of $\cV^{\ast}$ for all points $\pp\in\X$ and their liftings. 
Then  there is a natural continuous action of the {twisted chiral Lie algebra} on the fiber $\cM_{\pp}$ given by the formula:
\begin{equation}\label{eq:actionOfLV}
  \alpha\in \L_{\pp}(\cV^{G})\longmapsto\Res_{\pp}\braket*{\cY_{M,\pp}}{\alpha}\in\End\cM_{\pp}.  
\end{equation}
\begin{example}\label{eg:chiralactiononP1}
  More explicitly, the following table shows the expansion of $\Res_{\pp}\braket*{\cY_{M,\pp}}{\alpha}$ for the given typical $\alpha$ on the \emph{twisted projective line} $\x\colon\oC\to\PP^1$.
  \begin{table}[!h]
    \centering
      \renewcommand{\arraystretch}{1.5}
      \begin{tabular}{c|c|c}
      Local chart & Typical $\alpha$ & Expansion of $\Res_{\pp}\braket*{\cY_{M,\pp}}{\alpha}$   \\
      \hline
      $(0,z^{\sfrac{1}{T}})$ & 
      $\alpha=a\otimes z^{\frac{n}{T}}\d{z}$ & 
      $\displaystyle \iota_{z^{\sfrac{1}{T}}}\Res_{0}\braket*{\cY_{M,0}}{\alpha}=T\vo{a}{\frac{n}{T}}$ \\[.3em]
      \hline
      $(\infty,z^{-\sfrac{1}{T}})$ & 
      $\alpha=a\otimes z^{\frac{n}{T}}\d{z}$ & 
      $\displaystyle \iota_{z^{-\sfrac{1}{T}}}\Res_{\infty}\braket*{\cY_{M,\infty}}{\alpha}=-T\theta(\lo{a}{\frac{n}{T}})$ \\[.3em]
      \hline
      $(\pp,z-z(\pp))$ & 
      $\displaystyle \alpha=a\otimes(z-z(\pp))^n\d{z}$ & 
      $\displaystyle \iota_{z-z(\pp)}\Res_{\pp}\braket*{\cY_{M,\pp}}{\alpha}=\vo{a}{n}$ \\[.3em]
      \hline
      \end{tabular}
      \caption{The local actions of $\L_{-}(\cV^G)$ at $\pp$.}\label{tab:EXPY}
  \end{table}  
\end{example}

\subsubsection*{A consequence of Riemann-Roch Theorem}
The following lemma will be used later to guarantee the existence of certain elements of the twisted Lie algebra with desired local expressions.
\begin{lemma}\label{lem:RRcoro++}
  Let $U$ be an affine open subset of $\X$ and let $\qq_{\diamond}$ be all the branch points inside $U$. 
  Let $z_i^{\sfrac{1}{T}}$ be a special local coordinate above $\qq_{i}$ and $g_{i}$ the Deck generator at $\qq_{i}$, for each $i$. 
  Then  for any $a\in V_{g_{i}}^r$ and all $d,m\in\Z$, there exists an element $\alpha\in\L_{U\setminus\qq_{\diamond}}(\cV^G)$ such that 
  \[
    \begin{array}{ll}
      \iota_{z_i}\alpha \equiv a\otimes z_i^{\frac{r}{T}+d}\d{z_i} &\mod z_i^m(V\dbrack{z_i}\d{z_i})^{g_i}, \text{ for a fixed $i$, and} \\
      \iota_{z_j}\alpha \equiv 0 &\mod z_j^m(V\dbrack{z_j}\d{z_j})^{g_j}, \text{ for all }j\neq i. \\
    \end{array}
  \]
\end{lemma}
\begin{proof}
  It suffices to prove the statement when $a$ is homogeneous. 
  By \cref{lem:RRcoro}, there exists $\mu\in\Gamma\left(\x^{-1}(U\setminus\qq_{\diamond}),\Omega_{\tilX}^{\otimes 1-\wt a}\right)$ such that 
  \[
    \begin{array}{ll}
      \iota_{z_i}\mu \equiv z_i^{\frac{r}{T}+d}(\d{z_i})^{1-\wt a} &\mod z_i^m\C\dbrack{z_i^{\sfrac{1}{T}}}(\d{z_i})^{1-\wt a}, \text{ for a fixed $i$, and} \\
      \iota_{z_j}\mu \equiv 0 &\mod z_j^m\C\dbrack{z_j^{\sfrac{1}{T}}}(\d{z_j})^{1-\wt a}, \text{ for all }j\neq i. \\
    \end{array}
  \]
  Note that $a\otimes\mu$ is $G$-invariant. 

  On the other hand, by \cref{lem:GammaUV}, for any affine open subset $U$ of $\X$, 
  \[
    \Gamma\big(\x^{-1}(U),\cV\big) \cong \bigoplus_{k=0}^{\infty}V_k\otimes\Gamma\big(\x^{-1}(U),\Omega_{\tilX}^{\otimes -k}\big).
  \]
  Then  the following is evident.
  \[
    \L_{U}(\cV^G) \cong 
    \bigoplus_{k=0}^{\infty}\big(V_k\otimes\Gamma\big(\x^{-1}(U),\Omega_{\tilX}^{\otimes 1-k}\big)\big)^G.
  \]
  Applying this to the affine open $U$ in the lemma, we see that there exists an element $\alpha\in\L_{U\setminus\qq_{\diamond}}(\cV^G)$ corresponding to $a\otimes\mu$. It is then evident that $\alpha$ satisfies the desired properties.
\end{proof}

\subsection{Constraints of the twisted chiral Lie algebra}\label{sec:Constraints}
Let $\dQ$ be the \emph{ramification divisor}, namely $T$ (the ramification index) times the sum of all (the lifting of) the branch points. 
We want to consider a \emph{filtration} (we refer to \cite[Appendix A]{DGK23} for the generalities on split filtrations) on the twisted chiral Lie algebra provided by constraining the valuations at $\dQ$. 
\begin{notation}\label{nt:defL0}
  Let $U$ be any open subset of $\X$. 
  Denote the twisted chiral Lie algebra $\L_{U\setminus\qq_\diamond}(\cV^G)$ by $\L^{\circ}_{U}(\cV^G)$ and consider the following subspaces and subquotients of it:
  \begin{align*}
    \L^{\circ}_{U}(\cV^G)_{\le m} &:= \H^0\left(U,\left(\cV\otimes\Omega^1(-(\vo{L}{0}-m-1)\dQ)\right)^G/\Image\nabla\right), \\ 
    \L^{\circ}_{U}(\cV^G)_{<m} &:= \bigcup_{n<m}\L^{\circ}_{U}(\cV^G)_{\le n}, \\
    \L^{\circ}_{U}(\cV^G)_{m} &:= \L^{\circ}_{U}(\cV^G)_{\le m}/\L^{\circ}_{U}(\cV^G)_{<m}.
  \end{align*}
  Then $\L^{\circ}_{U}(\cV^G)_{\le \star}$ forms a \emph{left} filtration on $\L^{\circ}_{U}(\cV^G)$.
\end{notation}
The filtration $\L_{U}(\cV^G)_{\le \star}$ is both \emph{exhaustive} and \emph{separated}, and is \emph{split} since we are working with bundles. 
Furthermore, by comparing the valuations of both sides of \labelcref{eq:LieBracket}, we see that $\L^{\circ}_{U}(\cV^G)$ is a split-filtered Lie algebra with the associated graded Lie algebra $\L^{\circ}_{U}(\cV^G)_{\star}$.
In particular, $\L^{\circ}_{U}(\cV^G)_{0}$ is a subquotient Lie algebra of $\L^{\circ}_{U}(\cV^G)$.

\begin{definition}\label{def:constraint}
  The subquotient algebra $\L^{\circ}_{U}(\cV^G)_{0}$ is called the \textbf{constrained twisted chiral Lie algebra}.
\end{definition}

\begin{definition}
  Adopting the weight filtration in \cref{def:wt}, the constraining condition for an element $\alpha$ to be contained in $L^{\circ}_{U}(\cV^G)_{\prec m}$ (where $\prec$ stands for either $\le$ or $<$) can be expressed as 
  \begin{equation}\label{eq:constrain}
    \wt{\alpha}-1-v_{\qq}(\alpha)\prec m\txforall\qq\in U\setminus\Xc.
  \end{equation} 
  We call the number $\wt{\alpha}-1-v_{\qq}(\alpha)$ the \textbf{$\qq$-degree} of $\alpha$ and denote it by $\deg_{\qq}(\alpha)$.
\end{definition}

\begin{example}\label{eg:LgV-0}
  Following \cref{eg:LgV}, we have the following explicit descriptions on the \emph{twisted projective line} $\x\colon\oC\to\PP^1$ (where $w\neq 0,\infty$ and $\prec$ stands for either $\le$ or $<$)
  \begin{align*}
    \L_{\PP^1\setminus\Set{w}}(\cV^G)_{\prec 0}&=\spn\Set*{ \loo{a}[r]{m}{n} \given 
    \begin{array}{c}
      a\in V_g^r\text{ is homogeneous}, m\in\Z, n\in\N,  \\
      \frac{r}{T}+m-n \prec \wt{a}-1 \prec \frac{r}{T}+m 
    \end{array}},\\
    \L_{\PP^1\setminus\Set{\infty,w}}(\cV^G)_{\prec 0}&=\spn\Set*{ \loo{a}[r]{m}{n} \given 
    \begin{array}{c}
      a\in V_g^r\text{ is homogeneous}, m\in\Z, n\in\N,  \\
      \wt{a}-1 \prec \frac{r}{T}+m 
    \end{array}},\\
    \L_{\PP^1\setminus\Set{\infty}}(\cV^G)_{\prec 0}&=\spn\Set*{ \lo{a}{\frac{r}{T}+m} \given 
    \begin{array}{c}
      a\in V_g^r\text{ is homogeneous},  \\
       m\in\Z, \wt{a}-1 \prec \frac{r}{T}+m 
    \end{array}}=\L_g(V)_{\prec 0}.
  \end{align*}
\end{example}

\subsubsection*{The local constraints}
For $\qq$ a branch point, we adopt similar constructions. 
\begin{notation}
  Consider the following subspaces and subquotients of $\L_{\qq}(\cV^G)$:
  \begin{align*}
    \L_{\qq}(\cV^G)_{\le m} &:= \H^0\left(D_{\qq},\left(\cV\otimes\Omega^1(-(\vo{L}{0}-m-1)\dQ)\right)^G/\Image\nabla\right), \\ 
    \L_{\qq}(\cV^G)_{<m} &:= \bigcup_{n<m}\L_{\qq}(\cV^G)_{\le n}, \\ 
    \L_{\qq}(\cV^G)_{m} &:= \L_{\qq}(\cV^G)_{\le m}/\L_{\qq}(\cV^G)_{<m}.
  \end{align*}
  Then $\L_{\qq}(\cV^G)_{\le \star}$ forms a \emph{left} filtration on $\L_{\qq}(\cV^G)$.
\end{notation}
Note that, by choosing a special local coordinate $z^{\sfrac{1}{T}}$, we have 
\[
  \L_{\qq}(\cV^G)_{\le m}\overset{\iota_z}{\cong}\L_{g}(V)^{\sfL}_{\le m}\txand\L_{\qq}(\cV^G)_{< m}\overset{\iota_z}{\cong}\L_{g}(V)^{\sfL}_{< m},
\]
where $g$ is the Deck generator at $\qq$. In particular, $\iota_{z}$ induces an isomorphism of graded Lie algebras $\L_{\qq}(\cV^G)_{\star}\cong\L_{g}(V)$.

\subsubsection*{Restriction of opens and formal expansions}
\begin{lemma}\label{lem:flasque}
  The sheaf $U\mapsto\L^{\circ}_{U}(\cV^G)_{m}$ is flasque.
\end{lemma}
\begin{proof}
  Let $U'\subset U$ be two opens of $\X$. Then  the following epimorphisms and isomorphisms are evident from \labelcref{eq:constrain}
  \begin{align*}
    \L^{\circ}_{U}(\cV^G)_{m}=\frac{\L^{\circ}_{U}(\cV^G)_{\le m}}{\L^{\circ}_{U}(\cV^G)_{< m}}\longepimorphism&
    \frac{\L^{\circ}_{U}(\cV^G)_{\le m}}{\L^{\circ}_{U}(\cV^G)_{\le m}\cap\L^{\circ}_{U'}(\cV^G)_{< m}}\\
    &\cong\frac{\L^{\circ}_{U'}(\cV^G)_{\le m}}{\L^{\circ}_{U'}(\cV^G)_{< m}}=\L^{\circ}_{U'}(\cV^G)_{m}.\qedhere
  \end{align*}
\end{proof}
In particular, we have epimorphisms $\L^{\circ}_{U}(\cV^G)_{m}\epimorphism\L_{\qq}(\cV^G)_{m}$ for any branch point $\qq$ that is contained in $U$.  
Since the image of $\L^{\circ}_U(\cV^G)_{< m}$ in $\L_{\qq}(\cV^G)$ is contained in $\L_{\qq}(\cV^G)_{< m}$, the above epimorphism factors through the expansion $\iota_{\qq}|_{\L^{\circ}_{U}(\cV^G)_{\le m}}$.
We thus also adopt the notation $\iota_{\qq}$ for this epimorphism.

\subsubsection*{Explicit description of $\L^{\circ}_{U}(\cV^G)_{m}$}
\begin{notation}  
  By \cref{lem:RRcoro++}, for any branch point $\qq$ that is contained in $U$, any homogeneous $a\in V_{g_\qq}^{r}$, and any $m\in-\frac{r}{T}+\Z$, there is a unique class $\alpha\in\L^{\circ}_{U}(\cV^G)_{m}$ such that 
  \[
    \iota_{\qq}(\alpha)=\lo{a}{\wt a-m-1}\txand
    \iota_{\qq'}(\alpha)=0\quad\text{if}\quad\qq'\neq\qq.
  \]
  We denote this section by $\sqo{a}{\qq}{m}$. When $m=0$, we omit the superscript $(0)$.
\end{notation}
Then  we have the following explicit description of twisted chiral Lie algebras:
\begin{theorem}\label{thm:DesLUVCG0}
  For any open $U$, the vector space $\L^{\circ}_{U}(\cV^G)_{m}$ is spanned by the elements $\sqo{a}{\qq}{m}$ for all homogeneous $a\in V$ and all branch points $\qq\in U$.
\end{theorem}
\begin{proof}
  First note that the statement is locally true: the $\kappa_{\qq}$-space $\L_{\qq}(\cV^G)_{m}$ is spanned by $\lo{a}{\wt{a}-m-1}$ for all homogeneous $a\in V$. 

  To prove the statement, we consider the \emph{affine covering} $U_{\diamond}=\Set{U_i}_{i\in\diamond}$ of $U$, where each $U_i=U\cap(\Xc\sqcup\qq_i)$. 
  Since the sheaf $\L^{\circ}_{-}(\cV^G)_{m}$ is \emph{flasque}, we have
  \[
    \L^{\circ}_{U}(\cV^G)_{m} = 
    \prod_{i\in\diamond}\L^{\circ}_{U_i}(\cV^G)_{m}.
  \]
  Now, for any open neighborhood\footnote{if $\qq_i\notin U$, then $N$ is empty.} $N$ of $\qq_i$ inside $U_i$, we have 
  \[
    \L^{\circ}_{U_i}(\cV^G)_{< m} = 
    \L^{\circ}_{U_i}(\cV^G)_{\le m}\cap\L^{\circ}_{N}(\cV^G)_{< m}
  \]
  by \labelcref{eq:constrain}, since there is no branch point on the open $U_i\setminus N$. Mimicing the proof of \cref{lem:flasque} and passing to the limit, we have 
  \[
    \L^{\circ}_{U_i}(\cV^G)_{m}\overset{\iota_{\qq_i}}{\cong}\L_{\qq_{i}}(\cV^G)_{m}.
  \]
  Then the statement follows from its local version.
\end{proof}
\begin{corollary}\label{coro:DesLUVCG0}
  By abuse of notation, let $\sqo{a}{\qq}{m}$ also denote a representative of it in $\L^{\circ}_{U}(\cV^G)_{\le m}$. Then each $\L^{\circ}_{U}(\cV^G)_{\le m}$ is topologically spanned by $\sqo{a}{\qq}{n}$ for all homogeneous $a\in V$,  all branch points $\qq\in U$, and all $n\le m$.
\end{corollary}

In particular, the sheaf $\L^{\circ}_{-}(\cV^G)$ is a split-filtered sheaf of Lie algebras with the associated graded sheaf of Lie algebras $\L^{\circ}_{-}(\cV^G)_{\star}$. To mimic the argument in \cref{sec:assoalg}, we may complete the pair $(\L^{\circ}_{-}(\cV^G),\L^{\circ}_{-}(\cV^G)_{\star})$ to a sheaf of good triples, or merely maintain a sheaf of \emph{good pairs}.

\section{Twisted conformal blocks, correlation functions, and twisted intertwining operators}\label{sec:Cfb}

In this section, we introduce the notion of \emph{twisted conformal blocks}, and by working with \cref{eg:twistedP1} and \cref{eg:3ptP1}, we explain their relations to \emph{$g$-twisted correlation functions} introduced in \cite[\S 2.3]{PartI} and \emph{twisted intertwining operators}, respectively.

\subsection{Twisted conformal blocks}
We adopt the following definition from \cite{FBZ04,FS04}.
\begin{definition}\label{def:Cfb}
  Consider a datum $\Sigma=(\x\colon\tilX\to\X,\pp_{\bullet},\qq_{\diamond},M^{\bullet},N^{\diamond})$ of 
  \begin{itemize}
    \item a totally ramified orbicurve $\x\colon\tilX\to\X$, 
    \item distinct points $\pp_{\bullet}=(\pp_{1},\cdots,\pp_{m})$ on the unramified locus $\Xc$, 
    \item all the branch points $\qq_{\diamond}=(\qq_{1},\cdots,\qq_{n})$, each provides a \emph{Deck generator} $g_{i}$,
    \item admissible \emph{untwisted} modules $M^{\bullet}=(M^1,\cdots,M^m)$, one for each $\pp_{i}$, and
    \item admissible \emph{$g_{\diamond}$-twisted} modules $N^{\diamond}=(N^1,\cdots,N^n)$, one for each $\qq_{i}$.
  \end{itemize}
  Then the \textbf{space of (twisted) coninvariants associated to the datum $\Sigma$} is the quotient space 
  \begin{equation}\label{eq:defconinv}
    (\cM^{\bullet}_{\pp_{\bullet}}\otimes\cN^{\diamond}_{\qq_{\diamond}})_{\L_{\X\setminus\pp_{\bullet}}(\cV^G)}:= 
    (\cM^{\bullet}_{\pp_{\bullet}}\otimes\cN^{\diamond}_{\qq_{\diamond}})/ 
    \L_{\X\setminus\pp_{\bullet}}(\cV^G).(\cM^{\bullet}_{\pp_{\bullet}}\otimes\cN^{\diamond}_{\qq_{\diamond}}).
  \end{equation}
  Its linear dual is called the \textbf{space of (twisted) conformal blocks associated to the datum $\Sigma$} and is denoted by $\Cfb[\Sigma]$. 
  Equivalently, a \emph{(twisted) conformal block associated to the datum $\Sigma$} is a linear functional on $\cM^{\bullet}_{\pp_{\bullet}}\otimes\cN^{\diamond}_{\qq_{\diamond}}$ that is $\L_{\Xc\setminus\pp_{\bullet}}(\cV^G)$-invariant.
\end{definition}

\begin{lemma}[Propagation of vacua]\label{lem:propagation}
  There is a canonical isomorphism
  \[
    \xi_{\pp}\colon\Cfb[\x\colon\tilX\to\X,\pp,\pp_{\bullet},\qq_{\diamond},V,M^{\bullet},N^{\diamond}]\cong 
    \Cfb[\x\colon\tilX\to\X,\pp_{\bullet},\qq_{\diamond},M^{\bullet},N^{\diamond}],
  \]
  given by restricting to $\vac\otimes\cM^{\bullet}_{\pp_{\bullet}}\otimes\cN^{\diamond}_{\qq_{\diamond}}$. 
\end{lemma}
\begin{proof}
  The proof is essentially the same as in \cite[Theorem 10.3.1]{FBZ04}, except we need the \emph{twisted Jacobi identity} in addition to the untwisted one. 
  
  First, by applying the \emph{\nameref{lem:SRT}} to the definition above, we see that a linear functional $\varphi$ on $\cM^{\bullet}_{\pp_{\bullet}}\otimes\cN^{\diamond}_{\qq_{\diamond}}$ is a conformal block if and only if for any $\mu_{\bullet}\in\cM^{\bullet}_{\pp_{\bullet}}$ and $\nu_{\diamond}\in\cN^{\diamond}_{\qq_{\diamond}}$, the followings: 
  \[
    \braket*{\varphi}{\cdots\otimes\cY_{M^i,\pp_i}\mu_{i}\otimes\cdots\otimes\nu_{\diamond}}\txand
    \braket*{\varphi}{\mu_{\bullet}\otimes\cdots\otimes\cY_{N^j,\qq_{j}}\nu_{j}\otimes\cdots},
  \]
  can be extended to the \emph{same} meromorphic section of $(\cV^G)^{\ast}$ on $\X$ with possible poles at $\pp_{\bullet}$ and $\qq_{\diamond}$. 
  Denote this section by $\varphi_{\mu_{\bullet}\otimes\nu_{\diamond}}$. 
  Then  the assignment 
  \[
    \spa\otimes\mu_{\bullet}\otimes\nu_{\diamond}\in \cV_{\pp}\otimes\cM^{\bullet}_{\pp_{\bullet}}\otimes\cN^{\diamond}_{\qq_{\diamond}}
    \longmapsto
    \varphi_{\mu_{\bullet}\otimes\nu_{\diamond}}(\spa)
  \]
  gives a conformal block $\xi_{\pp}^{-1}\varphi$ (one can verify this using the twisted and untwisted Jacobi identities).   
  This gives the inverse of $\xi_{\pp}$.
\end{proof}
\begin{remark}\label{rem:xi_p}
  Through the (non-canonical) identification $\cV^G_{\pp}\cong\cV_{\tilde\pp}$ for any lifting $\tilde\pp$ of $\pp$, we obtain a linear functional on $\cV_{\tilde\pp}\otimes\cM^{\bullet}_{\pp_{\bullet}}\otimes\cN^{\diamond}_{\qq_{\diamond}}$ from $\xi_{\pp}^{-1}\varphi$. 
  We denote this linear functional by $\xi_{\tilde\pp}^{-1}\varphi$. 
  Then  the \emph{$G$-equivaraint} lifting of the meromorphic section $\varphi_{\mu_{\bullet}\otimes\nu_{\diamond}}\in(\cV^G)^{\ast}$ to $\cV^{\ast}$  can be expressed as $\braket*{\xi_{\tilde\pp}^{-1}\varphi}{-\otimes\mu_{\bullet}\otimes\nu_{\diamond}}$. 
\end{remark}

\subsection{Relation to $g$-twisted correlation functions}\label{sec:Cor}

Now, we establish the relationship between \cref{def:Cfb} and the notion of \emph{$g$-twisted correlation functions} in \cite[Definition 2.11]{PartI}. 
The datum we are interested in is
\begin{equation}\label{eq:def:datum}
  \Sigma_{\qq}(N^3, M^1, M^2):=(\x\colon\oC\to\PP^1, \infty, \x(\qq), 0, N^3, M^1, M^2),
\end{equation}
where $M^1$ (resp. $M^2$ and $N^3$) is an admissible \emph{untwisted} (resp. \emph{$g$-twisted} and \emph{$g^{-1}$-twisted}) module of conformal weight $h_1$ (resp. $h_2$ and $h_3$), and $\x\colon\oC\to\PP^1$ is the twisted projective line in \cref{eg:twistedP1}. 
Put $h=h_1+h_2-h_3$. 
\emph{In the rest of this subsection, we keep the following identification via the local charts given in \cite[\S 2.2]{PartI}: }
\[
  \cV^{G}_{\x(\pp)}\cong\cV_{\pp}\cong V, \quad
  (\cM^1)^{G}_{\x(\qq)}\cong\cM^1_{\qq}\cong M^1,\quad 
  \cM^2_{0}\cong M^2,\quad
  \cN^3_{\infty}\cong N^3.
\]
The spaces $(\cN^3_{\infty}\otimes\cM^1_{\qq}\otimes\cM^2_{0})^{\ast}$ for various $\qq\in\Cc$ form a bundle $(\cN^3_{\infty}\otimes\cM^1_{\square}\otimes\cM^2_{0})^{\ast}$ over the affine open subset $\Cc\subset\oC$. 
Taking a coordinate $w^{\sfrac{1}{T}}$ of $\Cc$, this bundle is trivialized as $(N^3\otimes M^1\otimes M^2)^{\ast}[w^{\pm\sfrac{1}{T}}]$.
\emph{Let $1$ denote any point on $\Cc$ such that $w(1)=1$}. Take any $\varphi_{1}=\varphi\in\Cfb[\Sigma_{1}(N^3, M^1, M^2)]$ and, at each point $\qq\in\Cc$, define $\varphi_{\qq}$ by the formula:
\begin{equation}\label{eq:CfbFamily}
  \braket*{\varphi_{\qq}}{v_3\otimes v\otimes v_2} := 
  \braket*{\varphi}{w^{\vo{L}{0}-h_3}v_3\otimes w^{-\vo{L}{0}+h_1}v\otimes w^{-\vo{L}{0}+h_2}v_2}.  
\end{equation}
We have:
\begin{enumerate}
  \item the assignment $\qq\to\varphi_{\qq}$ defines a rational section $\varphi_{\bullet}$ of $(\cN^3_{\infty}\otimes\cM^1_{\square}\otimes\cM^2_{0})^{\ast}$;
  \item the value of $\varphi_{\bullet}$ at each point $\qq\in\Cc$ gives a conformal block associated to the datum $\Sigma_{\x(\qq)}$.
\end{enumerate}

\begin{lemma}\label{lem:horizontal}
  The section $\varphi_{\bullet}$ is horizontal in the sense that it is annihilated by the following linear map\footnote{When $h\in\frac{1}{T}\Z$, the linear map $\nabla$ defines a connection on the bundle $(\cN^3_{\infty}\otimes\cM^1_{\square}\otimes\cM^2_{0})^{\ast}$. In general, this is NOT the case since $w^h$ is not a rational function on $\Cc$.} $\nabla\colon(\cN^3_{\infty}\otimes\cM^1_{\square}\otimes\cM^2_{0})^{\ast}\to(\cN^3_{\infty}\otimes\cM^1_{\square}\otimes\cM^2_{0})^{\ast}\otimes\Omega$
  \[
    \nabla = w^{h}\d\circ w^{-h} - \vo{L}{-1}^{\cM^1}\d{w},
  \]
  where $\vo{L}{-1}^{\cM^1}$ is the $\vo{L}{-1}$-operator on the bundle $\cM^1_{\square}$.
\end{lemma}
\begin{proof}
  Since $\varphi$ is a conformal block, for $\loo{\upomega}{1}{0}\in\L_{\PP^1\setminus\Set{0,\infty,1}}(\cV^{G})$, we have, 
  \begin{align*}
    0 &=\braket*{\varphi}{\loo{\upomega}{1}{0}.(v_3\otimes v\otimes v_2)}\\
    &= 
    \Res_{\infty}\braket*{\varphi}{\cY_{N^3,\infty}(\loo{\upomega}{1}{0})v_3\otimes v\otimes v_2} +
    \Res_{0}\braket*{\varphi}{v_3\otimes v\otimes \cY_{M^2,0}(\loo{\upomega}{1}{0})v_2} \\
    &\phantom{=}  + 
    T\Res_{1}\braket*{\varphi}{v_3\otimes \cY_{M^1,1}(\loo{\upomega}{1}{0})v\otimes v_2} \\
    &=T\left(-\vo{L}{0}v_3\otimes v\otimes v_2 + v_3\otimes v\otimes \vo{L}{0}v_2 + v_3\otimes (\vo{L}{0}+\vo{L}{-1})v\otimes v_2\right).
  \end{align*}
  Hence, $\vo{L}{0}v_3\otimes v\otimes v_2 - v_3\otimes v\otimes \vo{L}{0}v_2 - v_3\otimes \vo{L}{0}v\otimes v_2 - v_3\otimes\vo{L}{-1}v\otimes v_2$ is annihilated by $\varphi$. 
  Now, we have
  \begin{align*}
    \braket*{\nabla \varphi_{\qq}}{v_3\otimes v\otimes v_2} &= 
    w^{h}\d w^{-h} \braket*{\varphi}{w^{\vo{L}{0}-h_3}v_3\otimes w^{-\vo{L}{0}+h_1}v\otimes w^{-\vo{L}{0}+h_2}v_2} \\
    &\phantom{=}  - \braket*{\varphi}{w^{\vo{L}{0}-h_3}v_3\otimes \vo{L}{-1}w^{-\vo{L}{0}+h_1}v\otimes w^{-\vo{L}{0}+h_2}v_2}\d{w}\\
    &=0.\qedhere
  \end{align*}
\end{proof}

Applying \cref{lem:propagation}, we obtain a system of meromorphic functions on $\oC$:
\[
  S_{\varphi}\bracket{v_3}{(a^{\bullet},\pp_{\bullet})(v,\qq)}{v_2}:=\braket*{\xi_{\pp_{\bullet}}^{-1}\varphi_{\qq}}{v_3\otimes a^{\bullet}\otimes v\otimes v_2}.
\]
Recall that we have defined the vector space $\Cor[\Sigma_{1}(N^3, M^1, M^2)]$ of \emph{$g$-twisted correlation functions} in \cite[Definition 2.21]{PartI}.  
\begin{proposition}
  The system $S_{\varphi}$ belongs to $\Cor[\Sigma_{1}(N^3, M^1, M^2)]$. 
\end{proposition}
\begin{proof}
  The \emph{truncation property} is evident by the definition. 
  The \emph{$\vo{L}{-1}$-property} follows from \cref{lem:horizontal}. 
  The \emph{locality} and the \emph{vacuum property} follow from the construction of $\xi^{-1}$.
  Using the local coordinates near $\infty$, $\qq$, and $0$, we have
  \begin{align*}
    \MoveEqLeft
    \iota_{\pp_1=\infty}\cdots\iota_{\pp_n=\infty}S_{\varphi}\bracket{v_3}{\cfbseq(v,\qq)}{v_2}\\
    &=\braket*{\varphi}{Y_{N^3}(\vartheta(a^1),z_1^{-\frac{1}{T}})\cdots Y_{N^3}(\vartheta(a^n),z_n^{-\frac{1}{T}})v_3\otimes v\otimes v_2},\\
    \MoveEqLeft
    \iota_{\pp_1=\qq}\cdots\iota_{\pp_n=\qq}S_{\varphi}\bracket{v_3}{\cfbseq(v,\qq)}{v_2}\\
    &=\braket*{\varphi}{v_3\otimes Y_{M^1}(a^1,z_1-w)\cdots Y_{M^1}(a^n,z_n-w)v\otimes v_2},\\
    \MoveEqLeft
    \iota_{\pp_1=0}\cdots\iota_{\pp_n=0}S_{\varphi}\bracket{v_3}{\cfbseq(v,\qq)}{v_2}\\
    &=\braket*{\varphi}{v_3\otimes v\otimes Y_{M^2}(a^1,z_1)\cdots Y_{M^2}(a^n,z_n)v_2}.
  \end{align*}
  Then  the \emph{homogenus property}, the \emph{associativity}, and the \emph{generating properties} are evident. 
\end{proof}

Now, we prove the following theorem:
\begin{theorem}\label{thm:Cfb=Cor}
  Let $M^1$ (resp. $M^2$ and $N^3$) be an admissible \emph{untwisted} (resp. \emph{$g$-twisted} and \emph{$g^{-1}$-twisted}) module of conformal weight $h_1$ (resp. $h_2$ and $h_3$). 
  Put $h=h_1+h_2-h_3$. 
  Then  we have the following isomorphism of vector spaces:
  \[
    \Cfb[\Sigma_{1}(N^3, M^1, M^2)]\longrightarrow\Cor[\Sigma_{1}(N^3, M^1, M^2)],\qquad
    \varphi\longmapsto S_{\varphi}.
  \]
\end{theorem}
\begin{proof}
  Given a system of correlation functions $S\in\Cor[\Sigma_{1}(N^3, M^1, M^2)]$, we can obtain a conformal block $\varphi_{\qq}(v_3\otimes v\otimes v_2):=S\bracket*{v_3}{(v,\qq)}{v_2}$ associated to the datum $\Sigma_{\qq}(N^3, M^1, M^2)$ for each $\qq\in\Cc$. 
  Indeed, by \cref{eg:LgV}, it suffices to show $\braket*{\varphi_{\qq}}{\loo{a}[r]{m}{n}.(v_3\otimes v\otimes v_2)}=0$ for all homogeneous $a\in V^{r}$, $m\in\Z$, and $n\in\N$.
  By \cref{tab:EXPY}, we have
  \begin{align*}
    \MoveEqLeft
    \Res_{\infty}\braket*{\varphi_{\qq}}{\cY_{N^3,\infty}(\loo{a}[r]{m}{n})v_3\otimes v\otimes v_2} \\
    &= 
    -T\sum_{j\ge0}\binom{-n}{j}(-w)^jS\bracket*{\theta(\lo{a}{\frac{r}{T}+m-n-j})v_3}{(v,\qq)}{v_2} \\
    \overset{\ast}&{=}
    \Res_{\pp=\infty}S\bracket*{v_3}{(a,\pp)(v,\qq)}{v_2}z^{\frac{r}{T}+m}(z-w)^{-n}\d{z},\\
    \MoveEqLeft
    \Res_{0}\braket*{\varphi_{\qq}}{v_3\otimes v\otimes \cY_{M^2,0}(\loo{a}[r]{m}{n})v_2} \\
    &= 
    T\sum_{j\ge0}\binom{-n}{j}(-w)^{-n-j}S\bracket*{v_3}{(v,\qq)}{\vo{a}{\frac{r}{T}+m+j}v_2} \\
    \overset{\ast}&{=}
    \Res_{\pp=0}S\bracket*{v_3}{(a,\pp)(v,\qq)}{v_2}z^{\frac{r}{T}+m}(z-w)^{-n}\d{z},\\
    \MoveEqLeft
    \Res_{\qq}\braket*{\varphi_{\qq}}{v_3\otimes \cY_{M^1,\qq}(\loo{a}[r]{m}{n})v\otimes v_2} \\
    &= 
    \sum_{j\ge0}\binom{\frac{r}{T}+m}{j}w^{\frac{r}{T}+m-j}S\bracket*{v_3}{(\vo{a}{n+j}v,\qq)}{v_2} \\
    \overset{\ast\ast}&{=}
    \Res_{\pp=\qq}S\bracket*{v_3}{(a,\pp)(v,\qq)}{v_2}z^{\frac{r}{T}+m}(z-w)^{-n}\d{z},
  \end{align*}
  where $\ast$'s follow from the \emph{generating properties} of $S$, while $\ast\ast$ follows from the \emph{associativity}. 
  Then the desired statement follows from \cref{eg:RSF}.

  It remains to show that the conformal blocks $\varphi_{\qq}$ ($\qq\in\Cc$) form a \emph{horizontal} section and hence corresponding to a unique $\varphi\in\Cfb[\Sigma_{1}(N^3, M^1, M^2)]$. This boils down to show 
  \[
    w^{h}\d\left(w^{-h}S\bracket*{v_3}{(v,\qq)}{v_2}\right) = S\bracket*{v_3}{(\vo{L}{-1}v,\qq)}{v_2}\d{w},
  \]
  which is precisely the \emph{$\vo{L}{-1}$-property} of $S$.
\end{proof}

\subsection{Relation to twisted intertwining operators}
In the rest of this section, we relate \emph{3-pointed conformal blocks} with \emph{twisted intertwining operators}. 
The datum we are interested in is
\begin{equation}\label{eq:def:datum1}
  \Sigma_{w}\left((M^3)', M^1, M^2\right):=\left(\x\colon\oC_{w}\to\PP^1, \infty, w, 0, (M^3)', M^1, M^2\right),
\end{equation}
where $M^1$ (resp. $M^2$ and $M^3$) is an admissible \emph{$g_1$-twisted} (resp. \emph{$g_2$-twisted} and \emph{$g_3$-twisted}) module of conformal weight $h_1$ (resp. $h_2$ and $h_3$), and $\x\colon\oC_w\to\PP^1$ is the obicurve in \cref{eg:3ptP1} with Deck generators $g_1$, $g_2$, and $g_3^{-1}$ at the branch points $w$, $0$, and $\infty$ respectively. 
This setup corresponds to the case where $g_1,g_2,g_3$ have the same order $T$ and we adopt the convention that $g_1=g_2^\epsilon$.
The situation where $g_1=\id$ follows by a similar argument with the setup in \cref{sec:Cor}.

Recall that these obicurves $\x\colon\oC_w\to\PP^1$ for various $w\in\C^\times$ can be organized into a family $\mathfrak{C}\to\PP^1$. 
Then the spaces $((\cM^3)'_{\infty}\otimes\cM^1_{w}\otimes\cM^2_{0})^{\ast}$ for various $w\in\C^\times$ form a bundle $((\cM^3)'_{\infty}\otimes\cM^1_{\square}\otimes\cM^2_{0})^{\ast}$ on the affine open subset $\C^\times\subset\PP^1$, which can be trivialized as $((M^3)'\otimes M^1\otimes M^2)^{\ast}[w^{\pm1}]$. 
Hence we can organize conformal blocks for various $w\in\C^\times$ into a section over $\C^\times$ via the formula:  
\begin{equation}\label{eq:CfbFamily2}
  \braket*{\varphi_{w}}{v_3\otimes v\otimes v_2} := 
  \braket*{\varphi_{1}}{w^{\vo{L}{0}-h_3}v_3\otimes w^{-\vo{L}{0}+h_1}v\otimes w^{-\vo{L}{0}+h_2}v_2}.  
\end{equation}
\begin{lemma}\label{lem:horizontal2}
  The section $\varphi_{\bullet}$ is horizontal in the sense that it is annihilated by the following linear map $\nabla\colon((\cM^3)'_{\infty}\otimes\cM^1_{\square}\otimes\cM^2_{0})^{\ast}\to((\cM^3)'_{\infty}\otimes\cM^1_{\square}\otimes\cM^2_{0})^{\ast}\otimes\Omega$
  \[
    \nabla = w^{h}\d\circ w^{-h} - \vo{L}{-1}^{\cM^1}\d{w},
  \]
  where $\vo{L}{-1}^{\cM^1}$ is the $\vo{L}{-1}$-operator on the bundle $\cM^1_{\square}$.
\end{lemma}
\begin{proof}
  Similar to \cref{lem:horizontal}.  
\end{proof}

\begin{proposition}
  The following formula defines an intertwining operator of type $\Fusion$.
\[
  \braket*{v_3'}{I(v,w)v_2}:=w^{-h}\braket*{\varphi_w}{v_3'\otimes v\otimes v_2}.
\]
\end{proposition}
\begin{proof}
  The \emph{truncation property} \labelcref{axiom:IO-trunc} is evident. 
  The \emph{$\vo{L}{-1}$-derivative property} \labelcref{axiom:IO-der} follows from \cref{lem:horizontal2}.
  
  For any homogeneous $a\in V^r$, $m\in\Z$, and $n\in\N$, we have 
  \begin{align*}
      \MoveEqLeft
      \braket*{v_3'}{Y_{M^3}(a,z)I(v,w)v_2}
      \iota_{z^{-1}}\frac{z^{\frac{r}{T}+m}w^h}{(z-w)^{\epsilon\frac{r}{T}+n}} \\
      &= 
      \braket*{Y_{(M^3)'}(\vartheta(a),z^{-1})v_3'}{I(v,w)v_2}
      \iota_{z^{-1}}\frac{z^{\frac{r}{T}+m}w^h}{(z-w)^{\epsilon\frac{r}{T}+n}}\\
      &=\iota_{z^{-1}}\braket*{\varphi_w}{\cY_{(M^3)',\infty}(\loo{a}[r]{m}{n})v_3'\otimes v\otimes v_2},\\
      \MoveEqLeft
      \braket*{v_3'}{I(Y_{M^1}(a,z-w)v,w)v_2}
      \iota_{z-w}\frac{z^{\frac{r}{T}+m}w^h}{(z-w)^{\epsilon\frac{r}{T}+n}}  \\
      &=\iota_{z-w}\braket*{\varphi_w}{v_3'\otimes \cY_{M^1,w}(\loo{a}[r]{m}{n})v\otimes v_2},\\
      \MoveEqLeft
      \braket*{v_3'}{I(v,w)Y_{M^2}(a,z)v_2}
      \iota_{z}\frac{z^{\frac{r}{T}+m}w^h}{(z-w)^{\epsilon\frac{r}{T}+n}}  \\
      &=\iota_{z}\braket*{\varphi_w}{v_3'\otimes v\otimes \cY_{M^2,0}(\loo{a}[r]{m}{n})v_2}.
  \end{align*}
  Since $\varphi_w\in\Cfb[\Sigma_w]$, the local functions $\braket*{\varphi_w}{\cY_{(M^3)',\infty}(\loo{a}[r]{m}{n})v_3'\otimes v\otimes v_2}$, $\braket*{\varphi_w}{v_3'\otimes \cY_{M^1,w}(\loo{a}[r]{m}{n})v\otimes v_2}$, and $\braket*{\varphi_w}{v_3'\otimes v\otimes \cY_{M^2,0}(\loo{a}[r]{m}{n})v_2}$ extends to the same meromorphic function with possible poles at $\infty,w,0$. 
  Then the \emph{Twisted Jacobi identity} \labelcref{axiom:IO-Jac} follows. 
\end{proof}

Now, we prove the following theorem:
\begin{theorem}\label{thm:IO}
    Let $M^k$ be a $g_k$-twisted $V$-module of conformal weight $h_k$ for $k=1, 2, 3$. Set $h=h_1+h_2-h_3$. Then, as linear spaces,
    \[\Cfb\left(\Sigma_w\left((M^3)', M^1, M^2\right)\right)\cong \Fusion.\]
\end{theorem}
\begin{proof}
    Let $I$ be an intertwining operator of type $\Fusion$. 
    Then, by the \emph{truncation property} \labelcref{axiom:IO-trunc}, the following formula defines a linear functional $\varphi_w$ on the vector space $(M^3)'\otimes M^1\otimes M^2$:
    \[
      \braket*{\varphi_w}{v_3'\otimes v\otimes v_2}:=\braket*{v_3'}{I(v,w)v_2}w^h.
    \]
    To show that $\varphi_w$ is invariant under $\L_{\PP^1\setminus\Set{\infty,w,0}}(\cV^G)$, by \cref{eg:3ptsLV}, it suffices to $\braket*{\varphi_w}{\loo{a}[r]{m}{n}.(v_3'\otimes v\otimes v_2)}=0$ for all homogeneous $a\in V^{r}$, $m\in\Z$, and $n\in\N$.
    Indeed, we have\footnote{Note that $\Res_w$ means the residue at the point $w\in\PP^1$.} 
  \begin{align*}
    \MoveEqLeft
    \Res_{\infty}\braket*{\varphi_w}{\cY_{(M^3)',\infty}(\loo{a}[r]{m}{n})v_3'\otimes v\otimes v_2} \\
    &= 
    \braket*{\Res_{\infty}\cY_{(M^3)',\infty}(\loo{a}[r]{m}{n})v_3'}{I(v,w)v_2}w^h\\
    &= 
    \Res_{z=\infty}\frac{Tz^{\frac{r}{T}+m}}{(z-w)^{\epsilon\frac{r}{T}+n}}\braket*{Y_{(M^3)'}(\vartheta(a),z^{-1}))v_3'}{I(v,w)v_2}w^h\\
    &= 
    \Res_{z=\infty}\frac{Tz^{\frac{r}{T}+m}}{(z-w)^{\epsilon\frac{r}{T}+n}}\braket*{v_3'}{Y_{M^3}(a,z)I(v,w)v_2}w^h,\\
    \MoveEqLeft
    \Res_{w}\braket*{\varphi_w}{v_3'\otimes \cY_{M^1,w}(\loo{a}[r]{m}{n})v\otimes v_2} \\
    &= 
    \Res_{w}\braket*{v_3'}{I\left(\cY_{M^1,w}(\loo{a}[r]{m}{n})v,w\right)v_2}w^h \\
    &= 
    \Res_{z=w}\frac{Tz^{\frac{r}{T}+m}}{(z-w)^{\epsilon\frac{r}{T}+n}}\braket*{v_3'}{I(Y_{M^1}(a,z-w)v,w)v_2}w^h,\\
    \MoveEqLeft
    \Res_{0}\braket*{\varphi_w}{v_3\otimes v\otimes \cY_{M^2,0}(\loo{a}[r]{m}{n})v_2} \\
    &= 
    \Res_{0}\braket*{v_3'}{I(v,w)Y_{M^2}(a,z)v_2}w^h \\
    &= 
    \Res_{z=0}\frac{Tz^{\frac{r}{T}+m}}{(z-w)^{\epsilon\frac{r}{T}+n}}\braket*{v_3'}{I(v,w)Y_{M^2}(a,z)v_2}w^h.
  \end{align*}
  Then the desired equality $\braket*{\varphi_w}{\loo{a}[r]{m}{n}.(v_3'\otimes v\otimes v_2)}=0$ follows from the \emph{Twisted Jacobi identity} \labelcref{axiom:IO-Jac}. 
  
  It remains to show that the conformal blocks $\varphi_{w}$ ($w\in\C^\times$) form a \emph{horizontal} section and hence corresponding to a unique $\varphi\in\Cfb[\Sigma_{1}((M^3)', M^1, M^2)]$. This boils down to show 
  \[
    w^{h}\d\braket*{v_3'}{I(v,w)v_2} = \braket*{v_3'}{I(\vo{L}{-1}v,w)v_2}w^h\d{w},
  \]
  which is precisely the \emph{$\vo{L}{-1}$-derivative property} \labelcref{axiom:IO-der}.
\end{proof}

\section{Restrictions of conformal blocks}\label{sec:RestrictedCfb}
In this section, we introduce the notion of \emph{twisted restricted conformal blocks}. 
We will explore its relations with the notions of \emph{$g$-twisted restricted correlation functions} and \emph{$3$-pointed $g$-twisted restrict conformal blocks} introduced in \cite{PartI} and then study the restriction map from the space of twisted conformal blocks to the space of twisted restricted conformal blocks.

\subsection{Twisted restricted conformal blocks}

\subsubsection*{The bottom level of a twisted module}
Let $\qq$ be a branch point with the Deck generator $g$. 
Let $M$ be an admissible \emph{$g$-twisted} $V$-module of conformal weight $h$. 
Then by conjugating a degree shifting by $h$, the bottom level $U=M(0):=M_h$ of $M$ is preserved by the change of special local coordinates. 
Hence, $U$ gives rise to a $\kappa_{\qq}$-subspace $\cU_{\qq}$ of $\cM_{\qq}$.
\begin{lemma}\label{lm:lm6.1}
  With the aforementioned assumptions, $\cU_{\qq}$ is a $\L_{\qq}(\cV^G)_{0}$-module.
\end{lemma}
\begin{proof}
  The Lie algebra $\L_{\qq}(\cV^G)$ is  topologically spanned by the elements $\lo{a}{\frac{r}{T}+m}$, where $a\in V_g^r$ and $m\in\Z$. 
  Such an element $\lo{a}{\frac{r}{T}+m}$ acts on $M$ as the linear operator $\vo{a}{\frac{r}{T}+m}\in\End(M)$, and by \cref{def:admissible}, 
  $\vo{a}{\frac{r}{T}+m}$ preserves $M(0)$ if $\frac{r}{T}+m\ge \wt{a}-1$ and annihilates $M(0)$ if $\frac{r}{T}+m> \wt{a}-1$. 
  On the other hand, $v_{\qq}(\lo{a}{\frac{r}{T}+m})=\frac{r}{T}+m$. 
  Hence, we see that $\cU_{\qq}$ is preserved by $\L_{\qq}(\cV^G)_{\le0}$ and annihilated by $\L_{\qq}(\cV^G)_{<0}$. 
\end{proof}

Conversely, let $U$ be a continuous $\L_{g}(V)_{0}$-module. 
Then all the Virasoro operators $\vo{L}{k}$ ($k>0$) annihilate $U$ while $\vo{L}{0}$ preserves $U$. 
Hence, $U$ gives rise to a $\kappa_{\qq}$-space $\cU_{\qq}$ by twisting the choice of special local coordinates. 
It is then evident that $\cU_{\qq}$ is a $\L_{\qq}(\cV^G)_{0}$-module. 
We may further extend the action of $\L_{\qq}(\cV^G)_{0}$ on $U$ to one of $\L_{\qq}(\cV^G)_{\le 0}$ by letting $\L_{\qq}(\cV^G)_{<0}$ act trivially. 
We may organize this action $\rho$ into a formal series
\[
  Y_U(a,z):=\sum_{m\ge \wt{a}-1}\rho(\lo{a}{\frac{r}{T}+m})z^{-(\frac{r}{T}+m)-1},
\]
which is indeed either $\rho(\lo{a}{\wt a-1})z^{-\wt a}$ (if $r=0$) or $0$ (if $r\neq0$).  
Then we obtain an $\End\cU_{\qq}$-valued section $\cY_{U,\qq}$ of $\cV(-(\vo{L}{0}-1)T\qq)^{\ast}$ on $D_{\qq}$ via the formula:
\[
  \braket*{\varphi}{\cY_{U,\qq}(\spa)\nu} = \braket*{\iota_{z}\varphi}{Y_U(\iota_{z}(\spa),z)\iota_{z}\nu},
\]
for all $\varphi\in\cU_{\qq}^{\ast}$, $\nu\in\cU_{\qq}$ and all regular sections $\spa$ of $\cV(-(\vo{L}{0}-1)T\qq)$.  

\subsubsection*{Twisted restricted conformal blocks}
By restricting the $g_{\diamond}$-twisted modules $N^{\diamond}$ attached to the branch points $\qq_\diamond$ in the domain of a twisted conformal block in \cref{def:Cfb} to the bottom levels, we obtain the following notion.
\begin{definition}\label{def:Cfb_res}
  Consider a datum $\Sigma=(\x\colon\tilX\to\X,\pp_{\bullet},\qq_{\diamond},M^{\bullet},U^{\diamond})$ of 
  \begin{itemize}
    \item a totally ramified orbicurve $\x\colon\tilX\to\X$, 
    \item distinct points $\pp_{\bullet}=(\pp_{1},\cdots,\pp_{m})$ on the unramified locus $\Xc$, 
    \item all the branch points $\qq_{\diamond}=(\qq_{1},\cdots,\qq_{n})$, each provides a \emph{Deck generator} $g_{i}$,
    \item admissible \emph{untwisted} modules $M^{\bullet}=(M^1,\cdots,M^m)$, one for each $\pp_{i}$, and
    \item $A_{g_{\diamond}}(V)$-module $U^{\diamond}=(U^1,\cdots,U^n)$, one for each $\qq_{i}$.
  \end{itemize}
  Then the \textbf{space of (twisted) restricted coninvariants associated to the datum $\Sigma$} is the quotient space 
  \begin{equation}\label{eq:defconinv_res}
    (\cM^{\bullet}_{\pp_{\bullet}}\otimes\cU^{\diamond}_{\qq_{\diamond}})_{\L^{\circ}_{\X\setminus\pp_{\bullet}}(\cV^G)_{\le0}}:= 
    (\cM^{\bullet}_{\pp_{\bullet}}\otimes\cU^{\diamond}_{\qq_{\diamond}})/ 
    \L^{\circ}_{\X\setminus\pp_{\bullet}}(\cV^G)_{\le0}.(\cM^{\bullet}_{\pp_{\bullet}}\otimes\cU^{\diamond}_{\qq_{\diamond}}).
  \end{equation}
  Its linear dual is called the \textbf{space of (twisted) restricted conformal blocks associated to the datum $\Sigma$} and is denoted by $\Cfb[\Sigma]$. 
  Equivalently, a \emph{(twisted) restricted conformal block associated to the datum $\Sigma$} is a linear functional on $\cM^{\bullet}_{\pp_{\bullet}}\otimes\cU^{\diamond}_{\qq_{\diamond}}$ that is $\L^{\circ}_{\X\setminus\pp_{\bullet}}(\cV^G)_{\le0}$-invariant. 
\end{definition}
\begin{remark}\label{rem:Zhu}
  Recall that there is an epimorphism from the degree-zero subalgebra $\L_{g_{\diamond}}(V)_0$ of the \emph{$g_{\diamond}$-twisted ancillary Lie algebra} to the \emph{$g_{\diamond}$-twisted Zhu's algebra} $A_{g_{\diamond}}(V)$. 
  Hence, any $A_{g_{\diamond}}(V)$-module is an $\L_{g_{\diamond}}(V)_0$-module.
\end{remark}

\begin{lemma}[Propagation of vacua]\label{lem:propagation_res}
  There is a canonical isomorphism
  \[
    \xi_{\pp}\colon\Cfb[\x\colon\tilX\to\X,\pp,\pp_{\bullet},\qq_{\diamond},V,M^{\bullet},U^{\diamond}]\cong 
    \Cfb[\x\colon\tilX\to\X,\pp_{\bullet},\qq_{\diamond},M^{\bullet},U^{\diamond}],
  \]
  given by restricting to $\vac\otimes\cM^{\bullet}_{\pp_{\bullet}}\otimes\cU^{\diamond}_{\qq_{\diamond}}$. 
\end{lemma}
\begin{proof}
  The proof is similar to \cref{lem:propagation}. 
  First, by applying the \emph{\nameref{lem:SRT}} to the definition above, we see that a linear functional $\varphi$ on $\cM^{\bullet}_{\pp_{\bullet}}\otimes\cU^{\diamond}_{\qq_{\diamond}}$ is a conformal block if and only if for any $\mu_{\bullet}\in\cM^{\bullet}_{\pp_{\bullet}}$ and $\upsilon_{\diamond}\in\cU^{\diamond}_{\qq_{\diamond}}$, the followings: 
  \[
    \braket*{\varphi}{\cdots\otimes\cY_{M^i,\pp_i}\mu_{i}\otimes\cdots\otimes\upsilon_{\diamond}}\txand
    \braket*{\varphi}{\mu_{\bullet}\otimes\cdots\otimes\cY_{U^i,\qq_{i}}\upsilon_{i}\otimes\cdots},
  \]
  can be extended to the \emph{same} meromorphic section of $(\cV(-(\vo{L}{0}-1)\dQ)^G)^{\ast}$ on $\tilX$ with possible poles at listings of $\pp_{\bullet}$. 
  Denote this section by $\varphi_{\mu_{\bullet}\otimes\upsilon_{\diamond}}$. 
  Note that at the point $\pp$, we have $\cV(-(\vo{L}{0}-1)\dQ)_{\pp}=\cV_{\pp}$ as subspaces of $V\widehat\otimes\cK_{\pp}$.
  Then  the assignment 
  \[
    \spa\otimes\mu_{\bullet}\otimes\upsilon_{\diamond}\in \cV_{\pp}\otimes\cM^{\bullet}_{\pp_{\bullet}}\otimes\cU^{\diamond}_{\qq_{\diamond}}
    \longmapsto
    \varphi_{\mu_{\bullet}\otimes\upsilon_{\diamond}}(\spa)
  \]
  gives a restricted conformal block $\xi_{\pp}^{-1}\varphi$. This gives the inverse of $\xi_{\pp}$. 
\end{proof}

\subsection{Relation to twisted restricted correlation functions}
Now, we establish the relationship between \cref{def:Cfb_res} and the notion of \emph{$g$-twisted restricted correlation functions} in \cite[Definition 3.5]{PartI}. 
The datum we are interested in is
\begin{equation}\label{eq:def:datum_res}
  \Sigma_{\qq}(U^3, M^1, U^2):=(\x\colon\oC\to\PP^1, \infty, \x(\qq), 0, U^3, M^1, U^2),
\end{equation}
where $M^1$ is an admissible untwisted module of conformal weight $h_1$, and $U^2$ (resp. $U^3$) is a left (resp. right) $A_g(V)$-module where $[\upomega]$ acts as $h_2\id$ (resp. $h_3\id$). Put $h=h_1+h_2-h_3$. 
We organize restricted conformal blocks for various $\qq\in\Cc$ into a \emph{horizontal} (in the sense of \cref{lem:horizontal}) section of the bundle $(\cU^3_{\infty}\otimes\cM^1_{\square}\otimes\cU^2_{0})^{\ast}$ over $\Cc$ via the formula \labelcref{eq:CfbFamily}. 

Next, applying \cref{lem:propagation_res}, we obtain a system of meromorphic functions on $\oC$ (the relation between $\xi_{\pp_{\bullet}}^{-1}$ and $\xi_{\x(\pp_{\bullet})}^{-1}$ is as in \cref{rem:xi_p}):
\[
  S_{\varphi}\bracket{u_3}{(a^{\bullet},\pp_{\bullet})(v,\qq)}{u_2}:=\braket*{\xi_{\pp_{\bullet}}^{-1}\varphi_{\qq}}{u_3\otimes a^{\bullet}\otimes v\otimes u_2}
\]
Recall that we have defined the space $\Cor[\Sigma_{1}(U^3, M^1, U^2)]$ of twisted restricted correlation functions in \cite[Definition 3.5]{PartI}.
\begin{proposition}
  The system $S_{\varphi}$ belongs to $\Cor[\Sigma_{1}(U^3, M^1, U^2)]$.  
\end{proposition}
\begin{proof}
  The \emph{monomial property} is evident by the definition of $\varphi_{\qq}$ (noticing that the $\vo{L}{0}$ acts as constant scales on $U^2$ and $U^3$). 
  The \emph{$\vo{L}{-1}$-property} follows from \cref{lem:horizontal}. 
  The \emph{locality} and \emph{vacuum property} follow from the construction of $\xi^{-1}$.
  Using the local coordinates near $\qq$, we have
  \begin{align*}
    \MoveEqLeft
    \iota_{\pp_1=\qq}\cdots\iota_{\pp_n=\qq}S_{\varphi}\bracket{u_3}{\cfbseq(v,\qq)}{u_2}\\
    &=\braket*{\varphi}{u_3\otimes Y_{M^1}(a^1,z_1-w)\cdots Y_{M^1}(a^n,z_n-w)v\otimes u_2}.
  \end{align*}
  Then the \emph{associativity} is evident. 
  
  For any homogeneous $a\in V_g^r$ and any $m\in\Z$ such that $\frac{r}{T}+m\ge\wt{a}-1$, we have 
  \[
    S_{\varphi}\bracket{u_3}{(a,\pp)\cdots}{u_2}z^{\frac{r}{T}+m}\d{z} 
    =\braket*{\xi_{\pp}^{-1}\xi_{\pp_\bullet}^{-1}\varphi_{\qq}}{u_3\otimes\lo{a}{\frac{r}{T}+m}\otimes\cdots\otimes u_2}.
  \]
  Note that $\lo{a}{\frac{r}{T}+m}\in\L_{0}(\cV^G)_{\le0}$. Hence, the right-hand side descends to a meromorphic $1$-form on $\PP^1$ with possible poles at $\infty$, $0$, $\x(\qq)$, and $\x(\pp_{\bullet})$. 
  Multiplying both sides by $z^{-m}$, we see that $S_{\varphi}\bracket{u_3}{(a,\pp)\cdots}{u_2}z^{\frac{r}{T}}\d{z}$ descends to a meromorphic $1$-form on $\PP^1$ and the \emph{homogenus property} follows by an induction on the set of points $\pp_{\bullet}$. 
  
  When $m>\floor{\wt{a}-1-\frac{r}{T}}$, we have $\lo{a}{\frac{r}{T}+m}\in\L_{0}(\cV^G)_{<0}$. 
  If this is the case, 
  \begin{align*}
    \MoveEqLeft[1]
    \Res_{\pp=0}S_{\varphi}\bracket{u_3}{(a,\pp)\cdots}{u_2}z^{\frac{r}{T}+m}\d{z}\\
    &=\Res_{\pp=0}\braket*{\xi_{\pp}^{-1}\xi_{\pp_\bullet}^{-1}\varphi_{\qq}}{u_3\otimes\lo{a}{\frac{r}{T}+m}\otimes\cdots\otimes u_2}\\
    &=\braket*{\xi_{\pp_\bullet}^{-1}\varphi_{\qq}}{u_3\otimes\cdots\otimes \lo{a}{\frac{r}{T}+m}\cdot u_2}=0.  
  \end{align*}
  On the other hand, supposing $m\le\floor{\wt{a}-1-\frac{r}{T}}$ and viewing $\lo{a}{\frac{r}{T}+m}$ as a $1$-form on $C$, its image $\theta(\lo{a}{\frac{r}{T}+m})$ under the transition map belongs to $\L_{\infty}(\cV^G)_{\le0}$. 
  Then  
  \begin{align*}
    \MoveEqLeft[1]
    \Res_{\pp=\infty}S_{\varphi}\bracket{u_3}{(a,\pp)\cdots}{u_2}z^{\frac{r}{T}+m}\d{z}\\
    &=\Res_{\pp=\infty}\braket*{\xi_{\pp}^{-1}\xi_{\pp_\bullet}^{-1}\varphi_{\qq}}{u_3\otimes\lo{a}{\frac{r}{T}+m}\otimes\cdots\otimes u_2}\\
    &=\braket*{\xi_{\pp_\bullet}^{-1}\varphi_{\qq}}{\theta(\lo{a}{\frac{r}{T}+m})\cdot u_3\otimes a^{\bullet}\otimes v\otimes u_2}\\
    &=\begin{dcases*}
      S\bracket*{u_3o(a)}{\cdots}{u_2}z^{-\wt a} & if $\frac{r}{T}+m=\wt{a}-1$,\\
      0 & otherwise.
    \end{dcases*}
  \end{align*}
  Therefore, by \cref{eg:RSF}, expanding $S_{\varphi}\bracket*{u_3}{(a,\pp)\cdots}{u_2}$ at $\pp=0$ yields
  \begin{align*}
    \MoveEqLeft\iota_{\pp=0}S_{\varphi}\bracket*{u_3}{(a,\pp)\cdots}{u_2}\\
    &= 
    \sum_{m\le\floor{\wt{a}-1-\frac{r}{T}}}z^{-1-m-\frac{r}{T}}\bigg(S_{\varphi}\bracket*{\theta(\lo{a}{m+\frac{r}{T}})\cdot u_3}{\cdots}{u_2}\\
    &\qquad -\sum_{k=1}^{n}\sum_{i\ge 0}\binom{m+\frac{r}{T}}{i}z_k^{m+\frac{r}{T}-i}S_{\varphi}\bracket*{u_3}{\cdots(\vo{a}{i}a^{k},\pp_{k})\cdots}{u_2} \\
    &\qquad -\sum_{i\ge 0}\binom{m+\frac{r}{T}}{i}w^{m+\frac{r}{T}-i}S_{\varphi}\bracket*{u_3}{\cdots(\vo{a}{i}a^{k},\qq)}{u_2}\bigg) \\
    &= 
    S_{\varphi}\bracket*{u_3o(a)}{\cdots}{u_2}z^{-\wt a} \\
    &\phantom{=}  +\sum_{k=1}^{n}\sum_{i\ge 0} 
    \iota_{\pp=0}F_{\wt a-1+\delta(r)+\frac{r}{T},i}(\pp,\pp_{k}) 
    S_{\varphi}\bracket*{u_3}{\cdots(\vo{a}{i}a^{k},\pp_{k})\cdots}{u_2} \\
    &\phantom{=}  +\sum_{i\ge 0} 
    \iota_{\pp=0}F_{\wt a-1+\delta(r)+\frac{r}{T},i}(\pp,\qq) 
    S_{\varphi}\bracket*{u_3}{\cdots(\vo{a}{i}v,\qq)}{u_2}, 
  \end{align*}
  where the functions $F_{\wt a-1+\delta(r)+\frac{r}{T},i}$ are introduced in \cite[(3.5)]{PartI}.
  This proves the \emph{recursive formula for $U^3$} by the injectivity of $\iota_{\pp=0}$. 
  Finally, expanding $S_{\varphi}\bracket*{u_3}{(a,\pp)\cdots}{u_2}$ at $\pp=\infty$ yields another \emph{recursive formula}.
\end{proof}

Now, we state our next theorem:
\begin{theorem}\label{thm:Cfb=Cor_res}
  Let $M^1$ be an admissible untwisted module of conformal weight $h_1$, and $U^2$ (resp. $U^3$) a left (resp. right) $A_g(V)$-module where $[\upomega]$ acts as $h_2\id$ (resp. $h_3\id$). Put $h=h_1+h_2-h_3$. 
  Then  we have the following isomorphism of vector spaces:
  \[
    \Cfb[\Sigma_{1}(U^3, M^1, U^2)]\longrightarrow\Cor[\Sigma_{1}(U^3, M^1, U^2)],\qquad
    \varphi\longmapsto S_{\varphi}.
  \]
\end{theorem}
\begin{proof}
  Given a system of correlation functions $S\in\Cor[\Sigma_{1}(U^3, M^1, U^2)]$, we can obtain conformal blocks $\varphi_{\qq}(u_3\otimes v\otimes u_2):=S\bracket*{u_3}{(v,\qq)}{u_2}$. 
  The proof is similar to \cref{thm:Cfb=Cor}. 
  By \cref{eg:LgV}, it suffices to show that, for all homogeneous $a\in V^{r}$, $n\in\Z$, and $m\in\N$, satisfying $\frac{r}{T}+n-m\le\wt{a}-1\le\frac{r}{T}+n$, 
  the element $\loo{a}[r]{m}{n}.(u_3\otimes v\otimes u_2)$ is annihilated by $\varphi_{\qq}$.
  Indeed, we have
  \begin{align*}
    \MoveEqLeft
    \Res_{\infty}\braket*{\varphi_{\qq}}{\cY_{U^3,\infty}(\loo{a}[r]{m}{n})u_3\otimes v\otimes u_2} \\
    &= 
    -T\sum_{j\ge0}\binom{-m}{j}(-w)^jS\bracket*{\theta(\lo{a}{\frac{r}{T}+n-m-j})u_3}{(v,\qq)}{u_2} \\
    \overset{\ast}&{=}
    \Res_{\pp=\infty}S\bracket*{u_3}{(a,\pp)(v,\qq)}{u_2}z^{\frac{r}{T}+n}(z-w)^{-m}\d{z},\\
    \MoveEqLeft
    \Res_{0}\braket*{\varphi_{\qq}}{u_3\otimes v\otimes \cY_{U^2,0}(\loo{a}[r]{m}{n})u_2} \\
    &= 
    T\sum_{j\ge0}\binom{-m}{j}(-w)^{-m-j}S\bracket*{u_3}{(v,\qq)}{\vo{a}{\frac{r}{T}+n+j}u_2} \\
    \overset{\ast}&{=}
    \Res_{\pp=0}S\bracket*{u_3}{(a,\pp)(v,\qq)}{u_2}z^{\frac{r}{T}+n}(z-w)^{-m}\d{z},\\
    \MoveEqLeft
    \Res_{\qq}\braket*{\varphi_{\qq}}{u_3\otimes \cY_{M^1,\qq}(\loo{a}[r]{m}{n})v\otimes u_2} \\
    &= 
    \sum_{j\ge0}\binom{\frac{r}{T}+n}{j}w^{\frac{r}{T}+n-j}S\bracket*{u_3}{(\vo{a}{m+j}v,\qq)}{u_2} \\
    \overset{\ast\ast}&{=}
    \Res_{\pp=\qq}S\bracket*{u_3}{(a,\pp)(v,\qq)}{u_2}z^{\frac{r}{T}+n}(z-w)^{-m}\d{z},
  \end{align*}
  where $\ast$'s follow from the \emph{recursive formulas} of $S$ and $\ast\ast$ follows from the \emph{associativity}. 
  Then  $\braket*{\varphi_{\qq}}{\loo{a}[r]{m}{n}.(u_3\otimes v\otimes u_2)}=0$ follows from \cref{eg:RSF}.

  It remains to show that the conformal blocks $\varphi_{\qq}$ ($\qq\in\Cc$) form a \emph{horizontal} section and hence corresponding to a unique $\varphi\in\Cfb[\Sigma_{1}(U^3, M^1, U^2)]$. This boils down to show 
  \[
    w^{h}\d\left(w^{-h}S\bracket*{u_3}{(v,\qq)}{u_2}\right) = S\bracket*{u_3}{(\vo{L}{-1}v,\qq)}{u_2}\d{w},
  \]
  which is precisely the \emph{$\vo{L}{-1}$-property} of $S$.
\end{proof}

\subsection{An alternative definition of \texorpdfstring{$3$}{3}-pointed \texorpdfstring{$g$}{g}-twisted restricted conformal blocks}
In \cite{PartI}, we introduced the notion of space of \emph{$g$-twisted restricted conformal blocks associated to the twisted projective line $\x\colon\oC\to\PP^1$} in a different way. 

Let $J$ be the subspace of $U^3\otimes M^1\otimes U^2$ spanned by the elements 
 \begin{align}
    &u_3\otimes (\vo{L}{-1}+\vo{L}{0}-h_1+h)v\otimes u_2,\label{f-relation1}\\
    &u_3\cdot [a]\otimes v\otimes u_2-\sum_{j\ge 0}\binom{\wt a}{j} u_3\otimes \vo{a}{j-1}v\otimes u_2,\quad a\in V^0,\label{f-relation2}\\
    &u_3\otimes v\otimes [a]\cdot u_2-\sum_{j\ge 0}\binom{\wt a-1}{j}u_3\otimes \vo{a}{j-1}v\otimes u_2,\quad a\in V^0,\label{f-relation3}\\
    &\sum_{j\ge 0} \binom{\wt a-1+\frac{r}{T}}{j} u_3\otimes \vo{a}{j-1}v\otimes u_2,\quad a\in V^r, r\neq 0,\label{f-relation4}
  \end{align}
  where $u_3\in U^3$, $v\in M^1$, and $u_2\in U^2$. (The space $J$ is found by the calculation of the recursive formulas of the correlation functions in $\Cor[\Sigma_{1}(U^3, M^1, U^2)]$, see \cite[\S 4,5]{PartI}.) 
  We call (see \cite[Definition 4.1]{PartI})
  \begin{equation}\label{eq:firstdefCB}
  \Cfb[U^3, M^1, U^2]:=((U^3\otimes M^1\otimes U^2)/J)^\ast
  \end{equation}
  the \emph{space of $3$-pointed $g$-twisted restricted conformal blocks associated to the triple $(U^3, M^1, U^2)$} on the twisted projective line $\x\colon\oC\to\PP^1$.

 Our next theorem shows that the elements \labelcref{f-relation1,f-relation2,f-relation3,f-relation4} can also be obtained from the action of the constrained Lie algebra $\L^{\circ}_{\PP^1\setminus\Set{1}}(\cV^G)_{\le 0}$, and our previous definition \labelcref{eq:firstdefCB} in \cite{PartI} is a special case of \cref{def:Cfb_res}. 
 
\begin{theorem}\label{thm:Cfb=Cfb}
Consider the datum $\Sigma_{1}(U^3, M^1, U^2)$ given in \labelcref{eq:def:datum_res}. Then the subpaces $\Cfb[\Sigma_{1}(U^3, M^1, U^2)]$ and $\Cfb[U^3, M^1, U^2]$ coincide in the space of linear functionals $(U^3\otimes M^1\otimes U^2)^{\ast}$.
\end{theorem}
\begin{proof}
  First, by \cref{eg:LgV-0}, $\loo{\upomega}{1}{0}$, $\loo{a}{\wt a}{1}$, $\loo{a}{\wt a-1}{1}$ ($a\in V_g^0$), and $\loo{a}[r]{\wt a-1}{1}$ ($a\in V_g^r$ with $r\neq 0$) are contained in $\L^{\circ}_{\PP^1\setminus\Set{1}}(\cV^G)_{\le 0}$. 
  Then  for any $v\in M^1$, $u_2\in U^2$, and $u_3\in U^3$, let $\tau=u_3\otimes v\otimes u_2$, we have
  \begin{align*}
    \tfrac{1}{T}(\loo{\upomega}{1}{0}).{\tau} &=
    {u_3\otimes (\vo{L}{-1}+\vo{L}{0}-h_1+h)v\otimes u_2},\\
    \tfrac{1}{T}(\loo{a}{\wt a}{1}).{\tau} 
    &= {\sum_{j\ge 0}\binom{\wt a}{j} u_3\otimes \vo{a}{j-1}v\otimes u_2 - \theta(\lo{a}{\wt a-1})u_3\otimes v\otimes u_2},\\
    \tfrac{1}{T}(\loo{a}{\wt a-1}{1}).{\tau} 
    &= {\sum_{j\ge 0}\binom{\wt a-1}{j} u_3\otimes \vo{a}{j-1}v\otimes u_2 - u_3\otimes v\otimes \vo{a}{\wt a-1}u_2},\\
    \tfrac{1}{T}(\loo{a}[r]{\wt a-1}{1}).{\tau} 
    &= {\sum_{j\ge 0}\binom{\wt a-1+\frac{r}{T}}{j} u_3\otimes \vo{a}{j-1}v\otimes u_2}.
  \end{align*}
  Comparing the right-hand sides with the elements \labelcref{f-relation1,f-relation2,f-relation3,f-relation4}, we see that $\Cfb[\Sigma_{1}(U^3, M^1, U^2)]\subset\Cfb[U^3, M^1, U^2]$. 

Another direction follows from the following lemma.
\begin{lemma}\label{lem:strOfLgV-0}
  The Lie algebra $\L^{\circ}_{\PP^1\setminus\Set{1}}(\cV^G)_{\le 0}$ is spanned by the following elements\footnote{Note that $\loo{\upomega}{1}{0}=\loo{\upomega}{\wt\upomega}{1}-\loo{\upomega}{\wt\upomega-1}{1}$.}:
  $\loo{a}{\wt a}{1}$, $\loo{a}{\wt a-1}{1}$ ($a\in V_g^0$), and $\loo{a}[r]{\wt a-1}{1}$ ($a\in V_g^r$ with $r\neq 0$).
\end{lemma}
\begin{proof}
  By \cref{eg:LgV-0}, $\L^{\circ}_{\PP^1\setminus\Set{1}}(\cV^G)_{\le 0}$ is spanned by elements $\loo{a}[r]{m}{n}$, where $a\in V_g^r$ is homogeneous and $m\in\Z$, $n\in\N$ satisfy $\frac{r}{T}+m-n \le \wt{a}-1 \le \frac{r}{T}+m$. 

  Let $\g$ denote the subspace of $\L^{\circ}_{\PP^1\setminus\Set{1}}(\cV^G)_{\le 0}$ spanned by $\loo{a}{\wt a}{1}$, $\loo{a}{\wt a-1}{1}$ ($a\in V_g^0$), and $\loo{a}[r]{\wt a-1}{1}$ ($a\in V_g^r$ with $r\neq 0$). 
  
  We first show that $\loo{a}{\wt{a}-1+\red{m}}{\red{n}}\in\g$ for all homogeneous $a\in V_g^0$ by a double induction on $m\in\N$ and $n-m\in\N$. 
  Since $\lo{a}{\wt{a}-1} = \loo{a}{\wt a}{1}-\loo{a}{\wt a-1}{1}$, the base case is clear. 
  Suppose $\loo{a}{\wt{a}-1+m}{n}\in\g$ for all $a\in V_g^0$. 
  Then  the explicit description of $\nabla$ tells us that 
  \[
    \loo{a}{\wt{a}+m}{n+1} = 
    \tfrac{1}{n}\loo{(\vo{L}{-1}a)}{\wt{(\vo{L}{-1}a)}-1+m}{n} + 
    \tfrac{\wt{a}+m}{n} \loo{a}{\wt{a}-1+m}{n}\in \g.
  \]
  Therefore, we have 
  \[
    \loo{a}{\wt{a}-1+m}{n+1} = \loo{a}{\wt{a}+m}{n+1} - \loo{a}{\wt{a}-1+m}{n} \in\g.
  \]
  We summarize the above process of induction by the following diagram of the pairs $(m,n)$, where the red ones correspond to our starting cases.
  
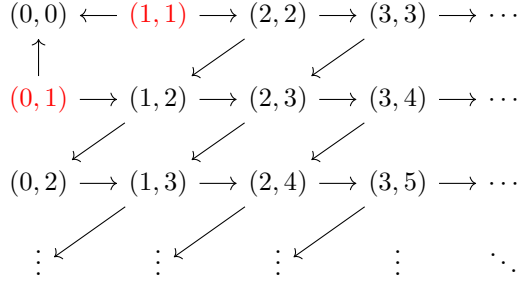
\[
    \begin{tikzcd}[sep=0.2in]
      {(0,0)} & \mathcolor{red}{(1,1)}\ar[l] \ar[r] & {(2,2)} \ar[r]\ar[dl] & {(3,3)} \ar[r]\ar[dl] & \cdots \\ 
      \mathcolor{red}{(0,1)}\ar[u] \ar[r] & {(1,2)} \ar[r]\ar[dl] & {(2,3)} \ar[r]\ar[dl] & {(3,4)} \ar[r]\ar[dl] & \cdots \\
      {(0,2)} \ar[r] & {(1,3)} \ar[r]\ar[dl] & {(2,4)} \ar[r]\ar[dl] & {(3,5)} \ar[r]\ar[dl] & \cdots \\[-0.1in]
      \vdots & \vdots & \vdots & \vdots & \ddots
    \end{tikzcd}
  \]

  Next, we show that $\loo{a}[r]{\wt{a}-1+\red{m}}{1+\red{n}}\in\g$ for all homogeneous $a\in V_g^r$ ($r\neq0$) and $m,n-m\in\N$ by a similar double induction on $m\in\N$ and $n-m\in\N$. 
  We summarize the process of induction by the following diagram of the pairs $(m,n)$, where the red one corresponds to our starting cases.
  
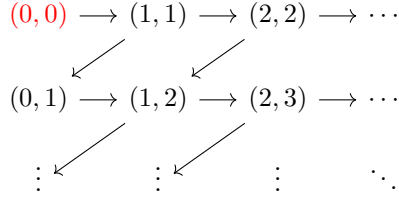
\[
    \begin{tikzcd}[sep=0.2in]
      \mathcolor{red}{(0,0)} \ar[r] & {(1,1)}\ar[r]\ar[dl]  & {(2,2)} \ar[r]\ar[dl] & \cdots \\ 
      {(0,1)} \ar[r] & {(1,2)} \ar[r]\ar[dl] & {(2,3)} \ar[r]\ar[dl] & \cdots \\[-0.1in]
      \vdots & \vdots & \vdots & \ddots
    \end{tikzcd}
  \]
 This completes the proof of \cref{lem:strOfLgV-0}.
\end{proof}
Using a similar argument, we see that the Lie algebra ideal  $\L^{\circ}_{\PP^1\setminus\Set{1}}(\cV^G)_{<0}$ is spanned by the elements
$\loo{a}{\wt a}{2}$ ($a\in V_g^0$) and $\loo{a}[r]{\wt a-1}{1}$ ($a\in V_g^r$ with $r\neq 0$). 
\begin{corollary}
 $\L^{\circ}_{\PP^1\setminus\Set{1}}(\cV^G)_{0}=\L^{\circ}_{\PP^1\setminus\Set{1}}(\cV^G)_{\le0}/\L^{\circ}_{\PP^1\setminus\Set{1}}(\cV^G)_{<0}$ is spanned by the classes $\loo{a}{\wt a}{1}$ and $\loo{a}{\wt a-1}{1}$, for all homogeneous $a\in V_g^0$.
\end{corollary}

Back to our proof of \cref{thm:Cfb=Cfb}. 
Suppose $\varphi\in\Cfb[U^3, M^1, U^2]$. 
Then  $\varphi$ is invariant under $\loo{a}{\wt a}{1}$, $\loo{a}{\wt a-1}{1}$ ($a\in V_g^0$), $\loo{a}[r]{\wt a-1}{1}$ ($a\in V_g^r$ with $r\neq 0$), and hence the Lie subalgebra generated by them. By \cref{lem:strOfLgV-0}, this algebra is precisely $\L^{\circ}_{\PP^1\setminus\Set{1}}(\cV^G)_{\le 0}$. Hence, $\varphi\in\Cfb[\Sigma_{1}(U^3, M^1, U^2)]$.
\end{proof}
\begin{remark}
  In the above argument, the point $1$ can be replaced by any point $w$ of $\PP^1$ other than $0$ and $\infty$.  
\end{remark}

\subsection{The restriction map of twisted conformal blocks}\label{sec:restriction}
Let $\Sigma=(\x\colon\tilX\to\X,\pp_{\bullet},\qq_{\diamond},M^{\bullet},N^{\diamond})$ be the datum given in \cref{def:Cfb} and let $U^{\diamond}$ be the bottom level of $N^{\diamond}$. Then  $\Sigma(0)=(\x\colon\tilX\to\X,\pp_{\bullet},\qq_{\diamond},M^{\bullet},U^{\diamond})$ forms a datum qualifying \cref{def:Cfb_res}. 
Then  the inclusion $\cU^{\diamond}\monomorphism\cM^{\diamond}$ induces a linear map \[\uppi\colon\Cfb[\Sigma]\to\Cfb[\Sigma(0)].\]

\begin{theorem}\label{thm:pi-inj}
  If $N^{\diamond}$ are lowest-weight modules, then the linear map $\uppi$ is injective.
\end{theorem}
\begin{proof}
  Let $\varphi\in\Cfb[\Sigma]$ with $\uppi(\varphi)=0\in\Cfb[\Sigma(0)]$. 
  That is to say $\braket*{\varphi}{\mu_{\bullet}\otimes\upsilon_{\diamond}}=0$ for all $\mu_{\bullet}\in\cM_{\pp_{\bullet}}^{\bullet}$ and $\upsilon_{\diamond}\in\cU_{\qq_{\diamond}}^{\diamond}$. For all homogeneous $a\in V_{g_i}^r$ and $m\in\Z$, by \cref{lem:RRcoro++}, there is an $\alpha\in\L_{\Xc\setminus\pp_{\bullet}}(\cV^G)$ such that 
  \[
    \iota_{\qq_i}(\alpha)\equiv\lo{a}{\frac{r}{T}+m}\mod\L_{\qq_i}(\cV^G)_{<0}\txand[\quad]
    \iota_{\qq_j}(\alpha)\in\L_{\qq_j}(\cV^G)_{<0}\txforall[\ ] j\neq i.
  \]
  Since $\varphi\in\Cfb[\Sigma]$, we must have $\braket*{\varphi}{\alpha.(\mu_{\bullet}\otimes\upsilon_{\diamond})}=0$. 
  By the above, we have
  \[\braket*{\varphi}{\mu_{\bullet}\otimes\upsilon_{\diamond\setminus i}\otimes\vo{a}{\frac{r}{T}+m}\upsilon_i}=0.\]
  Therefore, the subspace \[\Set*{\nu\in\cN_{\qq_i}^i\given\braket*{\varphi}{\mu_{\bullet}\otimes\upsilon_{\diamond\setminus i}\otimes\nu}=0}\] is indeed a $\cV$-submodule of $\cN_{\qq_i}^i$ containing $\cU_{\qq_i}^i$. Since $N^i$ is a lowest-weight module, this submodule must be the entire $N^i$. 
  By the arbitrariness of $i\in\diamond$, we conclude that $\varphi$ vanishes.
\end{proof}

\begin{theorem}\label{thm:pi-surj}
  If $N^{\diamond}$ are the {generalized Verma module} of its bottom level $U^{\diamond}$, then the linear map $\uppi\colon\Cfb[\Sigma]\to\Cfb[\Sigma(0)]$ is surjective.
\end{theorem}
\begin{proof}
  Recall that the \emph{generalized Verma module} $\sfM(U)$ of an $A_g(V)$-module $U$ is spanned by elements of the form 
  \[
    \lo{b^1}{\frac{r_1}{T}+m_1}\cdots\lo{b^l}{\frac{r_l}{T}+m_l}. u,
  \]
  where $b^i\in V_g^{r_i}$, $u\in U$, $m_i\in \Z$, $l\in \N$, and 
  \[\deg(\lo{b^1}{\frac{r_1}{T}+m_1})\ge\cdots\ge\deg(\lo{b^l}{\frac{r_l}{T}+m_l})\ge0.\]
  Hence, we can extend the linear functional $\varphi\in\Cfb[\Sigma(0)]$ to one on $\cM_{\pp_{\bullet}}^{\bullet}\otimes\cN_{\qq_{\diamond}}^{\diamond}$ by the following formula
  \begin{equation}\label{eq:def:extension}
    \braket*{\varphi}{\mu_{\bullet}\otimes \lo{b^{\star_\diamond}}{\frac{r_{\star_\diamond}}{T}+m_{\star_\diamond}}.\upsilon_{\diamond}}:=
    (-1)^{\abs{\star_\diamond}}\braket*{\varphi}{\lqo{b^{\star_\diamond}}{\qq_{\diamond}}{\frac{r_{\star_\diamond}}{T}+m_{\star_\diamond}}.\mu_{\bullet}\otimes\upsilon_{\diamond}},
  \end{equation}
  where $\lqo{b^{i}}{\qq_{j}}{\frac{r_i}{T}+m_i}$ is an element of $\L_{\Xc\setminus\pp_{\bullet}}(\cV^G)$ satisfying
  \begin{equation}\label{eq:def:liftingEls}
    \iota_{\qq_k}(\lqo{b^{i}}{\qq_{j}}{\frac{r_i}{T}+m_i})\equiv
    \begin{dcases*}
      \lo{b^{i}}{\frac{r_i}{T}+m_i}\mod\L_{\qq_j}(\cV^G)_{<0} & if $k=j$,\\
      0\mod\L_{\qq_k}(\cV^G)_{<0} & otherwise.
    \end{dcases*}
  \end{equation}
  It is evident that such an element is unique modulo $\L_{\X\setminus\pp_{\bullet}}(\cV^G)_{<0}$. Hence, to show the linear functional defined by \labelcref{eq:def:extension} is well-defined, it suffices to verify
  \[
    \begin{multlined}
      \sum_{i\ge 0} 
      \binom{l}{i}(-1)^i
      \left(
        \lqo{a}{\qq_k}{\frac{r}{T}+m+l-i}\lqo{b}{\qq_k}{\frac{s}{T}+n+i}-
        (-1)^{l}\lqo{b}{\qq_k}{\frac{s}{T}+n+l-i}\lqo{a}{\qq_k}{\frac{r}{T}+m+i}
      \right).\tau\\
      \equiv\sum_{j\ge 0}
      \binom{m+\frac{r}{T}}{j}
      \lqo{(\vo{a}{l+j}b)}{\qq_k}{\frac{r+s}{T}+m+n-j}.\tau
      \mod\L_{\qq_k}(\cV^G)_{<0}.\tau
    \end{multlined}  
  \]
  for all local sections $\tau$ of $\cM_{\pp_{\bullet}}^{\bullet}\otimes\cN_{\qq_{\diamond}}^{\diamond}$.
  This is evident from \labelcref{eq:def:liftingEls}. 

  It remains to show that $\varphi\in\Cfb[\Sigma]$. Indeed, for $i\in\diamond$, homogeneous $a\in V_{g_i}^r$, and $m\in\Z$, we have 
  \begin{align*}
    \MoveEqLeft
    \braket*{\varphi}{\lqo{a}{\qq_i}{\frac{r}{T}+m}.\left(\mu_{\bullet}\otimes \lo{b^{\star_\diamond}}{\frac{r_{\star_\diamond}}{T}+m_{\star_\diamond}}.\upsilon_{\diamond}\right)} \\ &=
    \braket*{\varphi}{\lqo{a}{\qq_i}{\frac{r}{T}+m}.\mu_{\bullet}\otimes\lo{b^{\star_\diamond}}{\frac{r_{\star_\diamond}}{T}+m_{\star_\diamond}}.\upsilon_{\diamond}} \\
    &\quad+ 
    \braket*{\varphi}{\mu_{\bullet}\otimes\lo{a}{\frac{r}{T}+m}\lo{b^{\star_i}}{\frac{r_{\star_i}}{T}+m_{\star_i}}.\upsilon_i\otimes\lo{b^{\star_{\diamond\setminus i}}}{\frac{r_{\star_{\diamond\setminus i}}}{T}+m_{\star_{\diamond\setminus i}}}.\upsilon_{\diamond\setminus i}}
    \\
    &=\left((-1)^{\abs{\star_\diamond}}+(-1)^{\abs{\star_\diamond}+1}\right)\braket*{\varphi}{\lqo{a}{\qq_i}{\frac{r}{T}+m}\lqo{b^{\star}}{\qq_{\diamond}}{\frac{r_\star}{T}+m_\star}.\mu_{\bullet}\otimes\upsilon_{\diamond}} =0.
  \end{align*}
  By \cref{coro:DesLUVCG0} and noticing that 
  \[
    \lqo{a}{\qq_i}{\frac{r}{T}+m}\equiv \sqo{a}{\qq_i}{\deg(\lo{a}{\frac{r}{T}+m})}\mod\L^{\circ}_{\X\setminus\pp_{\bullet}}(\cV^G)_{<\deg(\lo{a}{\frac{r}{T}+m})} 
  \]
  for all $m\in\Z$ such that $\deg(\lo{a}{\frac{r}{T}+m})\ge0$,
  any $\alpha\in\L_{\Xc\setminus\pp_{\bullet}}(\cV^G)$ can be written as the sum of elements of the form $\lqo{a}{\qq_i}{\frac{r}{T}+m}$ modulo $\L^{\circ}_{\X\setminus\pp_{\bullet}}(\cV^G)_{<0}$. 
  Hence, $\varphi$ is $\L_{\Xc\setminus\pp_{\bullet}}(\cV^G)$-invariant.
\end{proof}

When $V$ is \emph{$g$-rational} (see, e.g. \cite{DLM1}), all irreducible $g$-twisted $V$-modules are generalized Verma module. 
Then  we have: 
\begin{corollary}
    If $V$ is $g$-rational for all generator $g$ of $G$, then $\uppi$ is an isomorphism for any family of irreducible twisted $V$-modules $N^{\diamond}$.
\end{corollary}

\section{The space of coinvariants}\label{sec:coinvariants}
Let $\Sigma(0)=(\x\colon\tilX\to\X,\pp_{\bullet},\qq_{\diamond},M^{\bullet},U^{\diamond})$ be a datum qualifying \cref{def:Cfb_res}. 
This section delves to describe the space of \emph{coinvariants} 
$(\cM^{\bullet}_{\pp_{\bullet}}\otimes\cU^{\diamond}_{\qq_{\diamond}})_{\L^{\circ}_{\X\setminus\pp_{\bullet}}(\cV^G)_{\le0}}$ (see \labelcref{eq:defconinv_res}) in terms of modules over a geometric generalization of the \emph{twisted Zhu's algebra}.

\subsection{Twisted Zhu's algebra over an affine open}
\subsubsection*{Twisted universal enveloping algebra}
Let $U$ be an affine open subset of $\X$. 
By \cref{coro:DesLUVCG0}, $\L^{\circ}_U(\cV^G)$ is a split-filtered Lie algebra and $\L^{\circ}_U(\cV^G)_{\star}$ is its associated graded algebra. 
Then the universal enveloping algebras $\sfU=\sfU(\L^{\circ}_U(\cV^G))$ and $\sfU_{\star}=\sfU(\L^{\circ}_U(\cV^G)_{\star})$ give rise to a split-filtered associative algebra and its associated graded algebra. 
To mimic the construction in \cref{sec:assoalg}, we introduce the following notions.
\begin{definition}
  A \textbf{good pair} consists of a \emph{left} split-filtered associative algebra (not necessarily unital) together with its associated graded algebra. 
  A \textbf{good seminorm} on a good pair $(\sfU,\sfU_{\star})$ is \emph{almost canonical} seminrom on $\sfU$ in the sense of \cite[Definition A.6.8]{DGK23}. Note that this seminorm induces an almost canonical seminorm on $\sfU_{\star}$.
\end{definition}
The good pair $(\sfU,\sfU_{\star})$ admits a good seminorm whose fundamental system of neighborhoods is given by $\sfN^n\sfU:=\sfU\cdot\sfU_{\le-\frac{n}{T}}$. 
Then, by \cite[Lemma A.6.12]{DGK23}, the induced seminorm on the filtered completion $(\widehat{\sfU},\widehat{\sfU}_{\star})$ of the good pair $(\sfU,\sfU_{\star})$ is also \emph{good}. 
Now, consider the following \textbf{twisted Jacobi relations}
\begin{equation}\label{eq:JacobiRelonU}
  \begin{multlined}
    \sum_{i\ge 0} 
    \binom{l}{i}(-1)^i
    \left(
      \sqo{a}{\qq}{m-l-i}\sqo{b}{\qq}{n+i} - 
      (-1)^{l}\sqo{b}{\qq}{n-l-i}\sqo{a}{\qq}{m+i}
    \right)\\
    =\sum_{j\ge 0}
    \binom{m+\frac{r}{T}}{j}
    \sqo{(\vo{a}{l+j}b)}{\qq}{m+n-l},
  \end{multlined}
\end{equation}
where $\qq$ is a branch point inside $U$, $a,b\in V$ are homogeneous, $m,n\in\frac{1}{T}\Z$ such that $\lo{a}{\wt{a}-m-1},\lo{b}{\wt{b}-n-1}\in\L_{g_{\qq}}(V)$, and $l\in\Z$. 
Let $\sfJ_{\star}$ be the \emph{homogeneous} ideal of $\widehat{\sfU}_{\star}$ generated by the twisted Jacobi relations \labelcref{eq:JacobiRelonU} and $\sfJ$ the ideal of $\sfU$ generated by $\sfJ_{\star}$. 
Let $\bar{\sfJ}$ (resp. $\bar{\sfJ}_{\star}$) be the closure of $\sfJ$ (resp. $\sfJ_{\star}$) with respect to the seminorm. 
Then the pair $(\bar{\sfJ},\bar{\sfJ}_{\star})$ forms a good pair of ideals by \cite[Lemmas A.3.3 and A.5.10]{DGK23}. 
Finally, $\U_U(\cV^G):=\widehat{\sfU}/\bar{\sfJ}$ and $\U_U(\cV^G)_{\star}:=\widehat{\sfU}_{\star}/\bar{\sfJ}_{\star}$ form a good pair of associative algebras with complete good
seminorms by \cite[Lemma A.1.14 and A.6.11]{DGK23}.

\begin{definition}
  The split-filtered almost canonically seminormed associative algebra $\U_U(\cV^G)$ is called the \textbf{twisted universal enveloping algebra of $V$ on the affine open $U$}.
\end{definition}
\begin{remark}
  We may consider an affine open containing exactly one branch point $\qq$. Then the above construction provides a good pair $(\U_{\qq}(\cV^G),\U_{\qq}(\cV^G)_{\star})$, which recovers the good pair $(\U_{g_\qq}(V)^{\sfL},\U_{g_\qq}(V))$ introduced in \cref{def:UgV} via a formal expansion $\iota_{z}$.
\end{remark}

\subsubsection*{Twisted Zhu's algebra}
By definition, $\U_U(\cV^G)_{<0}$ is a two-sided ideal of $\U_U(\cV^G)_{\le 0}$ with 
\[
  \U_U(\cV^G)_{\le 0}/\U_U(\cV^G)_{<0} \cong \U_U(\cV^G)_{0}.
\]
Then the first neighborhood $\sfN^1\U_U(\cV^G)_{0}$ is a two-sided ideal of $\U_U(\cV^G)_{0}$.

Similar to \cite{FZ,He17,Han20,DGK23}, we introduce the following notion: 
\begin{definition}\label{def:Zhuoveraffine}
  The \textbf{twisted Zhu's algebra} of $V$ \textbf{over $U$} (resp. \textbf{at $\qq$}) is the (discrete) associative algebra 
  \begin{equation*}
    \begin{aligned}
      \cA_U(\cV^G)&:=\U_U(\cV^G)_0/\sfN^1\U_U(\cV^G)_0\\ 
      \text{(resp. }\cA_{\qq}(\cV^G)&:=\U_{\qq}(\cV^G)_0/\sfN^1\U_{\qq}(\cV^G)_0
      \text{).}
    \end{aligned}
  \end{equation*}
\end{definition}

By \cref{thm:DesLUVCG0}, the graded Lie algebra $\L^{\circ}_U(\cV^G)_{\star}$ admits the following \emph{gradation-presering} decomposition:
\begin{equation}\label{eq:LV0-decom}
  \L^{\circ}_U(\cV^G)_{\star} \cong \prod_{\qq\in U\setminus\Xc} \L_{\qq}(\cV^G)_{\star} \cong \prod_{\qq\in U\setminus\Xc} \L_{g_\qq}(V)_{\star}.
\end{equation}
Consequently, the degree-zero subalgebra $\U_U(\cV^G)_0$ and its first neighborhood $\sfN^1\U_U(\cV^G)_{0}$ admit the following continous decompositions:
\begin{align*}
  \U_U(\cV^G)_0 &\cong \bigotimes_{\qq\in U\setminus\Xc} \U_{\qq}(\cV^G)_0 \cong \bigotimes_{\qq\in U\setminus\Xc} \U_{g_\qq}(V)_0, \\
  \sfN^1\U_U(\cV^G)_0&\cong \sum_{\qq\in U\setminus\Xc} \bigg(\bigotimes_{\pp\in U\setminus\Xc,\pp\neq \qq} \U_{\pp}(\cV^G)_0 \otimes \sfN^1\U_{\qq}(\cV^G)_0\bigg) \\
  &\cong \sum_{\qq\in U\setminus\Xc} \bigg(\bigotimes_{\pp\in U\setminus\Xc,\pp\neq \qq}  \U_{g_{\pp}}(V)_0 \otimes \sfN^1_{\sfL} \U_{g_{\qq}}(V)_0\bigg). 
\end{align*}

These decompositions induce the following decomposition of the twisted Zhu's algebra over $U$: 
\begin{equation}\label{eq:AV0-decom}
  \cA_U(\cV^G) \cong \bigotimes_{\qq\in U\setminus\Xc} \cA_{\qq}(\cV^G) \cong \bigotimes_{\qq\in U\setminus\Xc} \cA_{g_\qq}(V).
\end{equation}
Moreover, we have a natural morphism $\L^{\circ}_U(\cV^G)_0\to\cA_U(\cV^G)_{\Lie}$ of Lie algebras.

\subsubsection*{Actions on $V$-modules}
Let $\pp_\bullet$ be a finite family of points on $\Xc$ and $M^\bullet$ be a family of admissible \emph{untwisted} modules attached to $\pp_\bullet$. Similar to \labelcref{eq:defconinv_res}, we define the space of coinvariants $(\cM^{\bullet}_{\pp_{\bullet}})_{\L^{\circ}_{\X\setminus\pp_\bullet}(\cV^G)_{<0}}$ with respect to the Lie algebra $\L^{\circ}_{\X\setminus\pp_\bullet}(\cV^G)_{<0}$ by 
\begin{equation}\label{eq:def<0coinv}
  (\cM^{\bullet}_{\pp_{\bullet}})_{\L^{\circ}_{\X\setminus\pp_\bullet}(\cV^G)_{<0}}:= (\cM^{\bullet}_{\pp_{\bullet}})/ \L^{\circ}_{\X\setminus\pp_\bullet}(\cV^G)_{<0}.(\cM^{\bullet}_{\pp_{\bullet}}).
\end{equation}
Clearly, $(\cM^{\bullet}_{\pp_{\bullet}})_{\L^{\circ}_{\X\setminus\pp_\bullet}(\cV^G)_{<0}}$ is a module over $\L^{\circ}_{\X\setminus\pp_\bullet}(\cV^G)_0$.

\begin{lemma}\label{lm:actiononM<0}
  The space $(\cM^{\bullet}_{\pp_{\bullet}})_{\L^{\circ}_{\X\setminus\pp_\bullet}(\cV^G)_{<0}}$ is also a module over the twisted Zhu's algebra $\cA_{\X\setminus\pp_\bullet}(\cV^G)$ and the $\L^{\circ}_{\X\setminus\pp_\bullet}(\cV^G)_0$-module structure factors through the natural morphism $\L^{\circ}_{\X\setminus\pp_\bullet}(\cV^G)_0\to\cA_{\X\setminus\pp_\bullet}(\cV^G)_{\Lie}$. 
\end{lemma}
\begin{proof}
  Since an admissible \emph{untwisted} $V$-module $M$ is a smooth $\U(V)$-module with respect to its canonical seminorm, the tensor product $\cM^{\bullet}_{\pp_{\bullet}}$ is a smooth $\U_{\X\setminus\pp_\bullet}(\cV^G)$-module. 
  Since $\sfU(\L^{\circ}_{\X\setminus\pp_\bullet}(\cV^G)_0)$ is \emph{dense} in $\U_{\X\setminus\pp_\bullet}(\cV^G)_0$ (see \cite{He17,Han20}), the space of coinvariants $(\cM^{\bullet}_{\pp_{\bullet}})_{\L^{\circ}_{\X\setminus\pp_\bullet}(\cV^G)_{<0}}$ is a smooth $\U_{\X\setminus\pp_\bullet}(\cV^G)_0$-module.
  Then the first neighborhood $\sfN^1\U_{\X\setminus\pp_\bullet}(\cV^G)_0:=(\U_{\X\setminus\pp_\bullet}(\cV^G)\cdot\U_{\X\setminus\pp_\bullet}(\cV^G)_{<0})_{0}$ clearly acts trivially on $(\cM^{\bullet}_{\pp_{\bullet}})_{\L^{\circ}_{\X\setminus\pp_\bullet}(\cV^G)_{<0}}$. 
  The statements then follow.
\end{proof}

\begin{definition}\label{def:A(M)}
We adopt the following notation to denote the space of coinvariants with respect to $\L^{\circ}_{\X\setminus\pp_\bullet}(\cV^G)_{<0}$ in \labelcref{eq:def<0coinv}: 
\begin{equation*}
  \cA_{\X\setminus\pp_\bullet}(\cM^{\bullet}_{\pp_{\bullet}}):=(\cM^{\bullet}_{\pp_{\bullet}})_{\L^{\circ}_{\X\setminus\pp_\bullet}(\cV^G)_{<0}}, 
\end{equation*}
 and call it the \textbf{strongly restricted $\cA_{\X\setminus\pp_\bullet}(\cV^G)$-module} of $\cM^{\bullet}_{\pp_{\bullet}}$.  
\end{definition}

\begin{remark}
  The algebra $\U_{X\setminus\pp_\bullet}(\cV^G)$ admits an \emph{antipode} given by the natural anti-involution $\alpha\mapsto-\alpha$ of the Lie algebra $\L^{\circ}_{X\setminus\pp_\bullet}(\cV^G)$. 
  Adopting this antipode, any \emph{left} $\cA_{\X\setminus\pp_\bullet}(\cV^G)$-module is also a \emph{right} $\cA_{\X\setminus\pp_\bullet}(\cV^G)$-module. 
  
  In what follows, \emph{we view the strongly restricted module $\cA_{\X\setminus\pp_\bullet}(\cM^{\bullet}_{\pp_{\bullet}})$ as a {right} $\cA_{\X\setminus\pp_\bullet}(\cV^G)$-module in this way}.
\end{remark}

\begin{example}\label{eg:Ag(M)}
  Consider the \emph{twisted projective line} $\x\colon\oC\to\PP^1$. 
  Let $\pp\in \PP^1$ be a point distinct from $0$ and $\infty$. 
  We have 
  \[\cA_{\PP^1\setminus\pp}(\cV^G)\cong
  \cA_{\infty}(\cV^G)\otimes\cA_{0}(\cV^G)\cong A_{g^{-1}}(V)\otimes A_g(V).\] 
  Let $M$ be an admissible {untwisted} $V$-module. 
  Through the above identification, $\cA_{\PP^1\setminus\pp}(\cM_{\pp})$ is an $A_{g^{-1}}(V)\otimes A_g(V)$-module. 
  Using the the anti-isomorphism $\theta$ \labelcref{eq:def:theta} which translates a right $A_{g^{-1}}(V)$-action to a left $A_{g}(V)$-action, we can view $\cA_{\PP^1\setminus\pp}(\cM_{\pp})$ as an $A_g(V)$-bimodule with the \emph{left} and \emph{right} actions
  \begin{equation}\label{eq:Ag(M)}
      [a]\ast[\mu]:= -(\theta[a]\otimes\vac).[\mu]\txand{}
      [\mu]\ast[a]:= -(\vac\otimes[a]).[\mu].
  \end{equation}
  
  Take $\pp=1$, by \cref{lem:strOfLgV-0}, 
  the constrained Lie algebra $\L^{\circ}_{\PP^1\setminus\Set{1}}(\cV^G)_{<0}$ is spanned by the elements $\loo{a}{\wt a}{2}$ ($a\in V_g^0$) and $\loo{a}[r]{\wt a-1}{1}$ ($a\in V_g^r$ with $r\neq 0$). 
  For any $u\in M$, by \cref{eg:chiralactiononP1}, we have
  \begin{align*}
    \loo{a}{\wt a}{2}.u&=\Res_{z=1} Y_M(a,z-1)u\frac{z^{\wt a}}{(z-1)^2}=a\circ_g u,\quad a\in V_g^0,\\
    \loo{a}[r]{\wt a-1}{1}.u&=\Res_{z=1} Y_M(a,z-1)u \frac{z^{\frac{r}{T}+\wt a-1}}{z-1}=a\circ_g u,\quad a\in V_g^r,\ r\neq 0.
  \end{align*}
  They are precisely the elements spanning $O_g(M)$ in \cite{PartI,JJ}. Therefore, we have $\L^{\circ}_{\PP^1\setminus\Set{1}}(\cV^G)_{<0}.M=O_g(M)$, and $\cA_{\PP^1\setminus\Set{1}}(\cM_{1})\cong M/O_g(M)$. 
  Spelling out the actions \labelcref{eq:Ag(M)} on $M/O_g(M)$, for any homogenous $a\in V^r$, we have\footnote{Note that $-a\otimes\frac{z^{\wt a}}{z-1}\d{z}$ and $-a\otimes\frac{z^{\wt a-1}}{z-1}\d{z}$ are two representatives of the classes $\vartheta(a)_{\infty}, a_{0}\in\cA_{\PP^1\setminus\Set{1}}(\cV^G)$ resepectively.}
  \begin{align*}
    [a]\ast[u] &= 
    \begin{dcases*}
      \Res_{z=1}Y_M(a,z-1)u\frac{z^{\wt a}}{z-1}\d{z} & if $r=0$,\\
      0 & otherwise.
    \end{dcases*}\\
    [u]\ast[a] &= 
    \begin{dcases*}
      \Res_{z=1}Y_M(a,z-1)u\frac{z^{\wt a-1}}{z-1}\d{z} & if $r=0$,\\
      0 & otherwise.
    \end{dcases*}
  \end{align*}
  This gives precisely the $A_g(V)$-bimodule $A_g(M)$ introduced in \cite{JJ}.
\end{example}

\subsection{The space of coinvariants characterized by twisted Zhu's algebra}

Let $\Sigma(0)=(\x\colon\tilX\to\X,\pp_{\bullet},\qq_{\diamond},M^{\bullet},U^{\diamond})$ be a datum qualifying \cref{def:Cfb_res}. 
In particular, each $\cU_{\qq_i}$ is assumed to be an $\cA_{\qq_i}(\cV^G)$-module. 
Then, by \labelcref{eq:AV0-decom}, $\cU^{\diamond}_{\qq_{\diamond}}=\bigotimes_{i=1}^n\cU_{\qq_i}$ is a module over the twisted Zhu's algebra $\cA_{\X\setminus\pp_\bullet}(\cV^G)$ over the affine open $\X\setminus\pp_\bullet$ in an obvious way.

\begin{theorem}\label{thm:coinvariants}
  We have the following isomorphism of vector spaces: 
  \[
  (\cM^{\bullet}_{\pp_{\bullet}}\otimes\cU^{\diamond}_{\qq_{\diamond}})_{\L^{\circ}_{\X\setminus\pp_{\bullet}}(\cV^G)_{\le0}} 
  \cong 
  \cA_{\X\setminus\pp_\bullet}(\cM^{\bullet}_{\pp_{\bullet}})\otimes_{\cA_{\X\setminus\pp_\bullet}(\cV^G)}\cU^{\diamond}_{\qq_{\diamond}}.
  \]
\end{theorem}
\begin{proof}
Note that $\L^{\circ}_{\X\setminus\pp_{\bullet}}(\cV^G)_{<0}$ annihilates $\cU^{\diamond}_{\qq_{\diamond}}$, and by \cref{lm:actiononM<0},
\begin{align*}
\L^{\circ}_{\X\setminus\pp_{\bullet}}(\cV^G)_{0}. (\cM^{\bullet}_{\pp_{\bullet}})_{\L^{\circ}_{\X\setminus\pp_{\bullet}}(\cV^G)_{<0}}&=\cA_{\X\setminus\pp_{\bullet}}(\cV^G).(\cM^{\bullet}_{\pp_{\bullet}})_{\L^{\circ}_{\X\setminus\pp_{\bullet}}(\cV^G)_{<0}},\\
\L^{\circ}_{\X\setminus\pp_{\bullet}}(\cV^G)_{0}.\cU^{\diamond}_{\qq_{\diamond}}&= \cA_{\X\setminus\pp_{\bullet}}(\cV^G).\cU^{\diamond}_{\qq_{\diamond}}, 
\end{align*}
Since we also have a split short exact sequence of Lie algebras 
 \[
 0\ra \L^{\circ}_{\X\setminus\pp_{\bullet}}(\cV^G)_{<0} \ra \L^{\circ}_{\X\setminus\pp_{\bullet}}(\cV^G)_{\le 0}\ra \L^{\circ}_{\X\setminus\pp_{\bullet}}(\cV^G)_{0}\ra 0.
 \]
It follows that 
    \begin{align*}
      (\cM^{\bullet}_{\pp_{\bullet}}\otimes\cU^{\diamond}_{\qq_{\diamond}})_{\L^{\circ}_{\X\setminus\pp_{\bullet}}(\cV^G)_{\le0}} &= 
      \left((\cM^{\bullet}_{\pp_{\bullet}}\otimes\cU^{\diamond}_{\qq_{\diamond}})_{\L^{\circ}_{\X\setminus\pp_{\bullet}}(\cV^G)_{<0}}\right)_{\L^{\circ}_{\X\setminus\pp_{\bullet}}(\cV^G)_{0}} \\
      &=\left((\cM^{\bullet}_{\pp_{\bullet}})_{\L^{\circ}_{\X\setminus\pp_{\bullet}}(\cV^G)_{<0}}\otimes\cU^{\diamond}_{\qq_{\diamond}}\right)_{\L^{\circ}_{\X\setminus\pp_{\bullet}}(\cV^G)_{0}} \\
      &=\cA_{\X\setminus\pp_\bullet}(\cM^{\bullet}_{\pp_{\bullet}})\otimes_{\cA_{\X\setminus\pp_{\bullet}}(\cV^G)}\cU^{\diamond}_{\qq_{\diamond}}.
    \end{align*}  
    The statement thus follows by definition.
\end{proof}

\begin{corollary}\label{cor:FusionRule}
  Let $\Sigma=(\x\colon\tilX\to\X,\pp_{\bullet},\qq_{\diamond},M^{\bullet},N^{\diamond})$ be a datum qualifying \cref{def:Cfb} and let $\Sigma(0)=(\x\colon\tilX\to\X,\pp_{\bullet},\qq_{\diamond},M^{\bullet},U^{\diamond})$ be the datum qualifying \cref{def:Cfb_res} with $U^{\diamond}$ the bottom levels of $N^{\diamond}$. 
  If $N^{\diamond}$ are lowest-weight modules that are the generalized Verma modules of their bottom levels $U^{\diamond}$, then we have the following isomorphisms of vector spaces
  \[
  \Cfb[\Sigma]\cong\Cfb[\Sigma(0)]
  \cong 
  \left(\cA_{\X\setminus\pp_\bullet}(\cM^{\bullet}_{\pp_{\bullet}})\otimes_{\cA_{\X\setminus\pp_\bullet}(\cV^G)}\cU^{\diamond}_{\qq_{\diamond}}\right)^{\ast}.
  \]
\end{corollary}

\begin{example}\label{eg:two-twisted}
Consider the \emph{twisted projective line} $\x\colon\oC\to\PP^1$ associated to the datum $\Sigma_1(U^3, M^1, U^2)=(\x\colon\oC\to\PP^1, \infty, 1, 0, U^3, M^1, U^2)$. By \cref{thm:Cfb=Cfb}, the space of coinvariants (whose dual is the space of conformal blocks) is given by the following (identical) vector spaces: 
\[(\cM^{1}_{1}\otimes\cU^{2}\otimes \cU^3)_{\L^{\circ}_{\PP^1\setminus 1}(\cV^G)_{\le0}} =(U^3\otimes M^1\otimes U^2)/J,\]
where $J$ is given by \labelcref{f-relation1,f-relation2,f-relation3,f-relation4} (see \cite[Definition 4.1]{PartI}). 
  By \cref{thm:coinvariants,eg:Ag(M)}, we have
  \begin{align*}
    (U^3\otimes M^1\otimes U^2)/J &\cong 
    \cA_{\PP^1\setminus\Set{1}}(\cM^1_{1})\otimes_{\cA_{\PP^1\setminus\Set{1}}(\cV^G)}(U^3\otimes U^2) \\
&\cong A_g(M^1)\otimes_{A_{g^{-1}}(V)\otimes A_g(V)} (U^3\otimes U^2)
    \\
    &\cong
    U^3\otimes_{A_g(V)}A_g(M^1)\otimes_{A_g(V)}U^2.
  \end{align*}
  This gives a conceptual and short proof of \cite[Theorem 6.5]{PartI}.
\end{example}

\begin{example}\label{eg:CfbSn+1}
  Consider the datum 
  \[\Sigma^n=(\x\colon\X\to\PP^1, \qq_\diamond, N^\diamond),\]
  where $\x\colon\X\to\PP^1$ is a totally ramified covering with Galois group of order $T$, $\qq_\diamond$ are all the branch points with Deck generator $g_\diamond$, and each $N^i$ is an admissible $g_i$-twisted $V$-module. 
  Then the \nameref{lem:propagation} allows us to identify the space $\Cfb[\Sigma^n]$ of twisted conformal blocks associated to this datum with the one associated to the following datum 
  \[
    \Sigma^{n+1}=(\x\colon\X\to\PP^1,\pp, \qq_\diamond, N^\diamond),
  \]
  where $\pp$ is any unramified point. 
  Assume furthermore that each $N^i$ is a lowest-weight module and is the generalized Verma module of its bottom level $U^i$. Then  by \cref{thm:pi-inj,thm:pi-surj}, the space $\Cfb[\Sigma^{n+1}]$ is isomorphic to the space $\Cfb[\Sigma^{n+1}(0)]$ of twisted restricted conformal blocks associated to the datum 
  \[
    \Sigma^{n+1}(0)=(\x\colon\X\to\PP^1,\pp, \qq_\diamond, U^\diamond).
  \]
  By \cref{thm:coinvariants}, we have 
  \begin{equation*}
    \Cfb[\Sigma^{n+1}(0)]\cong
    \left(
      \cA_{\PP^1\setminus\pp}(\cV_\pp)\otimes_{\cA_{\PP^1\setminus\pp}(\cV^G)}(U^\diamond)
    \right)^{\ast}.
  \end{equation*}
  This yields a characterization of the space $\Cfb[\Sigma^n]$ in terms of the bottom levels:
  \begin{equation}
    \Cfb[\Sigma^n] \cong 
    \left(
      \cA_{\PP^1\setminus\pp}(\cV_\pp)\otimes_{\cA_{\PP^1\setminus\pp}(\cV^G)}(U^\diamond)
    \right)^{\ast}.
  \end{equation}
\end{example}

\subsection{Finiteness of twisted conformal blocks}
In this subsection, we assume $V$ to be \textbf{of CFT-type}, i.e., $V_0=\C \vac$. Recall that we only consider non-negatively graded VOAs in this paper.

Let $\Sigma=(\x\colon\tilX\to\X,\pp_{\bullet},\qq_{\diamond},M^{\bullet},N^{\diamond})$ be a datum qualifying \cref{def:Cfb} and let $\Sigma(0)=(\x\colon\tilX\to\X,\pp_{\bullet},\qq_{\diamond},M^{\bullet},U^{\diamond})$ be the datum qualifying \cref{def:Cfb_res} with $U^{\diamond}$ the bottom levels of $N^{\diamond}$. 
This subsection aims to prove the following finiteness theorem. 
\begin{theorem}\label{thm:finiteness}
  Suppose $V$ is $C_2$-cofinite and of CFT-type, each $M^i$ is finitely generated, and each $N^i$ is a lowest-weight module that is the generalized Verma module of its bottom level $U^i$. 
  Then the space of coinvariants 
  \[
  (\cM^{\bullet}_{\pp_{\bullet}}\otimes\cU^{\diamond}_{\qq_{\diamond}})_{\L^{\circ}_{\X\setminus\pp_{\bullet}}(\cV^G)_{\le0}}
  \]
  is finite-dimensional.
\end{theorem}
\begin{proof}
By \cref{thm:coinvariants}, the space of coinvariants is isomorphic to 
\[\cA_{\X\setminus\pp_\bullet}(\cM^{\bullet}_{\pp_{\bullet}})\otimes_{\cA_{\X\setminus\pp_\bullet}(\cV^G)}\cU^{\diamond}_{\qq_{\diamond}}.\]
Then the proof of \cref{thm:finiteness} boils down to two parts: 
\begin{enumerate}
  \item $\cU^{\diamond}_{\qq_{\diamond}}$ is finite-dimensional; and
  \item $\cA_{\X\setminus\pp_\bullet}(\cM^{\bullet}_{\pp_{\bullet}})$ is finite-dimensional.
\end{enumerate}

For the first part, recall that the bottom level of a lowest-weight module is irreducible over $\L_g(V)_0$, and hence over $A_g(V)$.
Since $V$ is $C_2$-cofinite, each twisted Zhu's algebra $A_g(V)$ is finite-dimensional by \cite[Proposition 6.6]{PartI}.
Hence each $U^i$, as a irreducible module over $A_{g_i}(V)$, is finite-dimensional. 

Now we prove the second claim.
Fix a lift $\tilde\pp_i$ for each $\pp_i$, the space $\cM^{\bullet}_{\pp_\bullet}$ is identified with $\cM^{\bullet}_{\tilde\pp_\bullet}$.
Since $\pp_\bullet$ are unramified points, we have 
$\L_{\pp_i}(\cV^G)\cong\L_{\tilde\pp_i}(\cV)$ for each $\pp_i$. 
Hence the actions of the Lie algebras
\[
  \L^{\circ}_{\X\setminus\pp_\bullet}(\cV^G)_{\le m} \txand{}
  \L^{\circ}_{\tilX\setminus\tilde\pp_\bullet}(\cV)_{\le m}
\]
on $\cM^{\bullet}_{\pp_\bullet}$ ($\cong\cM^{\bullet}_{\tilde\pp_\bullet}$) coincide, where 
\[
  \L^{\circ}_{\tilX\setminus\tilde\pp_\bullet}(\cV)_{\le m}:=
  \H^0\left(\tilX\setminus\tilde\pp_\bullet,\left(\cV\otimes\Omega^1(-(\vo{L}{0}-m-1)\dQ)\right)/\Image\nabla\right).
\]
Then $\cA_{\X\setminus\pp_\bullet}(\cM^{\bullet}_{\pp_{\bullet}})$ is identified with $(\cM^{\bullet}_{\tilde\pp_{\bullet}})_{\L^{\circ}_{{\tilX}\setminus\tilde\pp_\bullet}(\cV)_{<0}}$. 
The finiteness of the latter space is proved in \cite[Proposition 5.1.1]{DGT19}.
\end{proof}
\begin{corollary}
  Suppose $V$ is $C_2$-cofinite and of CFT-type. 
  For $i=1,2,3$, let $M^i$ be a $g_i$-twisted admissible $V$-module such that $g_1g_2=g_3$ and they generate a cyclic group of order $T$. 
  Suppose $M^2$ and $M^3$ are lowest-weight modules, and $M^2$ and $(M^3)'$ are the generalized Verma modules of their bottom level.
  Assume further one of the following conditions is held:
  \begin{enumerate}
    \item $g_1=\id$ and $M^1$ is finitely generated.  
    \item $g_1$, $g_2$, and $g_3$ have the same order $T$, and $M^1$ is a lowest-weight module that is the generalized Verma module of its bottom level.    
  \end{enumerate}
  Then the fusion rule among twisted modules 
  \[
    \Nusion:=\dim\Fusion
  \]
  is finite.
\end{corollary}

\section*{Acknowledgments}
We wish to thank Professors Angela Gibney and Danny Karshen for their valuable discussions. 

\section*{Declarations}

No funding was received for conducting this research.
The authors have no competing interests to declare that are relevant to the content of this article.


\addcontentsline{toc}{section}{References}

\begin{thebibliography}{51}
  \expandafter\ifx\csname natexlab\endcsname\relax\def\natexlab#1{#1}\fi
  \providecommand{\url}[1]{\texttt{#1}}
  \providecommand{\href}[2]{#2}
  \providecommand{\path}[1]{#1}
  \providecommand{\DOIprefix}{doi:}
  \providecommand{\ArXivprefix}{arXiv:}
  \providecommand{\URLprefix}{URL: }
  \providecommand{\Pubmedprefix}{pmid:}
  \providecommand{\doi}[1]{\href{http://dx.doi.org/#1}{\path{#1}}}
  \providecommand{\Pubmed}[1]{\href{pmid:#1}{\path{#1}}}
  \providecommand{\bibinfo}[2]{#2}
  \ifx\xfnm\relax \def\xfnm[#1]{\unskip,\space#1}\fi
  
  \bibitem[{Gao et~al.(2023)Gao, Liu, and Zhu}]{PartI}
  \bibinfo{author}{X.~Gao}, \bibinfo{author}{J.~Liu}, \bibinfo{author}{Y.~Zhu}, \bibinfo{title}{Twisted restricted conformal blocks of vertex operator algebras {I}: $g$-twisted correlation functions and fusion rules}, \bibinfo{year}{2023}. \href{http://arxiv.org/abs/2312.16278}{{\tt arXiv:2312.16278}}.
  
  \bibitem[{Belavin et~al.(1984)Belavin, Polyakov, and Zamolodchikov}]{BPZ84}
  \bibinfo{author}{A.~A. Belavin}, \bibinfo{author}{A.~M. Polyakov}, \bibinfo{author}{A.~B. Zamolodchikov},
  \newblock \bibinfo{title}{Infinite conformal symmetry in two-dimensional quantum field theory},
  \newblock \bibinfo{journal}{Nuclear Physics, Section B} \bibinfo{volume}{241} (\bibinfo{year}{1984}).
  
  \bibitem[{Tsuchiya and Kanie(1987)}]{TK}
  \bibinfo{author}{A.~Tsuchiya}, \bibinfo{author}{Y.~Kanie},
  \newblock \bibinfo{title}{Vertex operators in the conformal field theory on {${\mathbb{P}}^1$} and monodromy representations of the braid group},
  \newblock \bibinfo{journal}{Lett. Math. Phys.} \bibinfo{volume}{13} (\bibinfo{year}{1987}) \bibinfo{pages}{303--312}. \DOIprefix\doi{10.1007/BF00401159}.
  
  \bibitem[{Tsuchiya et~al.(1989)Tsuchiya, Ueno, and Yamada}]{TUY89}
  \bibinfo{author}{A.~Tsuchiya}, \bibinfo{author}{K.~Ueno}, \bibinfo{author}{Y.~Yamada},
  \newblock \bibinfo{title}{Conformal field theory on universal family of stable curves with gauge symmetries},
  \newblock in: \bibinfo{booktitle}{Integrable systems in quantum field theory and statistical mechanics}, volume~\bibinfo{volume}{19} of \textit{\bibinfo{series}{Adv. Stud. Pure Math.}}, \bibinfo{publisher}{Academic Press}, \bibinfo{address}{Boston, MA}, \bibinfo{year}{1989}, pp. \bibinfo{pages}{459--566}.
  
  \bibitem[{Frenkel and Zhu(1992)}]{FZ}
  \bibinfo{author}{I.~B. Frenkel}, \bibinfo{author}{Y.~Zhu},
  \newblock \bibinfo{title}{Vertex operator algebras associated to representations of affine and virasoro algebras},
  \newblock \bibinfo{journal}{Duke Math. J.} \bibinfo{volume}{66} (\bibinfo{year}{1992}) \bibinfo{pages}{123--168}. \DOIprefix\doi{10.1215/S0012-7094-92-06604-X}.
  
  \bibitem[{Huang(1997)}]{H97}
  \bibinfo{author}{Y.-Z. Huang}, \bibinfo{title}{Two-dimensional conformal geometry and vertex operator algebras}, volume \bibinfo{volume}{148} of \textit{\bibinfo{series}{Prog. Math.}}, \bibinfo{publisher}{Boston, MA: Birkh{\"a}user}, \bibinfo{year}{1997}.
  
  \bibitem[{Frenkel et~al.(1993)Frenkel, Huang, and Lepowsky}]{FHL}
  \bibinfo{author}{I.~B. Frenkel}, \bibinfo{author}{Y.-Z. Huang}, \bibinfo{author}{J.~Lepowsky},
  \newblock \bibinfo{title}{On axiomatic approaches to vertex operator algebras and modules},
  \newblock \bibinfo{journal}{Mem. Am. Math. Soc.} \bibinfo{volume}{104} (\bibinfo{year}{1993}). \DOIprefix\doi{10.1090/memo/0494}.
  
  \bibitem[{Li(1999)}]{Li}
  \bibinfo{author}{H.~Li},
  \newblock \bibinfo{title}{Determining fusion rules by {$A(V)$}-modules and bimodules},
  \newblock \bibinfo{journal}{J. Algebra} \bibinfo{volume}{212} (\bibinfo{year}{1999}) \bibinfo{pages}{515--556}. \DOIprefix\doi{10.1006/jabr.1998.7655}.
  
  \bibitem[{Li(2001)}]{Li3}
  \bibinfo{author}{H.~Li},
  \newblock \bibinfo{title}{The regular representation, {Zhu}'s {{\(A(V)\)}}-theory, and induced modules},
  \newblock \bibinfo{journal}{J. Algebra} \bibinfo{volume}{238} (\bibinfo{year}{2001}) \bibinfo{pages}{159--193}. \DOIprefix\doi{10.1006/jabr.2000.8627}.
  
  \bibitem[{Zhu(1996)}]{Z}
  \bibinfo{author}{Y.~Zhu},
  \newblock \bibinfo{title}{Modular invariance of characters of vertex operator algebras},
  \newblock \bibinfo{journal}{J. Amer. Math. Soc.} \bibinfo{volume}{9} (\bibinfo{year}{1996}) \bibinfo{pages}{237--302}. \DOIprefix\doi{10.1090/S0894-0347-96-00182-8}.
  
  \bibitem[{Huang(2005{\natexlab{a}})}]{H05a}
  \bibinfo{author}{Y.-Z. Huang},
  \newblock \bibinfo{title}{Differential equations and intertwining operators},
  \newblock \bibinfo{journal}{Commun. Contemp. Math.} \bibinfo{volume}{7} (\bibinfo{year}{2005}{\natexlab{a}}) \bibinfo{pages}{375--400}. \DOIprefix\doi{10.1142/S0219199705001799}.
  
  \bibitem[{Huang(2005{\natexlab{b}})}]{H05b}
  \bibinfo{author}{Y.-Z. Huang},
  \newblock \bibinfo{title}{Differential equations, duality and modular invariance},
  \newblock \bibinfo{journal}{Commun. Contemp. Math.} \bibinfo{volume}{7} (\bibinfo{year}{2005}{\natexlab{b}}) \bibinfo{pages}{649--706}. \DOIprefix\doi{10.1142/S021919970500191X}.
  
  \bibitem[{Liu(2023)}]{Liu}
  \bibinfo{author}{J.~Liu},
  \newblock \bibinfo{title}{A proof of the fusion rules theorem},
  \newblock \bibinfo{journal}{Comm. Math. Phys.} \bibinfo{volume}{401} (\bibinfo{year}{2023}) \bibinfo{pages}{1237--1290}. \DOIprefix\doi{10.1007/s00220-023-04664-2}.
  
  \bibitem[{Zhu(1994)}]{Z94}
  \bibinfo{author}{Y.~Zhu},
  \newblock \bibinfo{title}{Global vertex operators on riemann surfaces},
  \newblock \bibinfo{journal}{Comm. Math. Phys.} \bibinfo{volume}{165} (\bibinfo{year}{1994}) \bibinfo{pages}{485--531}. \DOIprefix\doi{10.1007/BF02099421}.
  
  \bibitem[{Abe and Nagatomo(2003{\natexlab{a}})}]{AN1}
  \bibinfo{author}{T.~Abe}, \bibinfo{author}{K.~Nagatomo},
  \newblock \bibinfo{title}{Finiteness of conformal blocks over the projective line},
  \newblock in: \bibinfo{booktitle}{Vertex operator algebras in mathematics and physics. Proceedings of the workshop, Fields Institute for Research in Mathematical Sciences, Toronto, Canada, October 23--27, 2000}, \bibinfo{publisher}{Providence, RI: American Mathematical Society (AMS)}, \bibinfo{year}{2003}{\natexlab{a}}, pp. \bibinfo{pages}{1--12}.
  
  \bibitem[{Abe and Nagatomo(2003{\natexlab{b}})}]{AN2}
  \bibinfo{author}{T.~Abe}, \bibinfo{author}{K.~Nagatomo},
  \newblock \bibinfo{title}{Finiteness of conformal blocks over compact {Riemann} surfaces},
  \newblock \bibinfo{journal}{Osaka J. Math.} \bibinfo{volume}{40} (\bibinfo{year}{2003}{\natexlab{b}}) \bibinfo{pages}{375--391}.
  
  \bibitem[{Malikov et~al.(1999)Malikov, Schechtman, and Vaintrob}]{MSV}
  \bibinfo{author}{F.~Malikov}, \bibinfo{author}{V.~Schechtman}, \bibinfo{author}{A.~Vaintrob},
  \newblock \bibinfo{title}{Chiral de {Rham} complex},
  \newblock \bibinfo{journal}{Commun. Math. Phys.} \bibinfo{volume}{204} (\bibinfo{year}{1999}) \bibinfo{pages}{439--473}. \DOIprefix\doi{10.1007/s002200050653}.
  
  \bibitem[{Malikov and Schechtman(1999)}]{MS99}
  \bibinfo{author}{F.~Malikov}, \bibinfo{author}{V.~Schechtman},
  \newblock \bibinfo{title}{Chiral de {Rham} complex. {II}},
  \newblock in: \bibinfo{booktitle}{Differential topology, infinite-dimensional Lie algebras, and applications. D. B. Fuchs' 60th anniversary collection}, \bibinfo{publisher}{Providence, RI: American Mathematical Society}, \bibinfo{year}{1999}, pp. \bibinfo{pages}{149--188}.
  
  \bibitem[{Beilinson and Drinfeld(2004)}]{BD04}
  \bibinfo{author}{A.~Beilinson}, \bibinfo{author}{V.~Drinfeld}, \bibinfo{title}{Chiral Algebras}, volume~\bibinfo{volume}{51} of \textit{\bibinfo{series}{MAmer. Math. Soc. Colloq.}}, \bibinfo{publisher}{American Mathematical Society (AMS)}, \bibinfo{address}{Providence, RI}, \bibinfo{year}{2004}.
  
  \bibitem[{Frenkel and Ben-Zvi(2004)}]{FBZ04}
  \bibinfo{author}{E.~Frenkel}, \bibinfo{author}{D.~Ben-Zvi}, \bibinfo{title}{Vertex algebras and algebraic curves}, volume~\bibinfo{volume}{88} of \textit{\bibinfo{series}{Math. Surv. Monogr.}}, \bibinfo{edition}{2nd revised and expanded} ed., \bibinfo{publisher}{American Mathematical Society (AMS)}, \bibinfo{address}{Providence, RI}, \bibinfo{year}{2004}.
  
  \bibitem[{Nagatomo and Tsuchiya(2005)}]{NT}
  \bibinfo{author}{K.~Nagatomo}, \bibinfo{author}{A.~Tsuchiya},
  \newblock \bibinfo{title}{Conformal field theories associated to regular chiral vertex operator algebras. {I}: {Theories} over the projective line},
  \newblock \bibinfo{journal}{Duke Math. J.} \bibinfo{volume}{128} (\bibinfo{year}{2005}) \bibinfo{pages}{393--471}. \DOIprefix\doi{10.1215/S0012-7094-04-12831-3}.
  
  \bibitem[{Damiolini et~al.(pear)Damiolini, Gibney, and Tarasca}]{DGT19}
  \bibinfo{author}{C.~Damiolini}, \bibinfo{author}{A.~Gibney}, \bibinfo{author}{N.~Tarasca},
  \newblock \bibinfo{title}{On factorization and vector bundles of conformal blocks from vertex algebras},
  \newblock \bibinfo{journal}{Ann. Sci. {\'E}c. Norm. Sup{\'e}r. (4)}  (\bibinfo{year}{to appear}). \href{http://arxiv.org/abs/1909.04683}{{\tt arXiv:1909.04683}}.
  
  \bibitem[{Damiolini et~al.(2023)Damiolini, Gibney, and Krashen}]{DGK23}
  \bibinfo{author}{C.~Damiolini}, \bibinfo{author}{A.~Gibney}, \bibinfo{author}{D.~Krashen}, \bibinfo{title}{Conformal blocks on smoothings via mode transition algebras}, \bibinfo{year}{2023}. \href{http://arxiv.org/abs/2307.03767}{{\tt arXiv:2307.03767}}.
  
  \bibitem[{Gui(2024)}]{G}
  \bibinfo{author}{B.~Gui},
  \newblock \bibinfo{title}{Convergence of sewing {Conformal} blocks},
  \newblock \bibinfo{journal}{Commun. Contemp. Math.} \bibinfo{volume}{26} (\bibinfo{year}{2024}) \bibinfo{pages}{65}. \DOIprefix\doi{10.1142/S0219199723500074}, \bibinfo{note}{id/No 2350007}.
  
  \bibitem[{Lepowsky(1985)}]{Lepowsky}
  \bibinfo{author}{J.~Lepowsky},
  \newblock \bibinfo{title}{Calculus of twisted vertex operators},
  \newblock \bibinfo{journal}{Proc. Natl. Acad. Sci. USA} \bibinfo{volume}{82} (\bibinfo{year}{1985}) \bibinfo{pages}{8295--8299}. \DOIprefix\doi{10.1073/pnas.82.24.8295}.
  
  \bibitem[{Frenkel et~al.(1988)Frenkel, Lepowsky, and Meurman}]{FLM}
  \bibinfo{author}{I.~B. Frenkel}, \bibinfo{author}{J.~Lepowsky}, \bibinfo{author}{A.~Meurman}, \bibinfo{title}{Vertex operator algebras and the monster}, volume \bibinfo{volume}{134} of \textit{\bibinfo{series}{Pure and Applied Mathematics}}, \bibinfo{publisher}{Academic Press, Inc.}, \bibinfo{address}{Boston etc.}, \bibinfo{year}{1988}.
  
  \bibitem[{Dixon et~al.(1985)Dixon, Harvey, Vafa, and Witten}]{DHVW86}
  \bibinfo{author}{L.~Dixon}, \bibinfo{author}{J.~Harvey}, \bibinfo{author}{C.~Vafa}, \bibinfo{author}{E.~Witten},
  \newblock \bibinfo{title}{Strings on orbifolds},
  \newblock \bibinfo{journal}{Nuclear Physics B} \bibinfo{volume}{261} (\bibinfo{year}{1985}) \bibinfo{pages}{678--686}. \DOIprefix\doi{10.1016/0550-3213(85)90593-0}.
  
  \bibitem[{Dijkgraaf et~al.(1989)Dijkgraaf, Vafa, Verlinde, and Verlinde}]{DVVV89}
  \bibinfo{author}{R.~Dijkgraaf}, \bibinfo{author}{C.~Vafa}, \bibinfo{author}{E.~Verlinde}, \bibinfo{author}{H.~Verlinde},
  \newblock \bibinfo{title}{The operator algebra of orbifold models},
  \newblock \bibinfo{journal}{Commun. Math. Phys.} \bibinfo{volume}{123} (\bibinfo{year}{1989}) \bibinfo{pages}{485--526}. \DOIprefix\doi{10.1007/BF01238812}.
  
  \bibitem[{Dong(1994)}]{D94}
  \bibinfo{author}{C.~Dong},
  \newblock \bibinfo{title}{Twisted modules for vertex algebras associated with even lattices},
  \newblock \bibinfo{journal}{J. Algebra} \bibinfo{volume}{165} (\bibinfo{year}{1994}) \bibinfo{pages}{91--112}. \DOIprefix\doi{10.1006/jabr.1994.1099}.
  
  \bibitem[{Dong et~al.(1998)Dong, Li, and Mason}]{DLM1}
  \bibinfo{author}{C.~Dong}, \bibinfo{author}{H.~Li}, \bibinfo{author}{G.~Mason},
  \newblock \bibinfo{title}{Twisted representations of vertex operator algebras},
  \newblock \bibinfo{journal}{Math. Ann.} \bibinfo{volume}{310} (\bibinfo{year}{1998}) \bibinfo{pages}{571--600}. \DOIprefix\doi{10.1007/s002080050161}.
  
  \bibitem[{Dong et~al.(2000)Dong, Li, and Mason}]{DLM00}
  \bibinfo{author}{C.~Dong}, \bibinfo{author}{H.~Li}, \bibinfo{author}{G.~Mason},
  \newblock \bibinfo{title}{Modular-invariance of trace functions in orbifold theory and generalized {Moonshine}.},
  \newblock \bibinfo{journal}{Commun. Math. Phys.} \bibinfo{volume}{214} (\bibinfo{year}{2000}) \bibinfo{pages}{1--56}. \DOIprefix\doi{10.1007/s002200000242}.
  
  \bibitem[{Barron et~al.(2002)Barron, Dong, and Mason}]{BDM02}
  \bibinfo{author}{K.~Barron}, \bibinfo{author}{C.~Dong}, \bibinfo{author}{G.~Mason},
  \newblock \bibinfo{title}{Twisted sectors for tensor product vertex operator algebras associated to permutation groups},
  \newblock \bibinfo{journal}{Commun. Math. Phys.} \bibinfo{volume}{227} (\bibinfo{year}{2002}) \bibinfo{pages}{349--384}. \DOIprefix\doi{10.1007/s002200200633}.
  
  \bibitem[{Dong and Xu(2006)}]{DX06}
  \bibinfo{author}{C.~Dong}, \bibinfo{author}{F.~Xu},
  \newblock \bibinfo{title}{Conformal nets associated with lattices and their orbifolds},
  \newblock \bibinfo{journal}{Adv. Math.} \bibinfo{volume}{206} (\bibinfo{year}{2006}) \bibinfo{pages}{279--306}. \DOIprefix\doi{10.1016/j.aim.2005.08.009}.
  
  \bibitem[{Huang(2010)}]{H10}
  \bibinfo{author}{Y.-Z. Huang},
  \newblock \bibinfo{title}{Generalized twisted modules associated to general automorphisms of a vertex operator algebra},
  \newblock \bibinfo{journal}{Commun. Math. Phys.} \bibinfo{volume}{298} (\bibinfo{year}{2010}) \bibinfo{pages}{265--292}. \DOIprefix\doi{10.1007/s00220-010-0999-6}.
  
  \bibitem[{Dong et~al.(2017)Dong, Ren, and Xu}]{DRX17}
  \bibinfo{author}{C.~Dong}, \bibinfo{author}{L.~Ren}, \bibinfo{author}{F.~Xu},
  \newblock \bibinfo{title}{On orbifold theory},
  \newblock \bibinfo{journal}{Adv. Math.} \bibinfo{volume}{321} (\bibinfo{year}{2017}) \bibinfo{pages}{1--30}. \DOIprefix\doi{10.1016/j.aim.2017.09.032}.
  
  \bibitem[{Frenkel and Szczesny(2004)}]{FS04}
  \bibinfo{author}{E.~Frenkel}, \bibinfo{author}{M.~Szczesny},
  \newblock \bibinfo{title}{Twisted modules over vertex algebras on algebraic curves},
  \newblock \bibinfo{journal}{Adv. Math.} \bibinfo{volume}{187} (\bibinfo{year}{2004}) \bibinfo{pages}{195--227}. \DOIprefix\doi{10.1016/j.aim.2003.07.019}.
  
  \bibitem[{Jiang and Jiao(2016)}]{JJ}
  \bibinfo{author}{Q.~Jiang}, \bibinfo{author}{X.~Jiao},
  \newblock \bibinfo{title}{Bimodule and twisted representation of vertex operator algebras},
  \newblock \bibinfo{journal}{Sci. China, Math.} \bibinfo{volume}{59} (\bibinfo{year}{2016}) \bibinfo{pages}{397--410}. \DOIprefix\doi{10.1007/s11425-015-5033-1}.
  
  \bibitem[{Ueno(2008)}]{Ueno}
  \bibinfo{author}{K.~Ueno}, \bibinfo{title}{Conformal field theory with gauge symmetry.}, volume~\bibinfo{volume}{24} of \textit{\bibinfo{series}{Fields Inst. Monogr.}}, \bibinfo{publisher}{Providence, RI: American Mathematical Society (AMS); Toronto: The Fields Institute for Research in Mathematical Sciences}, \bibinfo{year}{2008}.
  
  \bibitem[{Damiolini and Gibney(2023)}]{DG23}
  \bibinfo{author}{C.~Damiolini}, \bibinfo{author}{A.~Gibney},
  \newblock \bibinfo{title}{On global generation of vector bundles on the moduli space of curves from representations of vertex operator algebras},
  \newblock \bibinfo{journal}{Algebr. Geom.} \bibinfo{volume}{10} (\bibinfo{year}{2023}) \bibinfo{pages}{298–326}.
  
  \bibitem[{Liu(2002)}]{liu2002algebraic}
  \bibinfo{author}{Q.~Liu}, \bibinfo{title}{Algebraic geometry and arithmetic curves}, volume~\bibinfo{volume}{6} of \textit{\bibinfo{series}{Oxford Graduate Texts in Mathematics}}, \bibinfo{publisher}{Oxford University Press}, \bibinfo{address}{Oxford}, \bibinfo{year}{2002}. \bibinfo{note}{Translated from the French by Reinie Ern\'{e}}.
  
  \bibitem[{Hartshorne(1977)}]{Hartshorne}
  \bibinfo{author}{R.~Hartshorne}, \bibinfo{title}{Algebraic geometry}, volume~\bibinfo{volume}{52} of \textit{\bibinfo{series}{Graduate Texts in Mathematics}}, \bibinfo{publisher}{Springer-Verlag}, \bibinfo{address}{New York-Heidelberg}, \bibinfo{year}{1977}.
  
  \bibitem[{Fulton(2008)}]{fulton}
  \bibinfo{author}{W.~Fulton}, \bibinfo{title}{Algebraic Curves: An Introduction to Algebraic Geometry}, \bibinfo{edition}{3rd} ed., \bibinfo{publisher}{avalible at author's webpage}, \bibinfo{address}{\url{https://dept.math.lsa.umich.edu/~wfulton/CurveBook.pdf}}, \bibinfo{year}{2008}.
  
  \bibitem[{Schlag(2014)}]{Schlag}
  \bibinfo{author}{W.~Schlag}, \bibinfo{title}{A course in complex analysis and {Riemann} surfaces}, volume \bibinfo{volume}{154} of \textit{\bibinfo{series}{Graduate Studies in Mathematics}}, \bibinfo{publisher}{American Mathematical Society (AMS)}, \bibinfo{address}{Providence, RI}, \bibinfo{year}{2014}.
  
  \bibitem[{Tate(1968)}]{Tate}
  \bibinfo{author}{J.~Tate},
  \newblock \bibinfo{title}{Residues of differentials on curves},
  \newblock \bibinfo{journal}{Ann. Sci. {\'E}c. Norm. Sup{\'e}r. (4)} \bibinfo{volume}{1} (\bibinfo{year}{1968}) \bibinfo{pages}{149--159}. \DOIprefix\doi{10.24033/asens.1162}.
  
  \bibitem[{Borcherds(1986)}]{B}
  \bibinfo{author}{R.~E. Borcherds},
  \newblock \bibinfo{title}{Vertex algebras, {K}ac-{M}oody algebras, and the {M}onster},
  \newblock \bibinfo{journal}{Proc. Natl. Acad. Sci. USA} \bibinfo{volume}{83} (\bibinfo{year}{1986}) \bibinfo{pages}{3068--3071}. \DOIprefix\doi{10.1073/pnas.83.10.3068}.
  
  \bibitem[{Lepowsky and Li(2004)}]{LL}
  \bibinfo{author}{J.~Lepowsky}, \bibinfo{author}{H.~Li}, \bibinfo{title}{Introduction to vertex operator algebras and their representations}, volume \bibinfo{volume}{227} of \textit{\bibinfo{series}{Prog. Math.}}, \bibinfo{publisher}{Birkh{\"a}user}, \bibinfo{address}{Boston, MA}, \bibinfo{year}{2004}.
  
  \bibitem[{Xu(1995)}]{X95}
  \bibinfo{author}{X.~Xu},
  \newblock \bibinfo{title}{Intertwining operators for twisted modules of a colored vertex operator superalgebra},
  \newblock \bibinfo{journal}{J. Algebra} \bibinfo{volume}{175} (\bibinfo{year}{1995}) \bibinfo{pages}{241--273}. \DOIprefix\doi{10.1006/jabr.1995.1185}.
  
  \bibitem[{Huang(2018)}]{H18}
  \bibinfo{author}{Y.-Z. Huang},
  \newblock \bibinfo{title}{Intertwining operators among twisted modules associated to not-necessarily-commuting automorphisms},
  \newblock \bibinfo{journal}{J. Algebra} \bibinfo{volume}{493} (\bibinfo{year}{2018}) \bibinfo{pages}{346--380}. \DOIprefix\doi{10.1016/j.jalgebra.2017.09.029}.
  
  \bibitem[{He(2017)}]{He17}
  \bibinfo{author}{X.~He},
  \newblock \bibinfo{title}{Higher level {Zhu} algebras are subquotients of universal enveloping algebras},
  \newblock \bibinfo{journal}{J. Algebra} \bibinfo{volume}{491} (\bibinfo{year}{2017}) \bibinfo{pages}{265--279}. \DOIprefix\doi{10.1016/j.jalgebra.2017.08.014}.
  
  \bibitem[{Huang and Yang(2012)}]{HY12}
  \bibinfo{author}{Y.-Z. Huang}, \bibinfo{author}{J.~Yang},
  \newblock \bibinfo{title}{Logarithmic intertwining operators and associative algebras},
  \newblock \bibinfo{journal}{J. Pure Appl. Algebra} \bibinfo{volume}{216} (\bibinfo{year}{2012}) \bibinfo{pages}{1467--1492}. \DOIprefix\doi{10.1016/j.jpaa.2011.12.006}.
  
  \bibitem[{Han and Xiao(2020)}]{Han20}
  \bibinfo{author}{J.~Han}, \bibinfo{author}{Y.~Xiao},
  \newblock \bibinfo{title}{Associative algebras and universal enveloping algebras associated to {VOAs}},
  \newblock \bibinfo{journal}{J. Algebra} \bibinfo{volume}{564} (\bibinfo{year}{2020}) \bibinfo{pages}{489--498}. \DOIprefix\doi{10.1016/j.jalgebra.2020.09.002}.
  
\end{thebibliography}
\end{document}